\documentclass[12pt]{amsart}
\usepackage{amssymb}
\usepackage{enumitem}
\usepackage{graphicx}
\usepackage[cmtip,all]{xy}
\usepackage{bbm}
\usepackage{comment}
\usepackage{mathrsfs}
\usepackage{mathtools}
\usepackage{xcolor}
\usepackage[dvipdfmx]{hyperref}
\usepackage{autobreak}
\hypersetup{
 colorlinks,
 linkcolor={blue!50!black},
 citecolor={blue!50!black},
 urlcolor={blue!80!black}
}

\everymath{\displaystyle}

\calclayout

\def\DefineSymbol#1#2{\newcommand{#1}{{\mathrm {#2}}}}
\def\DefineCategory#1#2{\newcommand{#1}{{\mathrm {#2}}}}

\theoremstyle{plain}
	\newtheorem{theorem}{Theorem}[section]
	\newtheorem{lemma}[theorem]{Lemma}
	\newtheorem{proposition}[theorem]{Proposition}
	\newtheorem{corollary}[theorem]{Corollary}
	\newtheorem{conjecture}[theorem]{Conjecture}
	
\theoremstyle{definition}
	\newtheorem{definition}[theorem]{Definition}
	\newtheorem{lemma and definition}[theorem]{Lemma and Definition}
	\newtheorem{example}[theorem]{Example}

	\newtheorem{construction}[theorem]{Construction}

\theoremstyle{remark}
	\newtheorem{remark}[theorem]{Remark}
	\numberwithin{equation}{section}

\DefineSymbol{\pr}{pr}
\DefineSymbol{\perf}{perf}
\DefineSymbol{\id}{id}
\DefineSymbol{\const}{const}
\DefineSymbol{\lc}{lc}
\DefineSymbol{\op}{op}
\DefineSymbol{\diag}{diag}
\DefineSymbol{\proet}{pro\acute{e}t}
\DefineSymbol{\cond}{cond}
\DefineSymbol{\cont}{cont}
\DefineSymbol{\conti}{conti}
\DefineSymbol{\Cond}{Cond}
\DefineSymbol{\dCond}{dCond}
\DefineSymbol{\disc}{disc}
\DefineSymbol{\Tot}{Tot}
\DefineSymbol{\triv}{triv}
\DefineSymbol{\Bun}{Bun}
\DefineSymbol{\Pic}{Pic}
\DefineSymbol{\LS}{LS}
\DefineSymbol{\Quot}{Quot}
\DefineSymbol{\WT}{WT}
\DefineSymbol{\FF}{FF}
\DefineSymbol{\adic}{adic}
\DefineSymbol{\rig}{rig}
\DefineSymbol{\nc}{nc}
\DefineSymbol{\nuc}{nuc}
\DefineSymbol{\cpt}{cpt}
\DefineSymbol{\an}{an}
\DefineSymbol{\alg}{alg}
\DefineSymbol{\la}{la}
\DefineSymbol{\rla}{rla}
\DefineSymbol{\dep}{dep}
\DefineSymbol{\red}{red}
\DefineSymbol{\ex}{ex}
\DefineSymbol{\et}{et}
\DefineSymbol{\fet}{fet}
\DefineSymbol{\wft}{wft}
\DefineSymbol{\cl}{cl}
\DefineSymbol{\Cl}{\mathbf{Cl}}
\DefineSymbol{\rCl}{\mathbf{rCl}}
\DefineSymbol{\Nil}{\mathbf{Nil}}
\DefineSymbol{\rNil}{\mathbf{rNil}}
\DefineSymbol{\Fin}{\mathbf{Fin}}
\DefineSymbol{\irreg}{irreg}

\DeclareMathOperator{\rank}{rank}
\DeclareMathOperator{\Hom}{Hom}
\DeclareMathOperator{\Isom}{Isom}
\DeclareMathOperator{\Surj}{Surj}

\DeclareMathOperator{\Ext}{Ext}
\DeclareMathOperator{\Exan}{Exan}
\DeclareMathOperator{\Exanbb}{\mathbf{Exan}}
\DeclareMathOperator{\Def}{Def}
\DeclareMathOperator{\Defbb}{\mathbf{Def}}
\DeclareMathOperator{\Der}{Der}
\DeclareMathOperator{\Map}{Map}

\DeclareMathOperator{\Ker}{Ker}
\DeclareMathOperator{\Coker}{Coker}
\let\Im\relax
\DeclareMathOperator{\Im}{Im}

\DeclareMathOperator{\Aut}{Aut}
\DeclareMathOperator{\Gal}{Gal}

\DeclareMathOperator{\height}{ht}

\DeclareMathOperator{\Nm}{Nm}
\DeclareMathOperator{\Tr}{Tr}

\DeclareMathOperator{\Spec}{Spec}

\DeclareMathOperator{\Spa}{Spa}

\DeclareMathOperator{\Spf}{Spf}

\DeclareMathOperator{\End}{End}

\DeclareMathOperator{\supp}{supp}
\DeclareMathOperator{\Fit}{Fit}

\DeclareMathOperator{\GL}{GL}
\DeclareMathOperator{\intHom}{\underline{Hom}}

\newcommand{\hotimes}{\mathbin{\hat{\otimes}}}

\newcommand{\triplerightarrow}{%
  \mathrel{\vcenter{\baselineskip=0.6ex \lineskiplimit=0pt
    \hbox{$\rightarrow$}\hbox{$\rightarrow$}\hbox{$\rightarrow$}}}%
}

\DefineCategory{\Set}{Set}
\DefineCategory{\Ab}{Ab}
\DefineCategory{\Ring}{Ring}
\DefineCategory{\Mod}{Mod}
\DefineCategory{\LMod}{LMod}
\DefineCategory{\Rep}{Rep}
\DefineCategory{\coMod}{coMod}
\DefineCategory{\Alg}{Alg}
\DefineCategory{\Ch}{Ch}
\DefineCategory{\Mon}{Mon}
\DefineCategory{\CMon}{CMon}
\DefineCategory{\Grpd}{Grpd}
\DefineCategory{\PCoh}{PCoh}
\DefineCategory{\Perf}{Perf}
\DefineCategory{\Vect}{Vect}
\DefineCategory{\Coh}{Coh}
\DefineCategory{\FP}{FP}
\DefineCategory{\cAff}{cAff}
\DefineCategory{\Aff}{Aff}
\DefineCategory{\Rig}{Rig}
\DefineCategory{\AffRing}{AffRing}
\DefineCategory{\AnRing}{AnRing}
\DefineCategory{\AnAdic}{AnAdic}
\DefineCategory{\AffAnAdic}{AffAnAdic}
\DefineCategory{\Ani}{Ani}
\DefineCategory{\CAlg}{CAlg}
\DefineCategory{\Corr}{Corr}
\DefineCategory{\Comm}{Comm}
\DefineCategory{\PShv}{PShv}

\newcommand{\Abb}{\mathbb{A}}
\renewcommand{\Bbb}{\mathbb{B}}
\newcommand{\Cbb}{\mathbb{C}}

\newcommand{\Gbb}{\mathbb{G}}

\newcommand{\Lbb}{\mathbb{L}}

\newcommand{\Pbb}{\mathbb{P}}
\newcommand{\Qbb}{\mathbb{Q}}

\newcommand{\Vbb}{\mathbb{V}}

\newcommand{\Zbb}{\mathbb{Z}}

\newcommand{\Acal}{\mathcal{A}}

\newcommand{\Ccal}{\mathcal{C}}
\newcommand{\Dcal}{\mathcal{D}}
\newcommand{\Ecal}{\mathcal{E}}
\newcommand{\Fcal}{\mathcal{F}}
\newcommand{\Gcal}{\mathcal{G}}

\newcommand{\Ical}{\mathcal{I}}

\newcommand{\Ocal}{\mathcal{O}}
\newcommand{\Pcal}{\mathcal{P}}

\newcommand{\Rcal}{\mathcal{R}}

\newcommand{\Xcal}{\mathcal{X}}
\newcommand{\Ycal}{\mathcal{Y}}
\newcommand{\Zcal}{\mathcal{Z}}

\newcommand{\Afrak}{\mathfrak{A}}

\newcommand{\Xfrak}{\mathfrak{X}}

\newcommand{\mfrak}{\mathfrak{m}}
\newcommand{\nfrak}{\mathfrak{n}}

\newcommand{\pfrak}{\mathfrak{p}}

\newcommand{\kbf}{\mathbf{k}}

\newcommand{\tbf}{\mathbf{t}}

\renewcommand{\tilde}{\widetilde}
\renewcommand{\hat}{\widehat}
\renewcommand{\bar}{\overline}

\begin{document}

\title{Geometric properties of the analytic Emerton--Gee stack}%[論文の上に表示されるタイトル]{正式なタイトル}
\author{Yutaro Mikami}
\date{\today}
\address{RIKEN, Center for Advanced Intelligence Project AIP, Mathematical Science Team, 1-4-1 Nihonbashi, Chuo-ku, Tokyo, 103-0027, Japan}
\email{yutaro.mikami@riken.jp}
%14G22 リジッド幾何学 13D09 可換環と導来圏 11F80 Galois representations 11S25 Galois cohomology
\begin{abstract}
The analytic Emerton--Gee stack is a rigid analytic moduli stack of $(\varphi,\Gamma)$-modules over the Robba ring, and is expected to play a fundamental role in the categorical $p$-adic local Langlands correspondence for locally analytic representations.
In this paper, we establish the fundamental geometric properties of the analytic Emerton--Gee stack: it is a rigid analytic Artin stack, normal, a local complete intersection, and equidimensional of the expected dimension.
As an application, our normality result removes the reducedness hypothesis in \cite[Theorem 5.1.3]{LSW26}, thereby yielding an unconditional proof of the conjecture on the global sections of the analytic Emerton--Gee stack formulated in \cite[Conjecture 5.1.36]{EGH23}.
\end{abstract}
\maketitle

\setcounter{tocdepth}{1}

\tableofcontents

\section*{Introduction}
\subsection{Background}
Let $K$ be a finite extension of $\Qbb_p$.
A \textit{$(\varphi,\Gamma_K)$-module over the Robba ring} (or equivalently, a \textit{$\Gal(\bar{K}/K)$-equivariant vector bundle} on the Fargues-Fontaine curve $X_{\Cbb_p}$) can be regarded as a generalization of a rational $p$-adic representation of $\Gal(\overline{K}/K)$.
Set $G_K=\Gal(\bar{K}/K)$.
The theory of $p$-adic families of $G_K$-equivariant vector bundles on $X_{\Cbb_p}$ is closely related to the theory of $p$-adic families of automorphic forms and eigenvarieties.
Recently, in \cite[Definition 5.1.15]{EGH23}, Emerton, Gee, and Hellmann introduced a rigid analytic moduli stack $\LS_{\GL_n(K),\FF}$ parametrizing such families of rank $n$, called the \textit{analytic Emerton-Gee stack}.
They also proposed, in \cite{EGH23}, a \textit{categorical $p$-adic local
Langlands correspondence} using this stack.

\begin{conjecture}[{\cite[Conjecture 6.2.4, Remark 6.2.9 (d)]{EGH23}}]
Let $n$ be a positive integer.
Let $\Rep_{\square}^{\la}(\GL_n(K))$ denote the stable $\infty$-category of locally analytic representations of $\GL_n(K)$ defined in \cite[Definition 3.2.3]{RJRC23}, and let $\Dcal(\LS_{\GL_n(K),\FF})$ denote the stable $\infty$-category of solid quasi-coherent complexes on $\LS_{\GL_n(K),\FF}$.
There exists an exact functor
\[\Afrak_{\GL_n(K)}^{\rig} \colon \Rep_{\square}^{\la}(\GL_n(K)) \to \Dcal(\LS_{\GL_n(K),\FF})\]
satisfying many desirable properties.
\end{conjecture}
\begin{remark}
More precisely, it seems necessary to restrict to the subcategories consisting of objects satisfying suitable finiteness conditions, but we will not pursue this in this paper.
\end{remark}

Understanding the geometric properties of $\LS_{\GL_n(K),\FF}$ is an important step towards this conjecture.
In fact, Emerton, Gee, and Hellmann proposed the following conjecture.
Set $d=[K:\Qbb_p]$.

\begin{conjecture}[{\cite[Conjecture 5.1.19]{EGH23}}]\label{conjecture:EGH,intro}
The analytic Emerton--Gee stack $\LS_{\GL_n(K),\FF}$ is a rigid analytic Artin stack.
Moreover, it is normal, a local complete intersection, and equidimensional of dimension $dn^2$.
\end{conjecture}

The main theorem of this paper is the following.

\begin{theorem}\label{thm:EGH conjecture,intro}
	Conjecture \ref{conjecture:EGH,intro} is true.
\end{theorem}

Before explaining the idea of the proof, we recall some previous results related to this conjecture.

For the (formal) Emerton--Gee stack, namely, the formal moduli stack of integral $(\varphi,\Gamma_K)$-modules, analogous geometric properties were established by Emerton and Gee \cite{EG23} (representability and dimension) and by B\"{o}ckle, Iyengar, and Pa\v{s}k\={u}nas \cite{BIP23} (the local complete intersection property and normality of the generic fiber).

In \cite{EGH23}, the results of Hellmann, Hernandez, and Schraen are also discussed. 
They announced that the analytic Emerton--Gee stack $\LS_{\GL_n(K),\FF}$ is a rigid analytic Artin stack on an open neighborhood of each classical point. Note that this does not suffice to prove the full representability of the stack; see \cite[Remark 5.1.32]{EGH23}. 
When $n=2$, they established the desired geometric properties around classical points \cite[Theorem 5.1.20]{EGH23}.
In the case $n=2$, the proof given in \cite{EGH23} is based on the slope filtrations by Kedlaya \cite{Ked08}, the deformation theory of integral $p$-adic local Galois representations, and the results of B\"{o}ckle, Iyengar, and Pa\v{s}k\={u}nas (\cite{BIP23}).

\subsection{Statements of the main results}
Let us now describe the main results of this paper.

\subsubsection{Representability}
Let $K$ be a complete non-archimedean field, and let $\Aff_K$ denote the category of affinoid $K$-algebras (i.e., Banach $K$-algebras topologically of finite type).

\begin{definition}[{Definition \ref{def:quasi-analytic spaces and rigid analytic Artin stacks}}]
	\begin{enumerate}
	\item A \textit{quasi-analytic space} over $K$ is an \'{e}tale sheaf $\Fcal$ on $\Aff_K^{\op}$ such that
	\begin{enumerate}
		\item The diagonal morphism $\Delta \colon \Fcal\to \Fcal\times \Fcal$ is representable by rigid spaces.
		\item There exists an \'{e}tale surjection $U\twoheadrightarrow \Fcal$ from a rigid space $U$.
	\end{enumerate}
	\item A \textit{rigid analytic Artin stack} over $K$ is an \'{e}tale stack $\Xcal$ on $\Aff_K^{\op}$ such that
	\begin{enumerate}
		\item The diagonal morphism $\Delta \colon \Xcal\to \Xcal\times \Xcal$ is representable by quasi-analytic spaces.
		\item There exists a smooth surjection $U\twoheadrightarrow \Xcal$ from a rigid space $U$.
	\end{enumerate}
	\end{enumerate}
\end{definition}

The primary result of this paper are the following theorem.

\begin{theorem}[{Theorem \ref{thm:EG stack representable for general G}}]\label{thm:intro EG stack}
	Let $K$ be a finite extension of $\Qbb_p$.
	Let $G$ be a smooth linear algebraic group over $K$.
	Let $\LS_{G(K),\FF}$ denote the \'{e}tale stack 
    \[\Aff_{K}\ni A\mapsto \{\text{$G_K$-equivariant $G$-torsors on $X_{\Cbb_p}\times_{\Spa \Qbb_p} \Spa A$}\}.\]
	Then $\LS_{G(K),\FF}$ is a rigid analytic Artin stack.
\end{theorem}

We write $X_{\Cbb_p,A}=X_{\Cbb_p}\times_{\Spa \Qbb_p} \Spa A$.
When $G=\GL_n$, $\LS_{\GL_n(K),\FF}$ is the moduli stack of $G_K$-equivariant vector bundles on $X_{\Cbb_p}$ of rank $n$.
Thus $\LS_{\GL_n(K),\FF}$ coincides with the analytic Emerton-Gee stack $\Xfrak_n$ in \cite[Definition 5.1.15]{EGH23} after base change to $K$. 
Therefore, Theorem \ref{thm:intro EG stack} provides an affirmative answer to the representability part of Conjecture \ref{conjecture:EGH,intro}.

Let us briefly explain the strategy of the above theorems.
The representability of $\LS_{G(K),\FF}$ reduces to the case of $\LS_{\GL_n(K),\FF}$; therefore, we explain the proof for $G=\GL_n$.
Let $F$ be a finite extension, and let $x\colon \Spa F\to \LS_{\GL_n(K),\FF}$ be a morphism corresponding to a $G_K$-equivariant vector bundle $V_x$ on $X_{\Cbb_p,F}$.
We need to construct a smooth cover of an open neighborhood of $x$.
The key observation is that we can explicitly compute Rim-Schlessinger's construction of ``formal versal deformation'' of $x$ 
\[f \colon \Spf F[[X_1,\ldots,X_{r_1}]]/(F_1,\ldots,F_{r_2})\to \LS_{\GL_n(K),\FF}\]
by using the Herr complex of $\underline{\End}(V_x)$, where $r_i=\dim_F \Ext_{\FF,G_K}^i(V_x,V_x)$ for $i=1,2$, and $F_1,\ldots, F_{r_2}\in (X_1,\ldots,X_{r_1})^2$.
Here, $\Ext_{\FF,G_K}^i(V_x,V_x)$ is the cohomology of $\underline{\End}(V_x)$, and the morphism $f$ corresponds to a compatible system of $G_K$-equivariant vector bundles $(V_l)_l$ on 
\[X_{\Cbb_p}\times \Spa F[X_1,\ldots,X_{r_1}]/(F_1,\ldots,F_{r_2}, (X_1,\ldots,X_{r_1})^l).\]
Moreover, we can control the norms of $F_i$ and transition functions of $V_l$.
This observation provides an extension of the morphism $f$
\[\tilde{f}\colon \Spa \left(F\left\langle \frac{X_1,\ldots,X_{r_1}}{p^m}\right\rangle/(F_1,\ldots,F_{r_2})\right)\to \LS_{\GL_n(K),\FF}\]
for $m$ sufficiently large.
Moreover, this construction works starting with a morphism $\Spa A\to  \LS_{\GL_n(K),\FF}$ corresponding to a $G_K$-equivariant vector bundle $V_A$ on $X_{\Cbb_p,A}$ such that both $\Ext^1_{\FF,G_K}(V_A,V_A)$ and $\Ext^2_{\FF,G_K}(V_A,V_A)$ are finite free $A$-modules, and thus we obtain a morphism
\[\tilde{f}_A\colon \Spa \left(A\left\langle \frac{X_1,\ldots,X_{r_1}}{p^m}\right\rangle/(F_1,\ldots,F_{r_2})\right)\to \LS_{\GL_n(K),\FF}\]
such that for any classical point $t\in \Spa A$, the restriction of $\tilde{f}_A$ \[\tilde{f}_t\colon \Spa \left(k(t)\left\langle \frac{X_1,\ldots,X_{r_1}}{p^m}\right\rangle/(F_1,\ldots,F_{r_2})\right)\to \LS_{\GL_n(K),\FF}\]
recovers the previous construction.
This resolves the problem that constructing smooth covers around only classical points is insufficient.

Finally, we need to establish a rigid analytic analogue of ``openness of versality'', which is described as follows.

\begin{theorem}[{Theorem \ref{thm:criterion for representablity}}]\label{thm:intro criterion for representability}
	Let $\Xcal$ be an \'{e}tale stack on $\Aff_K^{\op}$.
	Assume the following:
	\begin{enumerate}
		\item The stack $\Xcal$ is $\Fin$-homogeneous.
		\item The diagonal morphism $\Delta \colon \Xcal\to \Xcal\times \Xcal$ is representable by quasi-analytic spaces.
		\item The stack $\Xcal$ has coherent deformations and obstructions.
		\item For any morphism $\Spa R\to \Xcal$ from an affinoid rigid space $\Spa R$ and any (not necessarily classical) point $x\in \Spa R$, there exist the following data:
		\begin{itemize}
			\item A smooth affinoid rigid space $\Spa A$ over $K$ with a point $y\in \Spa A$, and a morphism $\Spa A\to \Spa R$ with $y\mapsto x$ and $k(y)\cong k(x)$.
			Moreover, the natural morphism $A\to k(y)$ is injective.
			\item A morphism $f\colon \Spa B \to \Spa A$ with a section $i\colon \Spa A\to \Spa B$.
			\item A morphism $g\colon \Spa B\to \Xcal$ fitting into a commutative diagram
			\[
			\xymatrix{
				\Spa A \ar[r] \ar[d]_{i} &\Spa R \ar[d]\\
				\Spa B\ar[r]^g & \Xcal.
			}
			\]
			Moreover, for any classical point $t\in (\Spa A)^{\cl}$, the induced morphism
			\[g\colon f^{-1}(t)\to \Xcal,\] where $f^{-1}(t)=\Spa (B\otimes_A k(t))$, is formally versal at $i(t)$.
		\end{itemize}
	\end{enumerate}
	Then $\Xcal$ is a rigid analytic Artin stack.
\end{theorem}

This criterion can be proved in a manner analogous to the arguments in \cite{Hall17Openness,War17}.
In our setting, the explicit analytic extension $\tilde{f}_A$ constructed in the previous paragraph precisely provides the geometric data required in Condition (4). 
By applying this criterion, we obtain the representability of $\LS_{\GL_n(K),\FF}$.

\begin{remark}
	We can regard $\tilde{f}_A\colon \Spa \left(A\left\langle \frac{X_1,\ldots,X_{r_1}}{\varpi^m}\right\rangle/(F_1,\ldots,F_{r_2})\right)\to \LS_{\GL_n(K),\FF}$ as an ``overconvergent versal deformation'' of $\alpha\colon \Spa A\to \LS_{\GL_n(K),\FF}$.
	The notion of overconvergent deformations in the case where $A=K$ has been systematically studied by Li in \cite{RayLi26}.
	Using this framework, he independently constructs an overconvergent versal deformation
	\[\Spa \left(K\left\langle \frac{X_1,\ldots,X_{r_1}}{\varpi^m}\right\rangle/(F_1,\ldots,F_{r_2})\right)\to \LS_{\GL_n(K),\FF}\]
	for the analytic Emerton-Gee stack.
\end{remark}

\subsubsection{Geometric properties of $\LS_{\GL_n(K),\FF}$}

Let us explain how to prove the remaining part of Conjecture \ref{conjecture:EGH,intro}.
Let $F$ be a finite extension of $K$.
Let $x\colon \Spa F\to \LS_{\GL_n(K),\FF}$ be a classical point corresponding to a $G_K$-equivariant vector bundle $V_x$ on $X_{\Cbb_p,F}$.
Set $r_i=\dim_F \Ext^i_{\FF,G_K}(V_x,V_x)$ for $i=0,1,2$.
From the previous construction, we get a smooth morphism
\[g\colon \Spa \left(F\left\langle\frac{X_1,\ldots,X_{r_1}}{p^m}\right\rangle/(F_1,\ldots,F_{r_2})\right)\to \LS_{\GL_n(K),\FF}\]
of pure dimension $r_0$ whose image contains $x$, where 
\[F_1,\ldots,F_{r_2}\in (X_1,\ldots,X_{r_1})^2F\left\langle\frac{X_1,\ldots,X_{r_1}}{p^m}\right\rangle.\]
By the Euler characteristic formula, we have $r_1-r_2-r_0=dn^2$, where $d=[K:\Qbb_p]$. 
Thus, the above theorem shows that $\LS_{\GL_n(K),\FF}$ is smooth of pure dimension $dn^2$ at every classical point $x$ such that $\Ext^2_{\FF,G_K}(V_x,V_x)=0$.
Let $Z$ denote the reduced Zariski closed substack on which this $\Ext^2_{\FF,G_K}$ does not vanish.
Then the proof of Theorem \ref{thm:EGH conjecture,intro} reduces to showing that $Z$ has codimension at least two.
More precisely, we prove the following theorem.

\begin{theorem}\label{thm:bad locus dimension upper bound, intro}
We have $\dim Z\leq dn^2-2$.
\end{theorem}

For integers $0<m<n$, let $P_{m,n-m}$ be the parabolic subgroup of $\GL_{n}$ of the form 
\[
\begin{pmatrix}
    \GL_m & * \\
    0 & \GL_{n-m}
\end{pmatrix}.
\]
Let $\LS_{P_{m,n-m}(K),\FF}$ be a rigid analytic moduli stack parametrizing families of $G_K$-equivariant vector bundles on $X_{\Cbb_p}$ of rank $n$ equipped with an equivariant subbundle of rank $m$.
To simplify the notation, assume $(n,m,K)\neq (2,1,\Qbb_p)$.
Roughly speaking, the proof of Theorem \ref{thm:bad locus dimension upper bound, intro} consists of two parts.
\begin{enumerate}
	\item Prove $\dim Z\leq \max_{0<m<n} \dim\LS_{P_{m,n-m}(K),\FF}$.
	\item Prove $\dim\LS_{P_{m,n-m}(K),\FF}\leq dn^2-2$.\footnote{When $(n,m)=(2,1)$, we have $\dim \LS_{P_{1,1}(K),\FF}\leq 3d$. 
	In particular, if $K=\Qbb_p$, then $\dim \LS_{P_{1,1}(\Qbb_p),\FF}=3>4-2$. However, this problem can be easily overcome by replacing $\LS_{P_{1,1}(\Qbb_p),\FF}$ with a suitable closed substack.}
\end{enumerate}

For (2), let us consider the morphism
\[f\colon \LS_{P_{m,n-m}(K),\FF}\to \LS_{\GL_m(K),\FF}\times \LS_{\GL_{n-m}(K),\FF}.\]
By induction, we have an upper bound on the dimension of the target.
For a classical point $x\colon \Spa F\to \LS_{\GL_{m}(K),\FF}\times \LS_{\GL_{n-m}(K),\FF}$ corresponding to a pair of $G_K$-equivariant vector bundles $V_{m,x}$ and $V_{n-m,x}$ on $X_{\Cbb_p}$, the fiber $f^{-1}(x)$ is equivalent to
\[[\Ext^1_{\FF,G_K}(V_{n-m,x},V_{m,x})/\Hom_{\FF,G_K}(V_{n-m,x},V_{m,x})],\]
where the action of $\Hom_{\FF,G_K}(V_{n-m,x},V_{m,x})$ is trivial.
In particular, by the Euler characteristic formula, we can control the dimension of the fiber $f^{-1}(x)$ by computing $\Ext^2_{\FF,G_K}(V_{n-m,x},V_{m,x})$.

For (1), let us consider a classical point $x\colon \Spa F\to Z$ corresponding to a $G_K$-equivariant vector bundle $V_x$ on $X_{\Cbb_p}$.
By definition, we have $\Ext^2_{\FF,G_K}(V_x,V_x)\neq 0$.
By Tate local duality, this implies that there exists a nonzero morphism $g\colon V_x\to V_x(1)$, where $V_x(1)$ is the Tate twist.
Then the kernel of $g$ gives a subbundle of $V_x$.
This argument implies that the image of the morphism
\[\bigsqcup_{0<m<n}\LS_{P_{m,n-m}(K),\FF}\to \LS_{\GL_n(K),\FF}\]
contains all classical points of $Z$.
However, this is insufficient to obtain (1) since the above morphism is not proper (even after taking a Drinfeld compactification, see \cite[Remark 5.3.15]{EGH23}).
To overcome this problem, we construct a locally Zariski closed covering of $Z$ by imposing suitable conditions.
On each part, we use exterior powers and the Pl\"{u}cker relations to construct the required subbundles in families.

\subsection{Global sections of $\LS_{\GL_n(K),\FF}$}
As an application of the normality, and hence reducedness, of $\LS_{\GL_n(K),\FF}$, we determine global sections on $\LS_{\GL_n(K),\FF}$.
Let $E$ be a finite extension of $K$ such that $\#\Hom_{\Alg_{\Qbb_p}}(K,E)=[K:\Qbb_p]$, and set $\Sigma_K=\Hom_{\Alg_{\Qbb_p}}(K,E)$.
In what follows, we regard $\LS_{\GL_n(K),\FF}$ as a rigid analytic Artin stack over $E$ via base change.
First, we have a morphism
\[\omega_n\colon \LS_{\GL_n(K),\FF}\to \prod_{\sigma\in \Sigma_K}\Abb^{n,\an}_E\]
which sends a $G_K$-equivariant vector bundle $V_A$ on $X_{\Cbb_p,A}$ to the coefficients of its Sen polynomials indexed by $\sigma\in\Sigma_K$.
Then, there is a natural commutative diagram
\[
\xymatrix{
	\LS_{\GL_n(K),\FF}\ar[r]^-{\omega_n}\ar[d]_-{\det} & \prod_{\sigma\in \Sigma_K}\Abb^{n,\an}_E\ar[d]\\
	\LS_{\GL_1(K),\FF}\ar[r]^-{\omega_1} & \prod_{\sigma\in \Sigma_K}\Abb^{1,\an}_E,
}
\]
where the right vertical morphism is the product of the projections onto the first coordinates.

For a rigid analytic Artin stack $\Xcal$, we write $\Gamma(\Xcal)=\Gamma(\Xcal,\Ocal_{\Xcal})$ for the ring of global sections on $\Xcal$.
\begin{conjecture}[{\cite[Conjecture 5.1.36]{EGH23}}]\label{conj:EGH global sections, intro}
	The morphism 
	\[\LS_{\GL_n(K),\FF}\xrightarrow{(\omega_n,\det)} \prod_{\sigma\in \Sigma_K}\Abb^{n,\an}_E\times_{\prod_{\sigma\in \Sigma_K}\Abb^{1,\an}_E}\LS_{\GL_1(K),\FF}\]
	induces an isomorphism on the global sections.
\end{conjecture}

Le, Strauch, and Wang proved this conjecture under the assumption that $\LS_{\GL_n(K),\FF}$ is reduced (\cite[Proposition 5.1.2, Theorem 5.1.3]{LSW26}).
Combining their theorem with the reducedness of $\LS_{\GL_n(K),\FF}$ proved in this paper, we obtain an unconditional proof of the conjecture.

\begin{theorem}
	Conjecture \ref{conj:EGH global sections, intro} holds.
\end{theorem}

\subsection{Outline of the paper}
This paper is organized as follows.

In Section 1, we collect preliminaries on rigid geometry which will be used in the next section. 

In Section 2, we introduce basic notions concerning rigid analytic Artin stacks, including normality, the local complete intersection property, and dimension. We also define Fitting ideals for coherent sheaves on rigid analytic Artin stacks, which play an important role in the proof of the main theorem.

In Section 3, we prove a rigid analytic analogue of Artin's criterion.

In Section 4, we prove the representability of $\LS_{G(K),\FF}$.

In Section 5.1, we state the main theorem and reduce its proof to obtaining an upper bound for the dimension of the ``bad locus'' of the analytic Emerton--Gee stack. We also explain how the relevant dimension estimates can be obtained inductively.
In Section 5.2, we treat the case $n=1$, which forms the first step of the induction.
In Section 5.3, we establish an upper bound for the dimension of the analytic Emerton--Gee stack associated with a parabolic subgroup.
In Section 5.4, using the upper bound established in Section 5.3, we obtain an upper bound for the dimension of the ``bad locus'' of the analytic Emerton--Gee stack.
In Section 5.5, we discuss the global sections on $\LS_{\GL_n(K),\FF}$.

In Appendix A, we prove that a smooth morphism of quasi-analytic spaces is formally smooth.

In Appendix B, we prove that a smooth surjection of locally noetherian analytic adic spaces admits a section \'{e}tale locally.

In Appendix C, we collect basic properties of $G$-torsors with a continuous action of a profinite group.

In Appendix D, we provide an example of an \'{e}tale sheaf $\Fcal$ which is a rigid analytic Artin stack but not a quasi-analytic space.

\subsection{Convention}

\begin{itemize}
\item
All rings are assumed unital and commutative.
\item
We use the terms \textit{$f$-adic ring} and \textit{affinoid pair} rather than \textit{Huber ring} and \textit{Huber pair}.
\item For an $f$-adic ring $A$, we denote the ring of power-bounded elements of $A$ by $A^{\circ}$.
\item An analytic adic space $X$ is called \textit{locally noetherian} if $X$ has an affinoid open cover $\{U_i=\Spa(A_i,A_i^+)\}$ such that each $A_i$ is strongly noetherian Tate.
\item For an adic space $X$ and a point $x\in X$, let $k(x)$ denote the completed residue field at $x$.
\item For a complete non-archimedean field $K$, we use the term \textit{affinoid $K$-algebra} to refer to a \textit{topological $K$-algebra topologically of finite type over $K$}.
Let $\Aff_K$ denote the category of affinoid $K$-algebras.
For an affinoid $K$-algebra $A$, we simply write $\Spa A$ for $\Spa (A,A^{\circ})$.
\item For a complete non-archimedean field $K$, we use the term \textit{rigid space over $K$} to refer to an \textit{adic space locally of finite type over $\Spa K=\Spa(K,\Ocal_K)$}.
When the base field $K$ is clear from the context, we simply call it a \textit{rigid space}.
For a rigid space $X$ over $K$, a point $x\in X$ is \textit{classical} if there is an affinoid open neighborhood $U=\Spa A\subset X$ of $x$ such that $x$ corresponds to a maximal ideal of $A$.
Let $X^{\cl}$ denote the subset of classical points of $X$.
\item For an affinoid $K$-algebra $A$, let $\Coh(A)$ denote the category of finitely generated $A$-modules.
\item Unless otherwise stated, all actions are assumed to be continuous.
\end{itemize}

\subsection*{Acknowledgement}
The author is grateful to Heng Du, Eugen Hellmann, and Yoichi Mieda for their comments on this paper.
The author would also like to thank Ray Li for sharing his paper in preparation, and Brian Conrad for sharing the unpublished manuscript \cite{Conrad-Gabber}.
This work was supported by the RIKEN Special Postdoctoral Researcher Program.

\subsection*{Use of AI}
Before completing the initial draft, the author used Gemini (Google) as a conversational aid for experimenting with statements in Appendix \ref{section:rigid analytic Artin stack and quasi-analytic spaces}, and the resulting discussion led the author to discover the counterexample in Example \ref{ex:rigid analytic Artin stack and quasi-analytic space}. 
After completing the initial draft, the author used ChatGPT (OpenAI) for proofreading and for checking the manuscript for internal inconsistencies and possible gaps. 
All AI-assisted suggestions were independently checked by the author, who takes full responsibility for the entire content of the manuscript.

%%%%%%%%%%%%%%%%%%%%%%%%%%%%%%%%%%%%%%%%%%%%%%%%%%%%%%%%%%%%%%%%%%%%%%%%%%%%%%%%%%%%%%

\section{Preliminaries on rigid geometry}
Let $K$ be a complete non-archimedean field.
``Rigid space'' always means ``rigid space over $K$''.
First, let us recall the notion of flat morphisms in rigid geometry.

\begin{definition}
	Let $f\colon X\to Y$ be a morphism of rigid spaces.
	\begin{enumerate}
		\item For a classical point $x\in X^{\cl}$, $f$ is \textit{flat at $x$} if the morphism $\Ocal_{Y,f(x)}\to \Ocal_{X,x}$ is flat.
		\item The morphism $f\colon X\to Y$ is \textit{flat} if for any classical point $x\in X^{\cl}$, $f$ is flat at $x$.
	\end{enumerate}
\end{definition}

\begin{lemma}[{\cite[Bemerkung 3.8]{Kiehl68}}]
	Let $f \colon \Spa A\to \Spa B$ be a morphism of affinoid rigid spaces over $K$, and let $x\in (\Spa A)^{\cl}$ be a classical point corresponding to a maximal ideal $\mfrak\subset A$.
	Let $\nfrak\subset B$ be the preimage of $\mfrak$ in $B$.
	The following are equivalent.
	\begin{enumerate}
		\item The morphism $f$ is flat at $x$.
		\item The morphism $B_{\nfrak}\to A_{\mfrak}$ is flat.
		\item The morphism $B_{\nfrak}^{\wedge}\to A_{\mfrak}^{\wedge}$ is flat.
	\end{enumerate}
\end{lemma}

\begin{corollary}
	Let $f \colon \Spa A\to \Spa B$ be a morphism of affinoid rigid spaces over $K$.
	Then $f$ is flat if and only if $f\colon B\to A$ is flat.
\end{corollary}

\begin{lemma}[{\cite[Satz 3.3]{Kiehl68}}]
	Let $f\colon X\to Y$ be a morphism of rigid spaces over $K$, and let $x\in X^{\cl}$ be a classical point.
	If $f$ is flat at $x$, then there is an open neighborhood $U$ of $x$ such that the restriction $f\vert_U \colon U\to Y$ is flat.
\end{lemma}

For later use, we record the following criteria for flatness and smoothness.

\begin{lemma}\label{lem:fiberwise criterion for flatness and smoothness}
Let $S$ be a rigid space, and let $f\colon X\to Y$ be a morphism of rigid spaces over $S$, and let $p_1\colon X\to S$ and $p_2\colon Y\to S$ be the structure morphisms.
Let $x\in X^{\cl}$ be a classical point and set $s=p_1(x)$.
If $p_1$ is flat at $x$ and the fiber $f\times_S \Spa k(s)\colon p_1^{-1}(s)\to p_2^{-1}(s)$ is flat (resp.\ smooth) at $x$, then $f\colon X\to Y$ is also flat (resp.\ smooth) at $x$.
\end{lemma}
\begin{proof}
The assertion for flatness follows immediately from \cite[Th\'{e}or\`{e}me 11.3.10]{EGA4-3}.
The case for smoothness then follows from this together with \cite[Proposition 2.9]{FRG3}.
\end{proof}

\begin{corollary}\label{cor:smooth fiber and base}
	Let $f\colon X\to Y$ and $g\colon X\to S$ be morphisms of rigid spaces.
	Assume that $S$ is smooth over $K$.
	Let $x\in X^{\cl}$ be a classical point and set $s=g(x)$.
	If $g$ is flat at $x$ and the restriction $f\vert_{g^{-1}(s)}\colon g^{-1}(s)\to Y$ is smooth at $x$, then $f\colon X\to Y$ is also smooth at $x$.
\end{corollary}
\begin{proof}
	By applying Lemma \ref{lem:fiberwise criterion for flatness and smoothness} to the morphism $(f,g)\colon X\to Y\times S$ over $S$, we find that $(f,g)$ is smooth at $x$.
	Since $S$ is smooth over $K$, the projection $Y\times S\to Y$ is smooth.
	Therefore, the composition $f\colon X\xrightarrow{(f,g)} Y\times S\to Y$ is smooth at $x$.
\end{proof}

Next, let us introduce the notions of formal smoothness and versality for morphisms of rigid spaces.

\begin{definition}\label{def:formally smooth}
	Let $f\colon X\to Y$ be a morphism of rigid spaces, and let $x\in X^{\cl}$ be a classical point.
	The morphism $f\colon X\to Y$ is \textit{formally smooth} at $x$ if for any square-zero closed immersion $i\colon \Spa A\to \Spa B$ with $A,B$ finite local $K$-algebras, and any commutative diagram
	\[
	\xymatrix{
		\Spa A \ar[r]^i\ar[d] & \Spa B\ar[d]\ar@{-->}[ld]\\
		X\ar[r]_f & Y
	}
	\]
	such that the image of the closed point of $\Spa A$ in $X$ is $x$, there exists a dotted arrow making the diagram commute.
\end{definition}

\begin{theorem}\label{thm:formally smooth implies smooth rigid space}
	Let $f\colon X\to Y$ be a morphism of rigid spaces, and let $x\in X^{\cl}$ be a classical point.
	Then $f$ is smooth at $x$ if and only if $f$ is formally smooth at $x$.
\end{theorem}

This theorem is a special case of \cite[Proposition 2.9.5]{War17}.
Since we only consider formal smoothness at classical points, the proof is much simpler.
For the reader's convenience, we provide the proof below based on the argument in \cite[02HT]{stacks-project}.
Before giving the proof, we state the following lemmas.
For a morphism $A\to B$ of affinoid $K$-algebras, let $\Omega_{B/A}$ denote the module of continuous differentials defined in \cite[1.6.2]{Hub96}.
For a morphism $X\to Y$ of rigid spaces over $K$, let $\Omega_{X/Y}$ denote the coherent sheaf of continuous differentials obtained by gluing these modules.

\begin{lemma}\label{lem:differential module 1}
	Let $A\to B$ be a morphism of affinoid $K$-algebras, and let $I\subset B$ be an ideal.
	We set $B_n=B/I^n$.
	Then the natural morphism
	\[\Omega_{B/A}\otimes_B B_n\to \Omega_{B_{n+1}/A}\otimes_{B_{n+1}}B_n\]
	is an isomorphism.
\end{lemma}
\begin{proof}
	By \cite[Proposition 1.6.3]{Hub96}, there is a natural exact sequence
	\begin{align}\label{eq:second exact sequence}
	I^{n+1}/(I^{n+1})^2\xrightarrow{d} \Omega_{B/A}\otimes_B B_{n+1} \to \Omega_{B_{n+1}/A}\to 0.
    \end{align}
	Note that the composition
    \[I^{n+1}/(I^{n+1})^2\xrightarrow{d} \Omega_{B/A}\otimes_B B_{n+1} \to \Omega_{B/A}\otimes_B B_{n};\; x_1\cdots x_{n+1}\mapsto d(x_1\cdots x_{n+1})\]
    is zero. 
    Therefore, by applying $-\otimes_{B_{n+1}}B_n$ to the exact sequence \eqref{eq:second exact sequence}, we obtain the isomorphism
    \[\Omega_{B/A}\otimes_B B_n\xrightarrow{\sim} \Omega_{B_{n+1}/A}\otimes_{B_{n+1}}B_n.\]
\end{proof}

\begin{lemma}\label{lem:differential module 2}
	Let $A\to B$ be a morphism of affinoid $K$-algebras, and let $I\subset B$ be an ideal.
	Let $\alpha \colon B\to B/I$ be the projection.
	Assume that there exists a morphism $\beta \colon B/I\to B$ of affinoid $A$-algebras such that $\alpha\circ\beta=\id$.
	Then the natural sequence
	\[0\to I/I^2\to \Omega_{B/A}\otimes_{B} B/I \to \Omega_{(B/I)/A}\to 0\]
	is a split exact sequence.
\end{lemma}
\begin{proof}
	The morphism
	\[D\colon B\to I/I^2;\; x\mapsto x-\beta(\alpha(x))\]
	is a continuous $A$-derivation.
	Let $\Omega_{B/A}\to I/I^2$ denote the corresponding $B$-linear morphism.
	This induces a morphism $\Omega_{B/A}\otimes_{B} B/I\to I/I^2$, which gives a retraction of the natural morphism $I/I^2\to \Omega_{B/A}\otimes_{B} B/I$.
\end{proof}

\begin{proof}[Proof of Theorem \ref{thm:formally smooth implies smooth rigid space}]
	The ``only if'' direction is clear.
	Let us prove the ``if'' direction.
	We may assume that $X$ and $Y$ are affinoid rigid spaces, and we set $X=\Spa B$ and $Y=\Spa A$.
	Let $\mfrak\subset B$ be the maximal ideal corresponding to $x$.
	We choose a presentation $B=A\langle X_1,\ldots,X_n\rangle/J$, and we also regard $\mfrak$ as a maximal ideal of $A\langle X_1,\ldots,X_n\rangle$.
	It suffices to show that the natural morphism
	\[d\colon J/J^2\to \Omega_{A\langle X_1,\ldots,X_n\rangle/A}\otimes_{A\langle X_1,\ldots,X_n\rangle}B=\bigoplus_{i=1}^n BdX_i\]
	is a split injection over an open neighborhood of $x$.
	We set $B'=A\langle X_1,\ldots,X_n\rangle/J^2$.
	By Lemma \ref{lem:differential module 1}, we have an isomorphism
	\begin{align*}
		\Omega_{B'/A}\otimes_{B'}B\cong \bigoplus_{i=1}^n BdX_i.
	\end{align*}
	Thus, we need to show that the morphism 
	\begin{align}\label{eq:diff}
	d\colon J/J^2\to\Omega_{B'/A}\otimes_{B'}B
    \end{align}
	is a split injection over an open neighborhood of $x$.

	For $N\geq 2$, we set
	\begin{align*}
		B'_N&=B'/\mfrak^N=A\langle X_1,\ldots,X_n\rangle/(J^2+\mfrak^N),\\
		B_N&=B/\mfrak^N=A\langle X_1,\ldots,X_n\rangle/(J+\mfrak^N),\\
		J_N&=\Ker(B'_N\to B_N).
	\end{align*}
	Since $f$ is formally smooth at $x$, there exists a dotted arrow making the following diagram commute:
	\[
	\xymatrix{
		\Spa B_N \ar[r]\ar[d] & \Spa B'_N \ar[d]\ar@{-->}[ld]\\
		\Spa B \ar[r]_-{f} &\Spa A.
	}
	\]
	In other words, we get a decomposition $B\to B'_N\to B_N=B/\mfrak^N$ of the projection $B\to B_N$.
	Since the preimage of $\mfrak \subset B_N$ in $B'_N=A\langle X_1,\ldots,X_n\rangle/(J^2+\mfrak^N)$ is $J+\mfrak=\mfrak$ (note that $J\subset \mfrak$), the image of $\mfrak\subset B$ under the morphism $B\to B'_N$ is contained in $\mfrak\subset B'_N$. 
	Therefore, the decomposition $B\to B'_N\to B_N$ induces a splitting
	\[B_N\to B'_N\to B_N.\]
	By Lemma \ref{lem:differential module 2}, we get a split exact sequence
	\[0\to J_N \xrightarrow{d} \Omega_{B'_N/A}\otimes_{B'_N} B_N \to \Omega_{B_N/A}\to 0,\]
	where we note that $J_N$ is square-zero in $B'_N$.

	By applying the Artin-Rees lemma to the $\mfrak$-adic filtration $\{J\cap \mfrak^N\}_N$ of $J\subset A\langle X_1,\ldots,X_n\rangle$, we get an integer $c>0$ such that for any $N\geq c$, 
	\begin{align}\label{eq:Artin-Rees}
	J\cap \mfrak^N=\mfrak^{N-c}(J\cap \mfrak^c)	
	\end{align}	
	in $A\langle X_1,\ldots,X_n\rangle$.
	We prove that for any $N\geq c$, the morphism
	\[d\colon J_N\otimes_{B_N}B_{N-c}\to \Omega_{B'_N/A}\otimes_{B'_N}B_{N-c}\]
	coincides with the morphism
	\[d\colon (J/J^2)\otimes_B B_{N-c}\to \Omega_{B'/A}\otimes_{B'}B_{N-c}.\]
	There is a natural commutative diagram
	\[
	\xymatrix{
		(J/J^2)\otimes_B B_{N-c}\ar[r]^-{d}\ar[d] & \Omega_{B'/A}\otimes_{B'}B_{N-c}\ar[d]\\
		J_N\otimes_{B_N}B_{N-c}\ar[r]^-{d}& \Omega_{B'_N/A}\otimes_{B'_N}B_{N-c},
	}
	\]
	so it is enough to show that vertical morphisms are isomorphisms.
	The right vertical morphism is an isomorphism by Lemma \ref{lem:differential module 1}.
	Let us prove the left vertical morphism is an isomorphism.
	We set $L=\Ker(J/J^2\to J_N)$.
	Since we have
	\[J_N=(J+\mfrak^N)/(J^2+\mfrak^N)\cong J/(J\cap(J^2+\mfrak^N)),\]
	we get $L=(J\cap(J^2+\mfrak^N))/J^2.$
	By applying $-\otimes_{B}B_{N-c}$ to the exact sequence
	\[0\to L\to J/J^2\to J_N\to 0,\]
	we get a right exact sequence
	\[L\otimes_{B}B_{N-c} \to (J/J^2)\otimes_{B}B_{N-c}\to J_N\otimes_{B}B_{N-c}\to 0.\]
	Therefore, it suffices to show that the morphism $L\otimes_{B}B_{N-c} \to (J/J^2)\otimes_{B}B_{N-c}$ is zero.
	In other words, we need to show that 
	\[L=(J\cap(J^2+\mfrak^N))/J^2\to (J/J^2)\otimes_{B}B/(J+\mfrak^{N-c})\] 
	is zero.
	This is equivalent to showing that 
	\[J\cap(J^2+\mfrak^N)\subset J(J+\mfrak^{N-c})=J^2+\mfrak^{N-c}J.\]
	Since $J\cap(J^2+\mfrak^N)=J^2+(J\cap \mfrak^N)$, this follows from \eqref{eq:Artin-Rees}.

	By combining the above arguments, the morphism \eqref{eq:diff}
	\begin{align*}
		d\colon J/J^2\to\Omega_{B'/A}\otimes_{B'}B
	\end{align*}
	is a split injection after taking $-\otimes_B B_N$ for any $N$.
	By \cite[02HO]{stacks-project}, \eqref{eq:diff} is a split injection after taking $-\otimes_B B_{\mfrak}$.
	Since the source and target are finitely presented $B$-modules, there exists $g\in B\setminus\mfrak$ such that \eqref{eq:diff} is a split injection after taking $-\otimes_B B_g$.
	In particular, \eqref{eq:diff} is a split injection over the Zariski open neighborhood of $x$ defined by $g\neq 0$.
\end{proof}

%%%%%%%%%%%%%%%%%%%%%%%%%%%%%%%%%%%%%%%%%%%%%%%%%%%%%%%%%%%%%%%%%%%%%%%

\section{Preliminaries on rigid analytic Artin stacks}
Let $K$ be a complete non-archimedean field, and let $\pi$ be a pseudo-uniformizer of $K$.
We often regard a rigid space $S$ as the \'{e}tale sheaf on $\Aff_K^{\op}$ defined by $\Aff_K^{\op} \ni \Spa A\mapsto S(\Spa A)$.
For an \'{e}tale sheaf $\Fcal$ on $\Aff_K^{\op}$ and a rigid space $S$, we set $\Fcal(S)=\Hom(S,\Fcal)$.

\subsection{Basic notions}

\begin{definition}\label{def:quasi-analytic spaces and rigid analytic Artin stacks}
	\begin{enumerate}
	\item A \textit{quasi-analytic space} over $K$ is an \'{e}tale sheaf $\Fcal$ on $\Aff_K^{\op}$ such that
	\begin{enumerate}
		\item The diagonal morphism $\Delta \colon \Fcal\to \Fcal\times \Fcal$ is representable by rigid spaces.
		\item There exists an \'{e}tale surjection $U\twoheadrightarrow \Fcal$ from a rigid space $U$.
	\end{enumerate}
	\item A \textit{rigid analytic Artin stack} over $K$ is an \'{e}tale stack $\Xcal$ on $\Aff_K^{\op}$ such that
	\begin{enumerate}
		\item The diagonal morphism $\Delta \colon \Xcal\to \Xcal\times \Xcal$ is representable by quasi-analytic spaces.
		\item There exists a smooth surjection $U\twoheadrightarrow \Xcal$ from a rigid space $U$.
	\end{enumerate}
	\end{enumerate}
\end{definition}

\begin{remark}
	Quasi-analytic spaces are referred to as \textit{\'{e}tale adic spaces} in \cite{War17}, and as \textit{\'{e}tale spaces} in \cite{OhShimizu25}.
	Rigid analytic Artin stacks are referred to as \textit{adic stacks} in \cite{War17,OhShimizu25}.
\end{remark}

\begin{remark}
	As usual, we can show that quasi-analytic spaces and rigid analytic Artin stacks satisfy fppf descent by using Theorem \ref{thm:smooth surjection etale topology}.
	Therefore, our definitions coincide with those in \cite[Definition 5.1.8 and 5.1.10]{EGH23}.
\end{remark}

\begin{remark}
Some references use the category $\Rig_K$ of rigid spaces over $K$ as a test category. 
Since we only deal with \'{e}tale sheaves and stacks on $\Aff_K^{\op}$ in this paper, one can safely use $\Rig_K$ as a test category as well.
\end{remark}

\begin{remark}
	In contrast to algebraic geometry, an \'{e}tale sheaf which is a rigid analytic Artin stack is not necessarily a quasi-analytic space.
	See Appendix \ref{section:rigid analytic Artin stack and quasi-analytic spaces}.
\end{remark}

Since $\Rig_K$ forms an \textit{admissible category} in the sense of \cite[Definition 2.2.6]{War17}, the above remark allows us to freely apply the notions introduced in \cite[Sections~3.2 and 3.3]{War17} (specifically, \cite[Definitions~3.2.7 and 3.3.2]{War17}).
For properties of morphisms of rigid spaces that are smooth or \'{e}tale local on the (source and) target, see also \cite[Proposition 3.12, Proposition 3.14]{OhShimizu25}.
Although the fact that a locally closed immersion is smooth local on the target is not mentioned above, we will need it later and thus provide a proof here.
\begin{lemma}\label{lem:locally closed immersion smooth local on the target}
	In the category of locally noetherian analytic adic spaces, the property of being a locally closed immersion is smooth local on the target.
\end{lemma}
\begin{proof}
	This can be proved in the same way as in \cite[02YM]{stacks-project} by noting that smooth morphisms between locally noetherian analytic adic spaces are open maps by \cite[Proposition 1.7.8]{Hub96}.
\end{proof}

Next, we define the notion of classical points of rigid analytic Artin stacks.

\begin{definition}
	Let $\Xcal$ be a rigid analytic Artin stack.
	Let $F_1$ and $F_2$ be finite extensions of $K$, and let $x\colon \Spa F_1\to \Xcal$ and $y\colon \Spa F_2\to \Xcal$ be morphisms.
	We say that $x$ and $y$ are equivalent if there exists a finite extension $F_3$ of $K$ and a commutative diagram
	\[
	\xymatrix{
		\Spa F_3 \ar[r]\ar[d] & \Spa F_2 \ar[d]^-{y}\\
		\Spa F_1 \ar[r]_-{x} &\Xcal.
	}
	\]
	This defines an equivalence relation on the set of morphisms $\Spa F\to \Xcal$ with $F/K$ finite, and we write $\Xcal^{\cl}$ for the set of equivalence classes.
	We refer to an element of $\Xcal^{\cl}$ as a classical point of $\Xcal$.
	By abuse of language, we also call a morphism $\Spa F\to \Xcal$ with $F/K$ finite a classical point of $\Xcal$.
\end{definition}

\begin{remark}
	For a rigid space $X$, the set of classical points of $X$ defined above coincides with the usual one.
\end{remark}

\begin{remark}\label{rem:empty classical point}
	Let $\Xcal$ be a rigid analytic Artin stack, and let $U\twoheadrightarrow \Xcal$ be a smooth surjection from a rigid space $U$.
	We set $R=U\times_{\Xcal}U$.
	Then $R^{\cl}$ defines an equivalence relation on $U^{\cl}$, and the natural map
	\[U^{\cl}/R^{\cl}\to \Xcal^{\cl}\]
	is bijective.
	Since any non-empty rigid space has a classical point, we find that any non-empty rigid analytic Artin stack has a classical point.
\end{remark}

\subsection{Ring theoretic properties of rigid analytic Artin stacks}

We would like to define properties of rigid analytic Artin stacks such as being reduced, normal, Cohen-Macaulay, etc.
The definitions of these notions are straightforward since the rings appearing in rigid geometry are excellent.
\begin{theorem}[{\cite[Satz 3.3.3]{BKKN67}, \cite[Theorem 1.1.3]{Conrad99}, \cite[7.3.2. Proposition 3]{BGR}}]\label{thm:excellence}
	The following hold.
	\begin{enumerate}
		\item An affinoid $K$-algebra $A$ is excellent.
		\item Let $X$ be a rigid space, and let $x\in X$ be a classical point.
		Then the local ring $\Ocal_{X,x}$ is excellent.
		\item For an affinoid rigid space $X=\Spa A$ and a classical point $x\in X$ corresponding to a maximal ideal $\mfrak\subset A$, the natural morphism $A_{\mfrak}\to \Ocal_{X,x}$ of local rings becomes an isomorphism after taking their completions.
	\end{enumerate}
\end{theorem}

Let $\Pcal$ be one of the following properties of noetherian rings: reduced, normal, regular, Gorenstein, Cohen-Macaulay, complete intersection, or Serre's condition $R_i$ or $S_i$. 

\begin{definition}
Let $X$ be a rigid space.
For a classical point $x\in X$, $X$ is said to be $\Pcal$ at $x$ if $\Ocal_{X,x}$ satisfies $\Pcal$.
We say that $X$ is $\Pcal$ if $X$ is $\Pcal$ at every classical point of $X$.
As usual, if $\Pcal$ is ``a complete intersection'', then we use the term ``local complete intersection''.
\end{definition}

\begin{remark}
	By Theorem \ref{thm:excellence}, a rigid space $X$ is $\Pcal$ at a classical point $x\in X$ if and only if the completion $\Ocal_{X,x}^{\wedge}$ satisfies $\Pcal$.
	Moreover, an affinoid rigid space $X=\Spa A$ is $\Pcal$ if and only if $A$ satisfies $\Pcal$.
\end{remark}

\begin{lemma}\label{lem:ring theoretic property descent}
	Let $f\colon X\to Y$ be a smooth morphism of rigid spaces.
	For a classical point $x\in X$, $X$ is $\Pcal$ at $x$ if and only if $Y$ is $\Pcal$ at $f(x)$.
	In particular, if $f$ is surjective, then $X$ is $\Pcal$ if and only if $Y$ is $\Pcal$. 
\end{lemma}
\begin{proof}
	Since $f$ is smooth, the morphism $\Ocal_{Y,f(x)}\to \Ocal_{X,x}$ of local rings is regular.
	Therefore, the claim follows from \cite[Theorems 23.9 and 32.2]{Matsumura-Comm-Ring} and \cite[1.9.2]{Avramov77}.
\end{proof}

For later use, we record the following corollary.
\begin{corollary}\label{cor:nilradical}
Let $f\colon X\to Y$ be a smooth morphism of rigid spaces, and let $\sqrt{(0)}_Y$ be the nilradical of $\Ocal_Y$.
Then $f^*\sqrt{(0)}_Y$ is the nilradical of $\Ocal_{X}$.
\end{corollary}
\begin{proof}
	Let $Y^{\red}$ be the reduction of $Y$.
	It suffices to show that $Y^{\red}\times_Y X$ is reduced, which follows from Lemma \ref{lem:ring theoretic property descent}.
\end{proof}

\begin{definition}
	Let $\Xcal$ be a rigid analytic Artin stack, and $x\in \Xcal^{\cl}$ be a classical point.
	We say that $\Xcal$ is $\Pcal$ at $x$ if there exist a smooth morphism $f\colon U\to \Xcal$ from a rigid space $U$ and a classical point $y\in U^{\cl}$ such that $f(y)=x$ and $U$ is $\Pcal$ at $y$.
	We say that $\Xcal$ is $\Pcal$ if $\Xcal$ is $\Pcal$ at every classical point of $\Xcal$.
\end{definition}

\begin{remark}
	Let $\Xcal$ be a rigid analytic Artin stack, and let $x\in \Xcal^{\cl}$ be a classical point.
	By Lemma \ref{lem:ring theoretic property descent}, if $\Xcal$ is $\Pcal$ at $x$, then for any smooth morphism $f\colon U\to \Xcal$ from a rigid space $U$ and a classical point $y\in U^{\cl}$ with $f(y)=x$, $U$ is $\Pcal$ at $y$.
	Similarly, if $\Xcal$ is $\Pcal$, then for any smooth morphism $f\colon U\to \Xcal$ from a rigid space $U$, $U$ is $\Pcal$.
\end{remark}

\begin{lemma}\label{lem:lci vs gci}
    Let $X=\Spa A$ be an affinoid rigid space.
    Then $X$ is a local complete intersection if and only if there exists an affinoid open covering $\{U_i=\Spa A_i\}_{i \in I}$ of $X$ such that each $A_i$ is isomorphic to a quotient of the form 
	\[K\langle X_1,\ldots, X_n\rangle/(f_1,\ldots,f_m)\]
	where $f_1,\ldots,f_m$ is a Koszul-regular sequence in $K\langle X_1,\ldots, X_n\rangle$ (i.e., the Koszul complex $K_{\bullet}(f_1,\ldots,f_m)$ is acyclic in positive degrees).
\end{lemma}
\begin{proof}
	The ``if'' direction is clear. Let us prove the ``only if'' direction.
	We choose a surjection $K\langle X_1,\ldots, X_n\rangle \to A$, and let $I$ be its kernel.
	Let $\mfrak \subset A$ be a maximal ideal, and let $\nfrak \subset K\langle X_1,\ldots, X_n\rangle$ be its preimage.
	Then we have 
	\[K\langle X_1,\ldots, X_n\rangle_{\nfrak}/I_{\nfrak}\cong A_{\mfrak}.\]
	Since $A$ is a local complete intersection and $K\langle X_1,\ldots, X_n\rangle_{\nfrak}$ is a regular local ring, $I_{\nfrak}$ is generated by a regular sequence $f_1,\ldots,f_m\in K\langle X_1,\ldots, X_n\rangle_{\nfrak}$.
	By a standard limit argument, we can take an element $g\in K\langle X_1,\ldots, X_n\rangle$ such that $g\notin \nfrak$ and the sequence $f_1,\ldots,f_m\in K\langle X_1,\ldots, X_n\rangle[1/g]$ generates $I[1/g]$ and forms a Koszul-regular sequence.
	Then for any $N>0$, we have
	\[K\langle X_1,\ldots, X_n\rangle\langle \pi^N/g\rangle/(f_1,\ldots,f_m) \cong A\langle \pi^N/g\rangle,\]
	where $f_1,\ldots,f_m$ is a Koszul-regular sequence in $K\langle X_1,\ldots, X_n\rangle\langle \pi^N/g\rangle$.
	Hence, we obtain an isomorphism
	\[K\langle X_1,\ldots, X_n,X_{n+1}\rangle/(gX_{n+1}-\pi^N,f_1,\ldots,f_m) \cong A\langle \pi^N/g\rangle,\]
	and $gX_{n+1}-\pi^N,f_1,\ldots,f_m$ is a Koszul-regular sequence in $K\langle X_1,\ldots, X_n,X_{n+1}\rangle$.
	By varying $N$ and $\mfrak$, we obtain the desired affinoid open covering.
\end{proof}

\subsection{Coherent sheaves and Fitting ideals}

\begin{definition}\label{def:coherent sheaves}
	Let $\Xcal$ be a rigid analytic Artin stack.
	We define the category $\Coh(\Xcal)$ of coherent sheaves on $\Xcal$
	by
	\[\Coh(\Xcal)=\varprojlim_{\Spa A\to \Xcal}\Coh(A),\]
	where the limit is taken over all morphisms from affinoid rigid spaces $\Spa A$ to $\Xcal$.
	For a morphism $f\colon \Ycal\to \Xcal$ of rigid analytic Artin stacks, let $f^*\colon \Coh(\Xcal)\to \Coh(\Ycal)$ denote the pullback functor.
\end{definition}

Since the categories of coherent sheaves on rigid spaces satisfy \'{e}tale descent, this definition is well-behaved.
Moreover, since smooth morphisms are flat, $\Coh(\Xcal)$ is an abelian category (see the following explicit description).

\begin{example}
	\begin{enumerate}
		\item Let $X$ be a quasi-analytic space over $K$, and $U\twoheadrightarrow X$ be an \'{e}tale surjection from a rigid space $U$.
		Then we have
		\[\Coh(X)\cong\varprojlim(\Coh(U)\rightrightarrows \Coh(U\times_{X}U)\triplerightarrow\Coh(U\times_X U\times_X U)).\]
		More explicitly, $\Coh(X)$ is the category of pairs $(F,\tau\colon p_1^*F\simeq p_2^*F)$ consisting of
		\begin{itemize}
			\item A coherent sheaf $F$ on $U$,
			\item An isomorphism $\tau\colon p_1^*F\simeq p_2^*F$ on $U\times_X U$ satisfying the cocycle condition, where $p_i\colon U\times_X U\to U$ is the $i$-th projection.
		\end{itemize}
		\item Let $\Xcal$ be a rigid analytic Artin stack over $K$, and let $U\twoheadrightarrow \Xcal$ be a smooth surjection from a rigid space $U$.
		Then we have
		\[\Coh(\Xcal)\cong\varprojlim(\Coh(U)\rightrightarrows \Coh(U\times_{\Xcal}U)\triplerightarrow\Coh(U\times_{\Xcal} U\times_{\Xcal} U)).\]
	\end{enumerate}
\end{example}

\begin{definition}
	Let $n$ be a non-negative integer.
	A coherent sheaf $M$ on $\Xcal$ is a vector bundle (of rank $n$) if for every morphism $f\colon \Spa A\to \Xcal$, $\Gamma(\Spa A,f^*M)$ is a finite projective $A$-module (of rank $n$).
\end{definition}
\begin{remark}
	Let $\Xcal$ be a rigid analytic Artin stack over $K$, and let $f\colon U\twoheadrightarrow \Xcal$ be a smooth surjection from a rigid space $U$. 
	Then a coherent sheaf $M$ on $\Xcal$ is a vector bundle (of rank $n$) if and only if $f^*M$ on $U$ is a vector bundle (of rank $n$).
\end{remark}

\begin{definition}
	Let $\Xcal$ be a rigid analytic Artin stack.
	An \textit{ideal sheaf} of $\Xcal$ is a coherent $\Ocal_{\Xcal}$-submodule of $\Ocal_{\Xcal}$.
\end{definition}

\begin{example}\label{example:nilradical}
	\begin{enumerate}
	\item Let  $X$ be a quasi-analytic space, and choose an \'{e}tale surjection $U\twoheadrightarrow X$ from a rigid space $U$.
	Let $p_i\colon U\times_{X} U\to U$ be the $i$-th projection.
	Let $\sqrt{(0)}_U$ be the nilradical of $\Ocal_U$.
	Then $p_i^*\sqrt{(0)}_U$ is the nilradical of $\Ocal_{U\times_{X}U}$ by Corollary \ref{cor:nilradical}.
	Therefore, $\sqrt{(0)}_U$ descends to an ideal sheaf of $X$, which we refer to as the \textit{nilradical} of $\Ocal_{X}$.
	The nilradical of $\Ocal_{X}$ is independent of the choice of the \'{e}tale surjection $U\twoheadrightarrow X$.
	\item Let $\Xcal$ be a rigid analytic Artin stack, and choose a smooth surjection $U\twoheadrightarrow \Xcal$ from a rigid space $U$.
    Since $U\times_{\Xcal}U$ is a quasi-analytic space, the preceding construction, together with Corollary \ref{cor:nilradical}, shows that $\sqrt{(0)}_U$ descends to an ideal sheaf $\sqrt{(0)}_{\Xcal}$, independently of the choice of $U$.
	We call $\sqrt{(0)}_{\Xcal}$ the \textit{nilradical} of $\Ocal_{\Xcal}$.
	\end{enumerate}
\end{example}

As usual, we can prove the following.

\begin{lemma}
	Let $\Xcal$ be a rigid analytic Artin stack.
	Then there is a natural bijection between the set of Zariski closed substacks of $\Xcal$ and the set of ideal sheaves of $\Xcal$.
\end{lemma}

\begin{definition}
	Let $\Xcal$ be a rigid analytic Artin stack. 
	The \textit{reduction} $\Xcal^{\red}$ is the Zariski closed substack defined by the nilradical of $\Ocal_{\Xcal}$.
\end{definition}

\begin{definition}
	Let $\Xcal$ be a rigid analytic Artin stack, and let $I$ be an ideal sheaf of $\Xcal$.
	The \textit{reduced Zariski closed substack} of $\Xcal$ defined by $I$ is the reduction of the Zariski closed substack defined by $I$.
\end{definition}

\begin{lemma}\label{lem:Zariski closed inclusion classical points}
	Let $\Xcal$ be a rigid analytic Artin stack, and let $Z_1,Z_2\subset \Xcal$ be Zariski closed substacks.
	Assume that $Z_1$ is reduced.
	If we have $Z_1^{\cl}\subset Z_2^{\cl}$, then $Z_1\subset Z_2$.
\end{lemma}
\begin{proof}
	By Remark \ref{rem:empty classical point}, we obtain $Z_1 \cap (\Xcal\setminus Z_2)=\emptyset$.
	Since $Z_1$ is reduced, the above implies that $Z_1\subset Z_2$.
\end{proof}

We would like to define Fitting ideals of coherent sheaves on a rigid analytic Artin stack.
Let us recall the properties of Fitting ideals.
For details, see \cite[20.2]{Eisenbud95} or \cite[07Z6]{stacks-project}.

\begin{definition}
	Let $A$ be a ring, and let $M$ be a finitely generated $A$-module.
	Let $k$ be a non-negative integer.
	We choose a presentation of $M$
	\[\bigoplus_{j\in J}A\to A^{\oplus n}\to M\to 0.\]
	Let $\alpha=(a_{ij})_{i=1,\ldots,n, j\in J}$ be the matrix corresponding to $\bigoplus_{j\in J}A\to A^{\oplus n}$.
	We define the $k$-th \textit{Fitting ideal} $\Fit_k(M)$ as the ideal generated by the $(n-k)\times (n-k)$-minors of $\alpha$, which is independent of the choice of the presentation (\cite[Corollary-Definition 20.4]{Eisenbud95} or \cite[07Z8]{stacks-project}).
\end{definition}

\begin{example}
	Let $A$ be a ring, and let $M$ be an $A$-module generated by $n$ elements.
	Then for $k\geq n$, the $k$-th Fitting ideal $\Fit_k(M)$ is the unit ideal $A$.
\end{example}

\begin{lemma}[{\cite[Corollary 20.5]{Eisenbud95}, \cite[07ZA]{stacks-project}}]\label{lem:Fitting ideal base change}
	Let $A\to A'$ be a morphism of rings, and let $M$ be a finitely generated $A$-module.
	Let $k$ be a non-negative integer.
	Then the ideal $\Fit_k(M\otimes_A A')\subset A'$ is generated by the image of $\Fit_k(M)$ under $A\to A'$.
\end{lemma}

\begin{lemma}[{\cite[Proposition 20.8]{Eisenbud95}, \cite[07ZD]{stacks-project}}]
	Let $A$ be a ring, and let $M$ be a finitely generated $A$-module.
	Let $r$ be a non-negative integer.
	The following are equivalent.
	\begin{enumerate}
		\item $M$ is a finite projective $A$-module of rank $r$.
		\item $\Fit_{r-1}(M)=0$ and $\Fit_{r}(M)=A$, where we set $\Fit_{-1}(M)=0$.
	\end{enumerate}
\end{lemma}

Lemma \ref{lem:Fitting ideal base change} allows us to define the Fitting ideals of coherent sheaves on rigid spaces by gluing.

\begin{definition}
	Let $\Xcal$ be a rigid analytic Artin stack, and let $M$ be a coherent sheaf on $\Xcal$.
    Let $k$ be a non-negative integer, and choose a smooth surjection $f\colon U\twoheadrightarrow\Xcal$ from a rigid space $U$.
    By Lemma \ref{lem:Fitting ideal base change}, the $k$-th Fitting ideal $\Fit_k(f^*M)$ descends to an ideal sheaf $\Fit_k(M)$ of $\Xcal$, independently of the choice of $U$, as in Example \ref{example:nilradical}.
    We call $\Fit_k(M)$ the \textit{$k$-th Fitting ideal} of $M$.
    We also set $\Fit_{-1}(M)=0$.
\end{definition}

\begin{example}
	Let $\Xcal$ be a rigid analytic Artin stack, and let $M$ be a coherent sheaf on $\Xcal$.
	Let $Z_k\subset \Xcal$ be the Zariski closed substack defined by the $k$-th Fitting ideal $\Fit_k(M)$, and set $Z_k^{\circ}=Z_{k-1}\setminus Z_{k}$. 
	Then the restriction of $M$ to $Z_k^{\circ}$ is a vector bundle of rank $k$.
\end{example}

\subsection{Dimension of rigid analytic Artin stacks}

First, let us recall the theory of dimension of rigid spaces.
Note that ``rigid space'' means an adic space locally of finite type over $\Spa K=\Spa (K,\Ocal_K)$.

\begin{definition}
	Let $X$ be a rigid space, and let $x\in X^{\cl}$ be a classical point.
	The dimension $\dim_{x}X$ of $X$ at $x$ is defined as the Krull dimension of the local ring $\Ocal_{X,x}$.
	The dimension $\dim X$ of $X$ is defined as $\sup_{x\in X^{\cl}} \dim_x X$. \footnote{In our convention, $\dim \emptyset =-\infty$.}
	For $n\in \Zbb_{\geq 0}\cup \{-\infty\}$, we say that $X$ is \textit{equidimensional} of dimension $n$ if $\dim X=n$ and for every classical point $x\in X^{\cl}$, $\dim_x X=n$.
\end{definition}

\begin{remark}\label{rem:comparison with Huber}
	Let $X=\Spa A$ be an affinoid rigid space.
	For a classical point $x\in X^{\cl}$ corresponding to a maximal ideal $\mfrak\subset A$, the dimension $\dim_x X$ at $x$ coincides with the Krull dimension $\dim A_{\mfrak}$.
	In particular, $\dim X$ is equal to the Krull dimension $\dim A$.
	Therefore, by \cite[Lemma 1.8.6 (ii)]{Hub96}, $\dim X$ coincides with the dimension of $X$ as a spectral space (i.e., the supremum of the lengths $n$ of the chains of specializations $x_0\succ x_1\succ\cdots\succ x_n$).
	Consequently, for a (not necessarily affinoid) rigid space $Y$, $\dim Y$ is equal to its dimension as a locally spectral space (\cite[Definition 1.8.1]{Hub96}).
\end{remark}

\begin{definition}\label{defn:relative dimension rigid space}
	Let $f\colon X\to Y$ be a morphism of rigid spaces, and let $x\in X^{\cl}$ be a classical point.
	The relative dimension $\dim_x f$ (or $\dim_x X/Y$) of $f$ at $x$ is defined as the Krull dimension $\dim \Ocal_{f^{-1}(f(x)),x}$, where $f^{-1}(f(x))=X\times_Y \Spa k(f(x))$.
	The relative dimension $\dim f$ (or $\dim X/Y$) of $f$ is defined as $\sup_{x\in X^{\cl}} \dim_x f$.
	We say that $f$ is of \textit{pure dimension} $n$ if $\dim f=n$ and for every classical point $x\in X^{\cl}$, $\dim_x f=n$.
\end{definition}

\begin{lemma}\label{lem:fiber dimension}
	Let $f\colon X\to Y$ be a morphism of rigid spaces, and let $x\in X^{\cl}$ be a classical point.
	Then we have $\dim_x X\leq \dim_x X/Y +\dim_{f(x)} Y$.
	If $\Ocal_{Y,f(x)}\to \Ocal_{X,x}$ is flat, then we have $\dim_x X= \dim_x X/Y +\dim_{f(x)} Y$.
	In particular, we have $\dim X\leq \dim X/Y +\dim Y$.
\end{lemma}
\begin{proof}
	This follows immediately from \cite[(13.B) Theorem 19]{Mat80}.
\end{proof}

\begin{corollary}\label{cor:fiber dimension transitive}
	Let $f\colon X\to Y$ and $g\colon Y\to Z $ be morphisms of rigid spaces, and let $x\in X^{\cl}$ be a classical point.
	Assume that $f$ is flat at $x$.
	Then we have 
	\[\dim_x (g\circ f)=\dim_x f+\dim_{f(x)}g.\]
\end{corollary}

\begin{definition}\label{defn:locally Zariski closed partition, rigid space}
	Let $X$ be a rigid space.
	A \textit{locally Zariski closed covering} of $X$ is a family of locally Zariski closed subspaces $\{X_i\}_{i\in I}$ of $X$ such that $X=\bigcup_{i\in I}X_i$ as sets.
	In this case, we write $X\sim \bigcup_{i\in I}X_i$.
	We say that $\{X_i\}_{i\in I}$ is a \textit{locally Zariski closed partition} if it is a locally Zariski closed covering and $X_i\cap X_j=\emptyset$ for any $i\neq j$.
	In this case, we write $X\sim \bigsqcup_{i\in I}X_i$.
\end{definition}

\begin{example}
	Let $X$ be a rigid space, and let $Z\subset X$ be a Zariski closed subspace.
	Then $Z \sqcup (X\setminus Z)$ is a locally Zariski closed partition of $X$.
\end{example}

\begin{example}
	Let $X$ be a rigid space, and let $M$ be a coherent sheaf on $X$.
	Let $Z_k\subset X$ be the Zariski closed subspace defined by the $k$-th Fitting ideal $\Fit_k(M)$, and set $Z_k^{\circ}=Z_{k-1}\setminus Z_{k}$. 
	Then $\bigsqcup_{k\geq 0}Z_k^{\circ}$ is a locally Zariski closed partition of $X$.
\end{example}

\begin{example}\label{ex:locally Zariski closed partition rigid space}
	Let $X$ be a rigid space, and let $n\geq 1$ be an integer.
	Let $\{Z_{\kbf}\}_{\kbf\in \Zbb_{\geq 0}^n}$ be a family of Zariski closed subspaces of $X$ such that $Z_0=X$ and $Z_{\kbf_2}\subseteq Z_{\kbf_1}$ whenever $\kbf_1\leq \kbf_2$, where the order on $\Zbb_{\geq 0}^n$ is the componentwise order.
	Assume that for any sequence $\kbf_1<\kbf_2<\cdots$, $\bigcap_m Z_{\kbf_m}=\emptyset$.
	For $\kbf \in \Zbb_{\geq 0}^n$, set $Z_{\kbf}^{\circ}=Z_{\kbf}\setminus \bigcup_{\kbf'>\kbf}Z_{\kbf'}$.
	Then $\{Z_{\kbf}^{\circ}\}_{\kbf\in \Zbb_{\geq 0}^n}$ is a locally Zariski closed covering of $X$.
	Note that it suffices to check the condition $\bigcap_m Z_{\kbf_m}=\emptyset$ on classical points.
	More precisely, if for any classical point $x\in X^{\cl}$, there exists $m\geq 1$ such that $x\notin Z_{\kbf_m}$, then $\bigcap_{m}Z_{\kbf_m}=\emptyset$.
\end{example}

\begin{lemma}\label{lem:locally Zariski closed partition dimension rigid space}
	Let $X$ be a rigid space, and let $X\sim \bigcup_{i\in I}X_i$ be a locally Zariski closed covering.
	Then we have $\dim X=\sup_{i\in I} \dim X_i$.
\end{lemma}
\begin{proof}
	This follows from Remark \ref{rem:comparison with Huber} and the fact that locally Zariski closed subspaces are stable under generalizations.
\end{proof}

\begin{corollary}
	Let $X$ be a rigid space, and let $Z\subset X$ be a Zariski closed subspace.
	Then we have $\dim X=\max \{\dim Z, \dim (X\setminus Z)\}.$
\end{corollary}

For later use, we record the following lemmas.

\begin{lemma}\label{lem:c.i.}
	Let $A$ be a smooth affinoid $K$-algebra such that $\Spa A$ is equidimensional of dimension $d\geq 0$.
	Let $f_1,\ldots, f_m \in A$, and set $X=\Spa A/(f_1,\ldots, f_m)$.
	If $\dim X=d-m$, then $X$ is a local complete intersection and equidimensional of dimension $d-m$.
\end{lemma}
\begin{proof}
	This follows immediately from \cite[(12.I) Theorem 18, (16.B) Theorem 32]{Mat80}.
\end{proof}

\begin{lemma}\label{lem:normal outside codimension 2}
	Let $n\geq 2$ be an integer. 
	Let $X$ be a Cohen-Macaulay rigid space such that for every classical point $x\in X^{\cl}$, $\dim_x X\geq n$, and let $Z\subset X$ be a Zariski closed subspace such that $\dim Z\leq n-2$.
	If $X\setminus Z$ is normal, then $X$ is normal.
\end{lemma}
\begin{proof}
	We may assume that $X=\Spa A$ is an affinoid rigid space, and $Z$ is defined by an ideal $I\subset A$.
	Since $A$ is Cohen-Macaulay, $A$ satisfies Serre's condition $(S_2)$.
	Let $\pfrak \subset A$ be a prime ideal such that $\height \pfrak \leq 1$.
	Since $\dim Z=\dim A/I\leq n-2$ and $\dim A/\pfrak \geq n-1$, there is a maximal ideal $\mfrak\subset A$ such that $I\not\subset \mfrak$ and $\pfrak\subset \mfrak$.
	Since $\mfrak$ corresponds to a classical point of $X\setminus Z$, $A_{\mfrak}$ is normal.
	Therefore, $A_{\pfrak}$ is also normal.
	Thus, $A_{\pfrak}$ is regular as $\dim A_{\pfrak}\leq 1$.
	In other words, $A$ satisfies Serre's condition $(R_1)$.
	Therefore, by \cite[(17.I) Theorem 39]{Mat80}, $A$ is normal.
\end{proof}

Next, we define the dimension of quasi-analytic spaces and rigid analytic Artin stacks.

\begin{definition}
	Let $X$ be a quasi-analytic space.
	For a classical point $x\in X^{\cl}$, we define the dimension $\dim_{x} X$ of $X$ at $x$ as follows:
	Choose an \'{e}tale morphism $f\colon U\to X$ from a rigid space $U$ and a classical point $y\in U^{\cl}$ such that $f(y)=x$.
	Then we define $\dim_x X=\dim_y U$.
	Note that, by Lemma \ref{lem:fiber dimension}, it does not depend on the choice of $f$ and $y$.
	The dimension $\dim X$ of $X$ is defined as $\sup_{x\in X^{\cl}}\dim_x X$.
\end{definition}

As in Definition \ref{defn:relative dimension rigid space}, we define the relative dimension of morphisms representable by quasi-analytic spaces.

\begin{definition}\label{def:relative dimension stack}
	Let $\Xcal$ and $\Ycal$ be rigid analytic Artin stacks, and let $f\colon \Xcal \to \Ycal$ be a morphism representable by quasi-analytic spaces.
	\begin{enumerate}
	\item For a classical point $x\in \Xcal^{\cl}$, the relative dimension $\dim_x f$ (or $\dim_x \Xcal/\Ycal$) of $f$ at $x$ is defined as follows:
	Choose a representative $\Spa F\to \Ycal$ of the classical point $f(x)$.
	We also choose a classical point $x'$ of $\Xcal\times_{\Ycal} \Spa F$ whose image in $\Xcal$ is $x$.
	We define $\dim_x f$ as the dimension of $\Xcal\times_{\Ycal} \Spa F$ at $x'$.
	It does not depend on the choices of $\Spa F\to \Ycal$ and $x'$ by the following lemma.
	\item The relative dimension $\dim f$ (or $\dim \Xcal/\Ycal$) of $f$ is defined as $\sup_{x\in \Xcal^{\cl}} \dim_x f$.
	\item We say that $f$ is of \textit{pure dimension} $n$ if $\dim f=n$ and for every classical point $x\in \Xcal^{\cl}$, $\dim_x f=n$.
	\end{enumerate}
\end{definition}

\begin{lemma}\label{lem:independent of the lifting}
	We use the same notation as in Definition \ref{def:relative dimension stack}.
	The definition of $\dim_x f$ is independent of the choices of $\Spa F\to \Ycal$ and $x'$.
\end{lemma}
\begin{proof}
	Let $F'$ be a finite extension of $F$, and take a classical point $x''$ of $\Xcal\times_{\Ycal} \Spa F'$ whose image in $\Xcal\times_{\Ycal} \Spa F$ is $x'$.
	Then we have 
	\[\dim_{x'} (\Xcal\times_{\Ycal} \Spa F)=\dim_{x''} (\Xcal\times_{\Ycal} \Spa F').\]
	Therefore, we can freely replace $F$ with its finite extension to compare different choices.

	Let $x_1$ and $x_2$ be classical points of $\Xcal\times_{\Ycal} \Spa F$ whose images in $\Xcal$ are $x$.
	After replacing $F$ with a suitable finite extension, we can choose a representative $x_i\colon \Spa F\to \Xcal\times_{\Ycal}\Spa F$ of $x_i$ for $i=1,2$ such that $\Spa F\xrightarrow{x_1} \Xcal\times_{\Ycal}\Spa F\to\Xcal$ coincides with $\Spa F\xrightarrow{x_2} \Xcal\times_{\Ycal}\Spa F\to \Xcal$.
	For $i=1,2$, let $y_i$ be the composition $\Spa F\xrightarrow{x_i} \Xcal\times_{\Ycal}\Spa F\to\Spa F$.
	Then we obtain morphisms
	\[(x_1,x_2)\colon \Spa F\to (\Xcal\times_{\Ycal}\Spa F)\times_{\Xcal}(\Xcal\times_{\Ycal}\Spa F)=\Xcal\times_{\Ycal}(\Spa F\times_{\Ycal}\Spa F)\]
	and
	\[(y_1,y_2)\colon \Spa F\to \Spa F\times_{\Ycal}\Spa F\]
	with a commutative diagram
	\[\xymatrix{
		\Spa F \ar[rr]^-{(x_1,x_2)}\ar[rrd]_-{(y_1,y_2)}&&\Xcal\times_{\Ycal}(\Spa F\times_{\Ycal}\Spa F)\ar[d]\\
		& & \Spa F\times_{\Ycal}\Spa F.
	}
	\]
	We set $\Xcal'=(\Xcal\times_{\Ycal}(\Spa F\times_{\Ycal}\Spa F)) \times_{\Spa F\times_{\Ycal}\Spa F,(y_1,y_2)}\Spa F$.
	By definition, $\Xcal'$ is equipped with a classical point $\tilde{x}$ induced by
	\[(x_1,x_2)\times \id \colon \Spa F\to (\Xcal\times_{\Ycal}(\Spa F\times_{\Ycal}\Spa F)) \times_{\Spa F\times_{\Ycal}\Spa F,(y_1,y_2)}\Spa F.\]
	For $i=1,2$, let $p_i\colon \Xcal'\to \Xcal\times_{\Ycal}\Spa F$ be the composition of the natural projection $\Xcal'\to \Xcal\times_{\Ycal}(\Spa F\times_{\Ycal}\Spa F)$ and the $i$-th projection 
	\[\Xcal\times_{\Ycal}(\Spa F\times_{\Ycal}\Spa F)= (\Xcal\times_{\Ycal}\Spa F)\times_{\Xcal}(\Xcal\times_{\Ycal}\Spa F)\to \Xcal\times_{\Ycal}\Spa F.\]
	Similarly, for $i=1,2$, let $q_i\colon \Spa F\to \Spa F$ be the composition of $(y_1,y_2)\colon \Spa F\to \Spa F\times_{\Ycal}\Spa F$ and the $i$-th projection $\Spa F\times_{\Ycal}\Spa F\to \Spa F$.
	By construction, for $i=1,2$, we have a Cartesian diagram 
	\[\xymatrix{
		\Xcal\times_{\Ycal}\Spa F\ar[r] &\Spa F\\
		\Xcal'\ar[u]^-{p_i}\ar[r] & \Spa F\ar[u]_-{q_i}.
	}\]
	Since $q_i$ is an isomorphism, $p_i\colon \Xcal'\to \Xcal\times_{\Ycal}\Spa F$ is also an isomorphism.
	By construction, we have, for $i=1,2$, $p_i(\tilde{x})=x_i$.
	Therefore, we obtain
	\[\dim_{x_1}(\Xcal\times_{\Ycal}\Spa F)=\dim_{\tilde{x}} \Xcal'=\dim_{x_2}(\Xcal\times_{\Ycal}\Spa F),\]
	which proves the claim.
\end{proof}

\begin{lemma}\label{lem:fiber dimension,quasi-analytic space}
	Let $f\colon X\to Y$ be a morphism of quasi-analytic spaces, and let $x\in X^{\cl}$ be a classical point.
	Then we have $\dim_x X\leq \dim_x X/Y +\dim_{f(x)} Y$.
	If $f$ is smooth, then we have $\dim_x X= \dim_x X/Y +\dim_{f(x)} Y$.
\end{lemma}
\begin{proof}
	The claim follows from Lemma \ref{lem:fiber dimension} by passing to \'{e}tale charts.
\end{proof}

\begin{corollary}\label{cor:fiber dimension transitive,rigid analytic Artin stack}
	Let $\Xcal$, $\Ycal$, and $\Zcal$ be rigid analytic Artin stacks.
	Let $f\colon \Xcal\to \Ycal$ and $g\colon \Ycal\to \Zcal$ be morphisms representable by quasi-analytic spaces, and let $x\in \Xcal^{\cl}$ be a classical point.
	Assume that $f$ is smooth.
	Then we have 
	\[\dim_x (g\circ f)=\dim_x f+\dim_{f(x)}g.\]
\end{corollary}

\begin{definition}
	Let $\Xcal$ be a rigid analytic Artin stack.
	For a classical point $x\in \Xcal^{\cl}$, we define the dimension $\dim_{x} \Xcal$ of $\Xcal$ at $x$ as follows:
	Choose a smooth morphism $f\colon U\to \Xcal$ from a rigid space $U$ and a classical point $y\in U^{\cl}$ such that $f(y)=x$.
	Then we define $\dim_x \Xcal=\dim_y U-\dim_y f$.
	Note that, by Lemma \ref{lem:fiber dimension,quasi-analytic space} and Corollary \ref{cor:fiber dimension transitive,rigid analytic Artin stack}, it does not depend on the choice of $f$ and $y$.
	The dimension $\dim \Xcal$ of $\Xcal$ is defined as $\sup_{x\in \Xcal^{\cl}}\dim_x \Xcal$.
\end{definition}

\begin{definition}
	Let $f\colon \Xcal \to \Ycal$ be a morphism of rigid analytic Artin stacks.
	\begin{enumerate}
	\item For a classical point $x\in \Xcal^{\cl}$, the relative dimension $\dim_x f$ (or $\dim_x \Xcal/\Ycal$) of $f$ at $x$ is defined as follows:
	Choose a representative $\Spa F\to \Ycal$ of the classical point $f(x)$.
	We also choose a classical point $x'$ of $\Xcal\times_{\Ycal} \Spa F$ whose image in $\Xcal$ is $x$.
	We define $\dim_x f$ as the dimension of $\Xcal\times_{\Ycal} \Spa F$ at $x'$.
	As in Lemma \ref{lem:independent of the lifting} it does not depend on the choices of $\Spa F\to \Ycal$ and $x'$.
	\item The relative dimension $\dim f$ (or $\dim \Xcal/\Ycal$) of $f$ is defined as $\sup_{x\in \Xcal^{\cl}} \dim_x f$.
	\item We say that $f$ is of \textit{pure dimension} $n$ if $\dim f=n$ and for every classical point $x\in \Xcal^{\cl}$, $\dim_x f=n$.
	\end{enumerate}
\end{definition}

\begin{remark}
	The dimension $\dim \Xcal$ could be a negative integer.
	For example, for a smooth rigid analytic group $G$, the dimension $\dim [\ast/G]$ of the classifying stack is $-\dim G$.
\end{remark}

\begin{lemma}\label{lem:stack fiber dimension}
	Let $f\colon \Xcal\to \Ycal$ be a morphism of rigid analytic Artin stacks, and let $x\in \Xcal^{\cl}$ be a classical point.
	Then we have $\dim_x \Xcal \leq \dim_x \Xcal/\Ycal +\dim_{f(x)} \Ycal$.
	In particular, we have $\dim \Xcal \leq \dim \Xcal/\Ycal +\dim \Ycal$.
\end{lemma}
\begin{proof}
	Let us consider the following diagram:
	\[
	\xymatrix{
		 & \Xcal\ar[r]^-{f} & \Ycal\\
	U\ar[ru]\ar[r]_-{h} & \Zcal \ar[r]_-{p_2}\ar[u]_-{p_1} &V\ar[u]_-{g},	
	}\]
	where $\Zcal=\Xcal\times_{\Ycal}V$, and $V\twoheadrightarrow \Ycal$ and $U\twoheadrightarrow \Zcal$ are smooth surjections from rigid spaces $V$ and $U$, respectively.
	Since the composition $p_1\circ h\colon U\twoheadrightarrow \Zcal \twoheadrightarrow \Xcal$ is a smooth surjection, we can choose a classical point $u\in U^{\cl}$ such that $p_1(h(u))=x$.
	Then by definition and Corollary \ref{cor:fiber dimension transitive,rigid analytic Artin stack}, we have
	\begin{align*}
		\dim_{f(x)} \Ycal&=\dim_{p_2(h(u))} V-\dim_{p_2(h(u))} V/\Ycal,\\
		\dim_{x} \Xcal&=\dim_u U-\dim_u U/\Xcal,\\
		\dim_{h(u)} \Zcal/\Xcal&=\dim_{p_2(h(u))} V/\Ycal,\\
		\dim_{h(u)} \Zcal/V&=\dim_{x} \Xcal/\Ycal,\\
		\dim_u U/\Xcal&= \dim_u U/\Zcal +\dim_{h(u)} \Zcal/\Xcal\\
		\dim_u U/V&= \dim_u U/\Zcal +\dim_{h(u)} \Zcal/V.\\
	\end{align*}
	By Lemma \ref{lem:fiber dimension,quasi-analytic space}, we have $\dim_u U\leq \dim_u U/V+\dim_{p_2(h(u))}V$.
	By combining this with the above equalities, we obtain the claim.
\end{proof}

\begin{lemma}\label{lem:smooth of pure dimension}
	Let $f\colon U\to \Xcal$ be a smooth morphism from a quasi-analytic space $U$ to a rigid analytic Artin stack $\Xcal$.
	Then there exists a unique decomposition $U=\bigsqcup_{n\geq 0} U_n$ into closed and open subspaces $U_n$ such that $f\vert_{U_n}$ is smooth of pure dimension $n$. 
\end{lemma}
\begin{proof}
	If both $U$ and $\Xcal$ are rigid spaces, then the claim is well-known.
	Assume that $\Xcal$ is a rigid space.
	Choose an \'{e}tale surjection $g\colon V\twoheadrightarrow U$ from a rigid space $V$.
	Then there exists a unique decomposition $V=\bigsqcup_{n\geq 0} V_n$ into closed and open subspaces $V_n$ such that $(f\circ g)\vert_{V_n}$ is smooth of pure dimension $n$. 
	By the uniqueness of the decomposition $V=\bigsqcup_{n\geq 0} V_n$, it descends to a decomposition $U=\bigsqcup_{n\geq 0} U_n$, and by construction, $f\vert_{U_n}$ is smooth of pure dimension $n$. 

	Finally, we consider the general case.
	Choose a smooth surjection $V\twoheadrightarrow \Xcal$ from a rigid space $V$, and set $f'\colon U'=U\times_{\Xcal} V\to V$.
	Then there exists a unique decomposition $U'=\bigsqcup_{n\geq 0} U'_n$ into closed and open subspaces $U'_n$ such that $f'\vert_{U'_n}$ is smooth of pure dimension $n$. 
	By the uniqueness of the decomposition $U'=\bigsqcup_{n\geq 0} U'_n$, it descends to a decomposition $U=\bigsqcup_{n\geq 0} U_n$, and by construction, $f\vert_{U_n}$ is smooth of pure dimension $n$. 
\end{proof}

\begin{definition}
	Let $\Xcal$ be a rigid analytic Artin stack.
	A \textit{locally Zariski closed covering} (resp.\ \textit{locally Zariski closed partition}) of $\Xcal$ is a family $\{\Xcal_i\}_{i\in I}$ of locally Zariski closed substacks $\Xcal_i\subset \Xcal$ such that for some smooth surjection $f\colon U\twoheadrightarrow \Xcal$ from a rigid space $U$, $\{f^{-1}(\Xcal_i)\}_{i\in I}$ is a locally Zariski closed covering (resp.\ locally Zariski closed partition) of $U$ in the sense of Definition \ref{defn:locally Zariski closed partition, rigid space}.
	In this case, we write $\Xcal\sim \bigcup_{i\in I}\Xcal_i$ (resp.\ $\Xcal\sim \bigsqcup_{i\in I}\Xcal_i$).
\end{definition}

\begin{remark}
	Let $\Xcal$ be a rigid analytic Artin stack, and let $\{\Xcal_i\}_{i\in I}$ be a locally Zariski closed covering (resp.\ locally Zariski closed partition) of $\Xcal$.
	Then for any morphism $f\colon U\to \Xcal$ from a rigid space $U$, $\{f^{-1}(\Xcal_i)\}_{i\in I}$ is a locally Zariski closed covering (resp.\ locally Zariski closed partition) of $U$.
\end{remark}

\begin{example}\label{ex:Fitting ideal locally Zariski closed partition}
	Let $\Xcal$ be a rigid analytic Artin stack, and let $M$ be a coherent sheaf on $\Xcal$.
	Let $Z_k\subset \Xcal$ be the Zariski closed substack defined by the $k$-th Fitting ideal $\Fit_k(M)$, and set $Z_k^{\circ}=Z_{k-1}\setminus Z_{k}$. 
	Then $\bigsqcup_{k\geq 0}Z_k^{\circ}$ is a locally Zariski closed partition of $\Xcal$.
\end{example}

\begin{example}\label{ex:locally Zariski closed partition rigid stack}
	Let $\Xcal$ be a rigid analytic Artin stack, and let $n$ be an integer.
	Let $\{Z_{\kbf}\}_{\kbf\in \Zbb_{\geq 0}^n}$ be a family of Zariski closed substacks of $\Xcal$ such that $Z_0=\Xcal$ and $Z_{\kbf_2}\subseteq Z_{\kbf_1}$ whenever $\kbf_1\leq \kbf_2$, where the order on $\Zbb_{\geq 0}^n$ is the componentwise order.
	Assume that for any sequence $\kbf_1<\kbf_2<\cdots$, $\bigcap_m Z_{\kbf_m}=\emptyset$.
	For $\kbf \in \Zbb_{\geq 0}^n$, set $Z_{\kbf}^{\circ}=Z_{\kbf}\setminus \bigcup_{\kbf'>\kbf}Z_{\kbf'}$.
	Then $\{Z_{\kbf}^{\circ}\}_{\kbf\in \Zbb_{\geq 0}^n}$ is a locally Zariski closed covering of $\Xcal$.
	As in Example \ref{ex:locally Zariski closed partition rigid space}, it suffices to check the condition $\bigcap_{m\geq 1}Z_{\kbf_m}=\emptyset$ on classical points.
\end{example}

\begin{lemma}\label{lem:locally Zariski closed partition dimension}
	Let $\Xcal$ be a rigid analytic Artin stack, and let $\Xcal\sim \bigcup_{i\in I}\Xcal_i$ be a locally Zariski closed covering.
	Then we have $\dim \Xcal=\sup_{i\in I} \dim \Xcal_i$.
\end{lemma}
\begin{proof}
	Let $x\in \Xcal^{\cl}$ be a classical point.
	By Lemma \ref{lem:smooth of pure dimension}, we can take a smooth morphism $f\colon U\to \Xcal$ of pure dimension $n$ from a rigid space $U$ such that $x\in f(U)$.
	Since a smooth morphism is an open map (\cite[Proposition 1.7.8]{Hub96}), $f(U)$ is an open substack of $\Xcal$.
	By definition, we have
	\[\dim f(U)=\dim U-n,\quad \dim (f(U)\cap \Xcal_i)=\dim f^{-1}(\Xcal_i)-n.\]
	By Lemma \ref{lem:locally Zariski closed partition dimension rigid space}, we have
	\[\dim U=\sup_{i\in I} \dim f^{-1}(\Xcal_i).\]
	Therefore, we obtain 
	\[\dim f(U) =\sup_{i\in I}\dim (f(U)\cap \Xcal_i).\]
	By varying $x\in \Xcal^{\cl}$, we obtain the desired equality 
	\[\dim \Xcal=\sup_{i\in I} \dim \Xcal_i.\]
\end{proof}

\begin{corollary}\label{cor:dimension stack closed and complement}
	Let $\Xcal$ be a rigid analytic Artin stack, and let $Z\subset \Xcal$ be a Zariski closed substack.
	Then we have $\dim \Xcal=\max\{\dim Z, \dim (\Xcal\setminus Z)\}$.
\end{corollary}

\begin{corollary}\label{cor:coherent sheaf to vector bundle}
	Let $\Xcal$ be a rigid analytic Artin stack, and let $M_1,\ldots, M_n$ be coherent sheaves on $\Xcal$.
	Then there exists a reduced substack $Z\subset \Xcal$ satisfying the following properties:
	\begin{enumerate}
		\item $Z$ is a rigid analytic Artin stack;
		\item $\dim Z=\dim \Xcal$;
		\item For every $i=1,\ldots, n$, the restriction of $M_i$ to $Z$ is a vector bundle.
	\end{enumerate}
\end{corollary}
\begin{proof}
	This follows immediately from Example \ref{ex:Fitting ideal locally Zariski closed partition} and Lemma \ref{lem:locally Zariski closed partition dimension}.
\end{proof}

For later use, we record the following easy lemma.
\begin{lemma}\label{lem:section dimension}
Let $\Xcal$ and $\Ycal$ be rigid analytic Artin stacks, and let $f\colon \Xcal \to \Ycal$ be a morphism representable by rigid spaces.
Assume that there exists a section of $f$, that is, there exists a morphism $g\colon \Ycal \to \Xcal$ such that $f\circ g=\id_{\Ycal}$.
Then we have $\dim \Xcal\geq \dim \Ycal$.
\end{lemma}
\begin{proof}
It suffices to show that for a classical point $y\in \Ycal$, $\dim_{g(y)} \Xcal\geq \dim_{y} \Ycal$.
Choose a smooth morphism $Y\to \Ycal$ of pure dimension $n$ whose image contains $y$.
By base changing along this morphism, we obtain 
\[Y\xrightarrow{g} X=\Xcal\times_{\Ycal} Y \xrightarrow{f} Y,\]
where we use the same notation for the base changes of $f$ and $g$.
Then, by assumption, $X$ is a rigid space.
We set $x=g(y)\in X^{\cl}$.
Then $f$ and $g$ induce morphisms
\[\Ocal_{Y,y}\to \Ocal_{X,x}\to \Ocal_{Y,y}\]
whose composition is the identity.
In particular, the local ring $\Ocal_{Y,y}$ is a quotient of $\Ocal_{X,x}$, and thus, we obtain $\dim \Ocal_{X,x}\geq \dim \Ocal_{Y,y}$.
From this, we obtain
\[\dim_{x}\Xcal =\dim_x X-n \geq \dim_y Y-n=\dim_y \Ycal.\]
\end{proof}

\begin{remark}
	This lemma also holds for morphisms representable by quasi-analytic spaces.
	However, morphisms representable by rigid spaces are enough for our purpose, so we omit the details (this is not hard).
	Note that the lemma does not hold for a general morphism of rigid analytic Artin stacks. 
	For example, consider $\ast/\Gbb_m^{\an}\to \ast.$
\end{remark}
%%%%%%%%%%%%%%%%%%%%%%%%%%%%%%%%%%%%%%%%%%%%%%%%%%%%%%%%%%%%%%%%%%%%%

\section{A criterion for representability}
Let $K$ be a complete non-archimedean field.
We regard a rigid space $S$ as the \'{e}tale sheaf on $\Aff_K^{\op}$ defined by $\Aff_K^{\op} \ni \Spa A\mapsto S(\Spa A)$.
For an \'{e}tale sheaf $\Fcal$ on $\Aff_K^{\op}$ and a rigid space $S$, we write $\Fcal(S)=\Hom(S,\Fcal)$.

In this section, we provide a criterion for a given \'{e}tale stack $\Xcal$ to be a rigid analytic Artin stack.
Our arguments are based on \cite{Hall17Openness,War17}.

\subsection{Extensions and obstructions}

\begin{lemma}
	Let $A$, $B$, and $C$ be affinoid $K$-algebras.
	Let $i\colon \Spa A\to \Spa B$ be a square-zero closed immersion, and $f\colon \Spa A\to \Spa C$ be a finite morphism.
	Then the topological $K$-algebra $B\times_A C$ is an affinoid $K$-algebra, and the diagram
	\[
	\xymatrix{
		\Spa A \ar[r]\ar[d]& \Spa C\ar[d]\\
		\Spa B\ar[r] &\Spa (B\times_A C)
	}
	\]
	is coCartesian in the category of rigid spaces.
\end{lemma}
\begin{proof}
	This is a special case of \cite[Proposition 2.10.4]{War17}.
\end{proof}

We define the following classes of morphisms of affinoid rigid spaces:
\begin{itemize}
	\item $\Nil$: nilpotent closed immersions,
	\item $\Cl$: closed immersions,
	\item $\rNil$: morphisms $\Spa A\to \Spa B$ such that there exists $(\Spa A_0\to \Spa A)\in \Nil$ with the composition $(\Spa A_0\to \Spa A\to \Spa B)\in \Nil$
	\item $\rCl$: morphisms $\Spa A\to \Spa B$ such that there exists $(\Spa A_0\to \Spa A)\in \Nil$ with the composition $(\Spa A_0\to \Spa A\to \Spa B)\in \Cl$,
	\item $\Fin$: finite morphisms.
\end{itemize}

\begin{remark}
	By definition, we have 
	\[
	\xymatrix{
		\rNil \ar@{^{(}->}[r] & \rCl \\
		\Nil \ar@{^{(}->}[r] \ar@{^{(}->}[u] & \Cl \ar@{^{(}->}[u].
	}
	\]
\end{remark}

\begin{remark}
	Any morphism $\Spa A\to \Spa B$ in $\rCl$ is finite (\cite[Lemma 4.1.1]{War17}).
\end{remark}

\begin{example}
	Let $A$ be an affinoid $K$-algebra, and let $x\in (\Spa A)^{\cl}$ be a classical point.
	For any finite local $K$-algebra $B$, if a morphism $\Spa B\to \Spa A$ with image $x$ induces an isomorphism of residue fields, it belongs to $\rCl$. 
\end{example}

\begin{definition}
	Let $\mathbf{P}$ be one of $\Nil, \Cl, \rNil, \rCl, \Fin$.
	An \'{e}tale stack $\Xcal$ on $\Aff_K^{\op}$ is \textit{$\mathbf{P}$-homogeneous} if for any nilpotent closed immersion $\Spa A\to \Spa B$ and any morphism $\Spa A\to \Spa C$ in $\mathbf{P}$, the natural morphism
	\[\Xcal(B\times_A C)\to \Xcal(B)\times_{\Xcal(A)}\Xcal(C)\]
	is an equivalence of groupoids.
\end{definition}

\begin{definition}\label{def:exan}
Let $\Xcal$ be an \'{e}tale stack on $\Aff_K^{\op}$.
\begin{enumerate}
	\item An $\Xcal$-extension is a square-zero closed immersion $i\colon \Spa A\to \Spa B$ of affinoid rigid spaces over $\Xcal$.
	Let $\Exanbb_{\Xcal}$ denote the category of $\Xcal$-extensions.
	\item There is a functor
	\[\Exanbb_{\Xcal} \to (\Aff_K^{\op})_{/\Xcal} ;\; (i\colon \Spa A\to \Spa B)\mapsto \Spa A.\]
	For $\Spa A\in (\Aff_K^{\op})_{/\Xcal}$, let $\Exanbb_{\Xcal}(A)$ denote the fiber category over $\Spa A$.
	\item For $\Spa A\in (\Aff_K^{\op})_{/\Xcal}$, there is a functor 
	\[\Exanbb_{\Xcal}(A) \to \Coh(A)^{\op} ;\; (i\colon \Spa A\to \Spa B)\mapsto \Ker (B\to A),\]
	where we regard $\Ker (B\to A)$ as a finitely generated $A$-module (note that $B\to A$ is a square-zero extension).
	For $I\in \Coh(A)$, let $\Exanbb_{\Xcal}(A,I)$ denote the fiber category over $I$.
	Note that $\Exanbb_{\Xcal}(A,I)$ is a groupoid.
	As in \cite[Proposition 2.4]{Hall17Openness}, the functor $\Exanbb_{\Xcal}(A) \to \Coh(A)^{\op}$ is a fibration in groupoids.
	Therefore, we obtain a functor 
	\[\Coh(A)\to \Grpd ;\; I\mapsto \Exanbb_{\Xcal}(A,I).\]
\end{enumerate}
\end{definition}

\begin{remark}
	For a morphism $I\to J$ in $\Coh(A)$, the functor $\Exanbb_{\Xcal}(A,I)\to \Exanbb_{\Xcal}(A,J)$ is described as follows: Let $(i\colon \Spa A\to \Spa B) \in \Exanbb_{\Xcal}(A,I)$. 
	By taking the pushout, we obtain a commutative diagram with exact rows:
	\[
	\xymatrix{
		0\ar[r] & I \ar[r]\ar[d] &B \ar[r]\ar[d] & A \ar[r]\ar@{=}[d] &0\\
		0\ar[r] & J\ar[r] &B' \ar[r] &A \ar[r]&0.
	}
	\]
	This diagram yields a square-zero closed immersion $\Spa A\to \Spa B'$ over $\Xcal$ with the identification $\Ker(B'\to A)=J$. 
\end{remark}

\begin{lemma}\label{lem:Exan presheaf}
	Let $\Xcal$ be an \'{e}tale stack on $\Aff_K^{\op}$.
	Let $\Spa A'\to \Spa A$ be an open immersion of affinoid rigid spaces.
	Then, for a morphism $\Spa A\to \Xcal$ and $I\in \Coh(A)$, there is a natural functor of groupoids
	\[\Exanbb_{\Xcal}(A,I) \to \Exanbb_{\Xcal}(A',A'\otimes_A I).\]
\end{lemma}
\begin{proof}
	Let $(i\colon \Spa A\to \Spa B)\in \Exanbb_{\Xcal}(A,I)$. 
	Since $i$ is a homeomorphism, the open subspace $\Spa A'\subset \Spa A$ corresponds to an open subspace $U\subset \Spa B$.
	By \cite[Proposition 7.3.14]{GR03}, $U$ is an affinoid open subspace.
	We set $B'=\Ocal_{B}(U)$.
	Then we have
	\[\Ker(B'\to A')\cong A'\otimes_A I.\]
	Therefore, we get $(i'\colon \Spa A'\to \Spa B')\in \Exanbb_{\Xcal}(A',A'\otimes_A I)$.
	This construction yields a functor 
	\[\Exanbb_{\Xcal}(A,I) \to \Exanbb_{\Xcal}(A',A'\otimes_A I).\]
\end{proof}

\begin{definition}
	Let $\Xcal$ be an \'{e}tale stack on $\Aff_K^{\op}$.
\begin{enumerate}
	\item A deformation associated to $\Xcal$ is a pair $(i\colon \Spa A\to \Spa B, r\colon \Spa B\to \Spa A)$, where $i$ is a square-zero closed immersion of affinoid rigid spaces over $\Xcal$ and $r$ is a retract of $i$ over $K$.
	Let $\Defbb_{\Xcal}$ denote the category of deformations associated to $\Xcal$.
	\item There is a forgetful functor
	\begin{align*}
		\Defbb_{\Xcal} \to \Exanbb_{\Xcal} ;\;&(i\colon \Spa A\to \Spa B,r\colon \Spa B\to \Spa A)\mapsto \\ &(i\colon \Spa A\to \Spa B).
	\end{align*}
	For $\Spa A\in (\Aff_K^{\op})_{/\Xcal}$, let $\Defbb_{\Xcal}(A)$ denote the fiber category over $\Exanbb_{\Xcal}(A)$, and for $I\in \Coh(A)$, let $\Defbb_{\Xcal}(A,I)$ denote the fiber category over $\Exanbb_{\Xcal}(A,I)$.
	Note that $\Defbb_{\Xcal}(A,I)$ is a groupoid.
	As in Definition \ref{def:exan}, the functor $\Defbb_{\Xcal}(A)\to \Coh(A)^{\op}$ is a fibration in groupoids.
	Therefore, we get a functor 
	\[\Coh(A)\to \Grpd ;\; I\mapsto \Defbb_{\Xcal}(A,I).\]
\end{enumerate}
\end{definition}

\begin{remark}\label{remark:tangent space}
	Let $\Xcal$ be an \'{e}tale stack on $\Aff_K^{\op}$.
	For $\alpha \colon \Spa A\to \Xcal$ and $I\in \Coh(A)$, the groupoid $\Defbb_{\Xcal}(A,I)$ is equivalent to the fiber of $\Xcal(A\oplus I)\to \Xcal(A)$ at $\alpha$, where $A\oplus I$ is the trivial square-zero extension, and $\Xcal(A\oplus I)\to \Xcal(A)$ is induced by the natural square-zero closed immersion $\Spa A\to \Spa(A\oplus I)$.
\end{remark}

As in Lemma \ref{lem:Exan presheaf}, we get the following lemma.

\begin{lemma}\label{lem:Def presheaf}
	Let $\Xcal$ be an \'{e}tale stack on $\Aff_K^{\op}$.
	Let $\Spa A'\to \Spa A$ be an open immersion of affinoid rigid spaces.
	Then, for a morphism $\Spa A\to \Xcal$ and $I\in \Coh(A)$, there is a natural functor of groupoids
	\[\Defbb_{\Xcal}(A,I) \to \Defbb_{\Xcal}(A',A'\otimes_A I).\]
\end{lemma}

\begin{lemma}
	Let $\Xcal$ be a $\Nil$-homogeneous \'{e}tale stack on $\Aff_K^{\op}$.
	Then, for $\Spa A\in (\Aff_K^{\op})_{/\Xcal}$, the functors
	\[\Exanbb_{\Xcal}(A,-) \colon \Coh(A)\to \Grpd, \quad \Defbb_{\Xcal}(A,-) \colon \Coh(A)\to \Grpd\]
	preserve finite products.
	In particular, for each $I\in \Coh(A)$, the groupoids $\Exanbb_{\Xcal}(A,I)$ and $\Defbb_{\Xcal}(A,I)$ admit natural Picard category structures (for the definition of Picard categories, see \cite[Expos\'{e} XVIII, D\'{e}finition 1.4.2]{SGA4-3}).
\end{lemma}
\begin{proof}
	It can be proved in the same way as in \cite[Proposition 2.4]{Hall17Openness}.
\end{proof}

\begin{definition}
	Let $\Xcal$ be a $\Nil$-homogeneous \'{e}tale stack on $\Aff_K^{\op}$.
	For $\Spa A\in (\Aff_K^{\op})_{/\Xcal}$ and $I \in \Coh(A)$, define the abelian groups
	\begin{align*}
	\Der_{\Xcal}(A,I)&=\Aut_{\Exanbb_{\Xcal}(A,I)}(\iota_{A,I}),\\
	\Aut_{\Xcal}(A,I)&=\Aut_{\Defbb_{\Xcal}(A,I)}(\iota_{A,I}),
	\end{align*}
	where $\iota_{A,I}$ is the unit object of the respective categories.
	Moreover, define $\Exan_{\Xcal}(A,I)$ (resp.\ $\Def_{\Xcal}(A,I)$) as the group of isomorphism classes of objects in $\Exanbb_{\Xcal}(A,I)$ (resp.\ $\Defbb_{\Xcal}(A,I)$).
	Then we obtain additive functors
	\[\Der_{\Xcal}(A,-), \Exan_{\Xcal}(A,-), \Aut_{\Xcal}(A,-), \Def_{\Xcal}(A,-)\colon \Coh(A)\to \Ab.\]
\end{definition}

\begin{remark}
	For $a\in A$, the morphism $\times a \colon I \to I$ induces functors $\Exanbb_{\Xcal}(A,I)\to \Exanbb_{\Xcal}(A,I)$ and $\Defbb_{\Xcal}(A,I)\to \Defbb_{\Xcal}(A,I)$.
	These define an $A$-module structure on $\Der_{\Xcal}(A,I)$, $\Exan_{\Xcal}(A,I)$, $\Aut_{\Xcal}(A,I)$, and $\Def_{\Xcal}(A,I)$.
	Therefore, we can enhance the above functors to 
	\[\Der_{\Xcal}(A,-), \Exan_{\Xcal}(A,-), \Aut_{\Xcal}(A,-), \Def_{\Xcal}(A,-)\colon \Coh(A)\to \Mod_A.\]
	Moreover, for $I\in \Coh(A)$, Lemma \ref{lem:Exan presheaf} yields a presheaf of $\Ocal_{\Spa A}$-modules \[\Dcal er_{\Xcal}(A,I)\colon (\Spa A'\subset \Spa A)\mapsto \Der_{\Xcal}(A',A'\otimes_A I).\]
	Similarly, by Lemmas \ref{lem:Exan presheaf} and \ref{lem:Def presheaf}, we have presheaves of $\Ocal_{\Spa A}$-modules $\Ecal xan_{\Xcal}(A,I)$, $\Acal ut_{\Xcal}(A,I)$, and $\Dcal ef_{\Xcal}(A,I)$.
	These define additive functors 
	\[\Dcal er_{\Xcal}(A,-), \Ecal xan_{\Xcal}(A,-), \Acal ut_{\Xcal}(A,-), \Dcal ef_{\Xcal}(A,-)\colon \Coh(A)\to \PShv(\Mod_{\Ocal_{\Spa A}}).\]
\end{remark}

\begin{example}\label{ex:cotangent complex deformation}
	When $\Xcal=\Spa K$, we simply write $\Der(A,I)$ and $\Exan(A,I)$.
	Then there are natural isomorphisms (of $A$-modules)
	\begin{align*}
		\Der(A,I)&\cong \Ext^0_A(\Lbb_{A},I)\cong \Hom_A(\Omega_A,I),\\
		\Exan(A,I)&\cong \Ext^1_A(\Lbb_A,I),
	\end{align*}
	where $\Lbb_A$ is the analytic cotangent complex of the morphism $\Spa A\to \Spa K$ defined in \cite[Definition 7.2.32]{GR03}.
\end{example}

\begin{lemma}\label{lem:Exan half-exact}
	Let $\Xcal$ be a $\rNil$-homogeneous \'{e}tale stack on $\Aff_K^{\op}$, and let $\Spa A \to \Xcal$ be a morphism from an affinoid rigid space $\Spa A$.
	Then for an exact sequence of finitely generated $A$-modules 
	\[0\to I_1\to I_2\to I_3\to 0,\]
	there is a natural exact sequence
	\[
	\xymatrix{
    0 \ar[r] & \Der_{\Xcal}(A,I_1) \ar[r] & \Der_{\Xcal}(A,I_2) \ar[r] & \Der_{\Xcal}(A,I_3) \ar@(r,l)[dll]\\
             & \Exan_{\Xcal}(A,I_1) \ar[r] & \Exan_{\Xcal}(A,I_2) \ar[r] & \Exan_{\Xcal}(A,I_3).
}
\]
In particular, additive functors 
\[\Dcal er_{\Xcal}(A,-), \Ecal xan_{\Xcal}(A,-)\colon \Coh(A)\to \PShv(\Mod_{\Ocal_{\Spa A}})\]
are half-exact.
\end{lemma}
\begin{proof}
	It can be proved in the same way as in \cite[Corollary 2.5]{Hall17Openness}.
\end{proof}

\begin{lemma}\label{lem:Def Exan exact sequence}
	Let $\Xcal$ be a $\Nil$-homogeneous \'{e}tale stack on $\Aff_K^{\op}$.
	Then, for $\Spa A\in (\Aff_K^{\op})_{/\Xcal}$ and $I\in \Coh(A)$, there is a natural exact sequence of $A$-modules
	\[
	\xymatrix{
    0 \ar[r] & \Aut_{\Xcal}(A,I) \ar[r] & \Der_{\Xcal}(A,I) \ar[r] & \Der(A,I) \ar@(r,l)[dll]\\
             & \Def_{\Xcal}(A,I) \ar[r] & \Exan_{\Xcal}(A,I) \ar[r] & \Exan(A,I).
}
\]
\end{lemma}
\begin{proof}
	It can be proved in the same way as in \cite[Proposition 6.7]{Hall17Openness}.
\end{proof}

The following definition is an analogue of \cite[Section 5]{Hall17Openness}.

\begin{definition}
	For $A\in \Aff_K$, an additive functor $F\colon \Coh(A)\to \PShv(\Mod_{\Ocal_{\Spa A}})$ is said to be \textit{coherent} if it satisfies the following conditions:
	\begin{enumerate}
		\item For any $I\in \Coh(A)$, the presheaf $F(I)$ is a coherent sheaf.
		\item There exist a morphism $M\to N$ in $\Coh(A)$ and $\alpha \in F(M)$ such that for any $I$, the sequence
		\[\Hom_A(N,I) \to \Hom_A(M,I)\to F(I)\to 0\]
		is right exact, where $\Hom_A(M,I)\to F(I)$ is induced by $\alpha$.
	\end{enumerate}
\end{definition}

\begin{example}\label{ex:Der Exan coherent}
	By Example \ref{ex:cotangent complex deformation}, for $A\in \Aff_K$, the functors 
	\[\Dcal er(A,-), \Ecal xan(A,-)\colon \Coh(A)\to \PShv(\Mod_{\Ocal_{\Spa A}})\]
	are coherent, where we note that $\Lbb_A$ is pseudo-coherent by \cite[Theorem 7.1.33, Remark 7.2.45]{GR03}.
\end{example}

\begin{lemma}\label{lem:coherent 5-terms}
	Let $A\in \Aff_K$ be an affinoid $K$-algebra.
	For each $i=1,\ldots,5$, let $F_i\colon \Coh(A)\to \PShv(\Mod_{\Ocal_{\Spa A}})$ be an additive functor fitting into an exact sequence
	\[F_1\to F_2\to F_3\to F_4\to F_5.\]
	If $F_1$, $F_2$, $F_4$, $F_5$ are coherent, then $F_3$ is also coherent.
\end{lemma}
\begin{proof}
	For each $I\in \Coh(A)$, there is an exact sequence of presheaves of $\Ocal_{\Spa A}$-modules
	\[F_1(I)\to F_2(I)\to F_3(I)\to F_4(I)\to F_5(I).\]
	Since the presheaves $F_1(I)$, $F_2(I)$, $F_4(I)$, $F_5(I)$ are coherent sheaves, the presheaf $F_3(I)$ is also a coherent sheaf.
	The rest follows from \cite[Theorem 1.1]{Hartshorne90}.
\end{proof}

\begin{definition}
	Let $\Xcal$ be a $\Nil$-homogeneous \'{e}tale stack on $\Aff_K^{\op}$.
	We say that $\Xcal$ has \textit{coherent automorphisms} (resp.\ \textit{coherent deformations}) if for any $\Spa A\in (\Aff_K^{\op})_{/\Xcal}$, the functor \[\Acal ut_{\Xcal}(A,-)\colon \Coh(A)\to \PShv(\Mod_{\Ocal_{\Spa A}})\] (resp.\ $\Dcal ef_{\Xcal}(A,-)\colon \Coh(A)\to \PShv(\Mod_{\Ocal_{\Spa A}})$) is coherent.

	We say that $\Xcal$ has \textit{coherent obstructions} if for any $\Spa A\in (\Aff_K^{\op})_{/\Xcal}$, there exist a coherent functor $\Ocal bs_{\Xcal}(A,-)\colon \Coh(A)\to \PShv(\Mod_{\Ocal_{\Spa A}})$ and a functorial injection 
	\[\Coker(\Ecal xan_{\Xcal}(A,-)\to \Ecal xan(A,-))\hookrightarrow \Ocal bs_{\Xcal}(A,-).\]
\end{definition}

\subsection{Formal smoothness and formal versality}

\begin{definition}\label{def:formally smooth and versal}
	Let $\Xcal$ be an \'{e}tale stack on $\Aff_K^{\op}$, let $f\colon T \to \Xcal$ be a morphism from a rigid space $T$, and let $t\in T^{\cl}$ be a classical point. 
	\begin{enumerate}
		\item The morphism $f$ is \textit{formally smooth at $t$} if for any $\Xcal$-extension $i\colon \Spa A\to \Spa B$ with $A,B$ finite local $K$-algebras, and any morphism $\alpha \colon \Spa A\to T$ over $\Xcal$ satisfying $\alpha(v) = t$ for the closed point $v \in \Spa A$, $\alpha$ lifts to a morphism $\Spa B\to T$ over $\Xcal$.
		\item The morphism $f$ is \textit{formally versal at $t$} if for any $\Xcal$-extension $i\colon \Spa A\to \Spa B$ with $A,B$ finite local $K$-algebras, and any morphism $\alpha \colon \Spa A\to T$ over $\Xcal$ satisfying $\alpha(v) = t$ for the closed point $v \in \Spa A$ and $k(t) \cong k(v)$, $\alpha$ lifts to a morphism $\Spa B\to T$ over $\Xcal$.
		\item The morphism $f$ is \textit{formally smooth} (resp. \textit{formally versal}) if for any $t\in T^{\cl}$, $f$ is formally smooth (resp. formally versal) at $t$.
	\end{enumerate}
\end{definition}

\begin{remark}
	When $\Xcal$ is a rigid space, the above definition of formal smoothness coincides with Definition \ref{def:formally smooth}.
\end{remark}

\begin{remark}
	If $f\colon T \to \Xcal$ is formally smooth, then for any $t\in T^{\cl}$, $f$ is formally versal at $t$.
\end{remark}

\begin{lemma}\label{lem:formally smooth basechange}
	Let $\Xcal$ be an \'{e}tale stack on $\Aff_K^{\op}$.
	Assume that the diagonal morphism $\Delta \colon \Xcal\to \Xcal\times \Xcal$ is representable by quasi-analytic spaces.
	Let $f\colon T\to \Xcal$ and $g\colon Y\to \Xcal$ be morphisms from rigid spaces.
	Let $Z\twoheadrightarrow T\times_{\Xcal} Y$ be an \'{e}tale surjection from a rigid space $Z$, and let $h\colon Z\to T$ be the composition of $Z\to T\times_{\Xcal} Y$ and $T\times_{\Xcal} Y\to T$.
	For a classical point $z\in Z^{\cl}$, if the morphism $f\colon T\to \Xcal$ is formally smooth at $h(z)$, then the composition $Z\to T\times_{\Xcal} Y\to Y$ is formally smooth at $z$.
\end{lemma}
\begin{proof}
	Let $i\colon \Spa A\to \Spa B$ with $A,B$ finite local $K$-algebras be a $Y$-extension, and let $\alpha \colon \Spa A\to Z$ be a morphism over $Y$ satisfying $\alpha(v) = z$ for the closed point $v \in \Spa A$.
	To summarize, we have the following commutative diagram:
	\[
	\xymatrix{
		\Spa A\ar[d]_{\alpha}\ar[rr]& &\Spa B\ar[d]\\
		Z \ar[r]\ar[rd]_{h}& T\times_{\Xcal} Y \ar[r]\ar[d] & Y\ar[d]\\
		 & T \ar[r]_{f} &\Xcal.
	}
	\]
	Since the morphism $f\colon T\to \Xcal$ is formally smooth at $h(z)$, there exists a morphism $\Spa B\to T$ making the diagram above commutative.
	By the universal property of the fiber product, this morphism uniquely lifts to $\Spa B\to T\times_{\Xcal} Y$.
	Since $Z\to T\times_{\Xcal}Y$ is \'{e}tale, it lifts to $\Spa B\to Z$.
\end{proof}

\begin{proposition}\label{prop:formally smooth implies smooth}
	Let $\Xcal$ be an \'{e}tale stack on $\Aff_K^{\op}$, let $T$ be a rigid space, and let $f\colon T \to \Xcal$ be a morphism representable by quasi-analytic spaces.
	Then $f$ is smooth if and only if $f$ is formally smooth.
\end{proposition}
\begin{proof}
	The ``only if'' direction easily follows from Theorem \ref{thm:smooth implies formally smooth}.
	Let us prove the ``if'' direction.
	We fix any morphism $g\colon Y\to \Xcal$ from a rigid space $Y$ and an \'{e}tale surjection $Z\twoheadrightarrow T\times_{\Xcal} Y$ from a rigid space $Z$.
	We need to show that the composition $Z\to T\times_{\Xcal} Y\to Y$ is smooth.
	By Lemma \ref{lem:formally smooth basechange}, $Z\to T\times_{\Xcal} Y\to Y$ is formally smooth.
	Therefore, the claim follows from Theorem \ref{thm:formally smooth implies smooth rigid space}.
\end{proof}

For later use, we note the following lemma.

\begin{lemma}\label{lem:Hensel's lemma}
	Let $\Xcal$ be an \'{e}tale stack on $\Aff_K^{\op}$, and let $f\colon T \to \Xcal$ be a morphism from a rigid space $T$.
	Assume that $K$ is a perfect field.
	Then for a classical point $t\in T^{\cl}$, $f$ is formally versal at $t$ if and only if for any $\Xcal$-extension $i\colon \Spa A\to \Spa B$ such that $A$ and $B$ are finite local $k(t)$-algebras with residue field $k(t)$, and any morphism $\alpha \colon \Spa A\to T$ over $\Xcal$ with image $t$ inducing the identity on $k(t)$, $\alpha$ lifts to a morphism $\Spa B\to T$ over $\Xcal$.
\end{lemma}
\begin{proof}
	The ``only if'' direction is clear.
	Let us prove the ``if'' direction.
	Let $i\colon \Spa A\to \Spa B$ be an $\Xcal$-extension of finite local $K$-algebras, and let $\alpha \colon \Spa A\to T$ be a morphism over $\Xcal$ with image $t$ such that $k(t) \cong k(v)$ for the closed point $v \in \Spa A$.
	Since $K$ is a perfect field, the finite extension $k(t)/K$ is separable.
    By Hensel's lemma, $A$ and $B$ are uniquely endowed with $k(t)$-algebra structures compatible with their residue fields.
    Then the claim follows from the assumption.
\end{proof}

\begin{proposition}\label{prop:formally versal Exan}
Let $\Xcal$ be an $\rCl$-homogeneous \'{e}tale stack on $\Aff_K^{\op}$, and let $f\colon \Spa A \to \Xcal$ be a morphism from an affinoid rigid space $\Spa A$.
Then, for a classical point $t\in (\Spa A)^{\cl}$, $f$ is formally versal at $t$ if and only if \[\Exan_{\Xcal}(A,k(t))=0.\]
\end{proposition}
\begin{proof}
	This can be proved in the same way as in \cite[Lemma 4.3]{Hall17Openness}.\footnote{To prove the ``only if'' direction, it suffices to assume $\Cl$-homogeneity.}
\end{proof}

\begin{proposition}\label{prop:formally smooth Exan}
	Let $\Xcal$ be a $\Fin$-homogeneous \'{e}tale stack on $\Aff_K^{\op}$. 
	Let $f\colon \Spa A \to \Xcal$ be a morphism from an affinoid rigid space $\Spa A$, and let $t\in (\Spa A)^{\cl}$ be a classical point.
	If $\Exan_{\Xcal}(A,I)=0$ for all $I\in \Coh(A)$ supported on $\{t\}$, then $f$ is formally smooth at $t$.
\end{proposition}
\begin{proof}
	This can be proved in the same way as in \cite[Lemma 4.3]{Hall17Openness}.
\end{proof}

\begin{corollary}\label{cor:formally versal iff formally smooth}
	Let $\Xcal$ be a $\Fin$-homogeneous \'{e}tale stack on $\Aff_K^{\op}$. 
	Let $f\colon \Spa A \to \Xcal$ be a morphism from an affinoid rigid space $\Spa A$, and let $t\in (\Spa A)^{\cl}$ be a classical point.
	Then $f$ is formally versal at $t$ if and only if $f$ is formally smooth at $t$.
\end{corollary}
\begin{proof}
	By Propositions \ref{prop:formally versal Exan} and \ref{prop:formally smooth Exan}, it suffices to show that if $\Exan_{\Xcal}(A,k(t))=0$, then for any $I\in \Coh(A)$ supported on $\{t\}$, $\Exan_{\Xcal}(A,I)=0$.
	Since $\Exan_{\Xcal}(A,-)$ is half-exact by Lemma \ref{lem:Exan half-exact}, this follows from the standard reduction argument.
\end{proof}

\begin{theorem}\label{thm:openness of versality}
	Let $\Xcal$ be an $\rCl$-homogeneous \'{e}tale stack on $\Aff_K^{\op}$.
	Assume that $\Xcal$ has coherent deformations and obstructions.
	Then, for a morphism $f\colon \Spa A \to \Xcal$ from an affinoid rigid space $\Spa A$, the subset 
	\[(\Spa A)^{\mathrm{fv}}=\{t\in (\Spa A)^{\cl}\mid \text{$f$ is formally versal at $t$}\}\]
	is Zariski open in $(\Spa A)^{\cl}$.
	In other words, there exists a unique Zariski open subspace $U\subset \Spa A$ such that $U\cap (\Spa A)^{\cl}=(\Spa A)^{\mathrm{fv}}.$
\end{theorem}
\begin{proof}
	It suffices to show that for any $t\in (\Spa A)^{\mathrm{fv}}$, there exists a Zariski open neighborhood $U\subset \Spa A$ of $t$ such that $U^{\cl}\subset (\Spa A)^{\mathrm{fv}}$.
	Let $t\in (\Spa A)^{\mathrm{fv}}$, and $\mfrak\subset A$ be the maximal ideal corresponding to $t$.
	By assumption, Example \ref{ex:Der Exan coherent}, Lemma \ref{lem:Def Exan exact sequence}, and Lemma \ref{lem:coherent 5-terms}, the additive functor $\Ecal xan_{\Xcal}(A,-)\colon \Coh(A)\to \PShv(\Mod_{\Ocal_{\Spa A}})$ is coherent.
	Moreover, it is half-exact by Lemma \ref{lem:Exan half-exact}.
	As in \cite[Remark 3.3]{Hartshorne90}, by taking filtered colimits, we extend the functor 
	\[\Exan_{\Xcal}(A,-)\colon \Coh(A)\to \Coh(A)\] 
	to the functor 
	\[\Exan_{\Xcal}(A,-)\colon \Mod_A \to \Mod_A.\]
	By Proposition \ref{prop:formally versal Exan}, we have $\Exan_{\Xcal}(A,A/\mfrak)=0$.
	Therefore, by \cite[Proposition 3.4, Corollary 3.10]{Hartshorne90}, we obtain \[\Exan_{\Xcal}(A,A)_{\mfrak}\cong\Exan_{\Xcal}(A,A_{\mfrak})=0.\]
	Since $\Exan_{\Xcal}(A,A)$ is a finitely generated $A$-module, there exists $f\in A\setminus\mfrak$ such that 
	\[\Exan_{\Xcal}(A,A_f)\cong\Exan_{\Xcal}(A,A)_f=0.\]
	Then for any maximal ideal $\nfrak$ of $A_f$, we have  
	\[\Exan_{\Xcal}(A,A_{\nfrak})\cong\Exan_{\Xcal}(A,A_f)_{\nfrak}=0.\]
	Therefore, by \cite[Proposition 3.9]{Hartshorne90}, we obtain $\Exan_{\Xcal}(A,A/\nfrak)=0.$
	In other words, for any classical point $t'\in (\Spa A\setminus \{f=0\})^{\cl}$, we have $\Exan_{\Xcal}(A,k(t'))=0.$
\end{proof}

\begin{corollary}\label{cor:formally versal smooth local on the source}
	Let $\Xcal$ be an \'{e}tale stack on $\Aff_K^{\op}$.
	Assume the following:
	\begin{itemize}
		\item The stack $\Xcal$ is $\Fin$-homogeneous.
		\item The diagonal morphism $\Delta \colon \Xcal\to \Xcal\times \Xcal$ is representable by quasi-analytic spaces.
		\item The stack $\Xcal$ has coherent deformations and obstructions.
	\end{itemize}
	Let $f\colon T\to \Xcal$ be a morphism from a rigid space $T$, and let $g\colon S\to T$ be a smooth morphism of rigid spaces.
	Then, for a classical point $s\in S^{\cl}$, if $f\circ g$ is formally versal at $s$, then $f$ is formally versal at $g(s)$.
\end{corollary}
\begin{proof}
	We may assume that $S$ and $T$ are affinoid rigid spaces.
	We set $T=\Spa A$ and $S=\Spa B$.
	Moreover, by Theorem \ref{thm:openness of versality}, we may assume that $f\circ g$ is formally versal after shrinking $\Spa B$.
	Since the morphism $\Spa B\to \Spa A$ is smooth, it is an open morphism by \cite[Proposition 1.7.8]{Hub96}.
	Therefore, we may assume that $\Spa B\to \Spa A$ is surjective.
	By Proposition \ref{prop:formally smooth implies smooth} and Corollary \ref{cor:formally versal iff formally smooth}, the morphism $\Spa B\to \Xcal$ is smooth.
	Since smoothness is smooth local on the source (\cite[Proposition 3.14]{OhShimizu25}), we find that $\Spa A\to \Xcal$ is also smooth by the standard argument.
	Therefore, it is formally smooth (and hence formally versal) by Proposition \ref{prop:formally smooth implies smooth}.
\end{proof}

\subsection{A criterion for representability}
\begin{theorem}\label{thm:criterion for representablity}
	Let $\Xcal$ be an \'{e}tale stack on $\Aff_K^{\op}$.
	Assume the following:
	\begin{enumerate}
		\item The stack $\Xcal$ is $\Fin$-homogeneous.
		\item The diagonal morphism $\Delta \colon \Xcal\to \Xcal\times \Xcal$ is representable by quasi-analytic spaces.
		\item The stack $\Xcal$ has coherent deformations and obstructions.
		\item For any morphism $\Spa R\to \Xcal$ from an affinoid rigid space $\Spa R$ and any (not necessarily classical) point $x\in \Spa R$, there exist the following data:
		\begin{itemize}
			\item A smooth affinoid rigid space $\Spa A$ over $K$ with a point $y\in \Spa A$, and a morphism $\Spa A\to \Spa R$ with $y\mapsto x$ and $k(y)\cong k(x)$.
			Moreover, the natural morphism $A\to k(y)$ is injective.
			\item A morphism $f\colon \Spa B \to \Spa A$ with a section $i\colon \Spa A\to \Spa B$.
			\item A morphism $g\colon \Spa B\to \Xcal$ fitting into a commutative diagram
			\[
			\xymatrix{
				\Spa A \ar[r] \ar[d]_{i} &\Spa R \ar[d]\\
				\Spa B\ar[r]^g & \Xcal.
			}
			\]
			Moreover, for any classical point $t\in (\Spa A)^{\cl}$, the induced morphism
			\[g\colon f^{-1}(t)\to \Xcal,\] where $f^{-1}(t)=\Spa (B\otimes_A k(t))$, is formally versal at $i(t)$.
		\end{itemize}
	\end{enumerate}
	Then $\Xcal$ is a rigid analytic Artin stack.
\end{theorem}
\begin{proof}
	By Theorem \ref{thm:openness of versality}, we can take the largest Zariski open subspace $U\subset \Spa B$ such that $g\vert_U \colon U\to \Xcal$ is formally versal.

	\begin{description}
		\item[Claim] $i(y)\in U$.
	\end{description}
	
	Assuming the claim for the moment, we complete the proof.
	By Corollary \ref{cor:formally versal iff formally smooth} and Proposition \ref{prop:formally smooth implies smooth}, $U\to \Xcal$ is smooth.
	Let $X$ be the disjoint union of such spaces $U$ obtained as above by varying the morphisms $\Spa R\to \Xcal$ and the points $x\in \Spa R$ (note that the index set is small up to isomorphism).
	We need to show that $X\to \Xcal$ is smooth and surjective.
	It suffices to show that for any $\Spa R\to \Xcal$ and any $x\in \Spa R$, the image of $U\times_{\Xcal}\Spa R \subset X\times_{\Xcal} \Spa R\to \Spa R$ contains $x$.
	We have the following commutative diagram:
	\[
	\xymatrix{
	i^{-1}(U)\ar[d] & \\
	U\times_{\Xcal} \Spa R \ar[r]\ar[d] &\Spa R\ar[d]\\
	U\ar[r] & \Xcal.}
	\]
	By the claim, the image of $i^{-1}(U)\to U\times_{\Xcal} \Spa R \to \Spa R$ contains $x$.

	We now prove the claim.
	First, we assume that $x$ is a classical point.
	In this case, we have $\Spa A=\Spa k(x)$.
	Therefore, we find that $g\colon \Spa B=f^{-1}(x)\to \Xcal$ is formally versal at $i(x)$.
	% Hence, we get the claim by Proposition \ref{prop:formally smooth implies smooth}.

	Next, we treat the case of a general $x$.
	Since $i^{-1}(U)$ is a Zariski open subspace of $\Spa A$ and $A\to k(x)$ is injective, it suffices to show that $i^{-1}(U)$ is non-empty.
	In particular, we may replace $\Spa A$ with any non-empty affinoid open subspace (which does not necessarily contain $x$).
	By \cite[Theorem 2.21]{BhattHansen22}, we may assume that $f$ is flat after shrinking $\Spa A$ (as $A$ is smooth, hence geometrically reduced).

	We prove that for any classical point $t\in (\Spa A)^{\cl}$, $g$ is smooth at $i(t)$.
	We already know that there exists a smooth morphism $\Spa R_1\to \Xcal$ whose image contains the image of $t$ under $\Spa A\to \Xcal$.
	Let $S\twoheadrightarrow \Spa B\times_{\Xcal} \Spa R_1$ be an \'{e}tale surjection from a rigid space $S$.
	Then we have the following diagram:
	\[
	\xymatrix{
		S\ar[r]\ar[rd]_-{h}&\Spa B\times_{\Xcal} \Spa R_1 \ar[r]\ar[d]& \Spa R_1\ar[d]\\
		&\Spa B\ar[r]^g & \Xcal.
	}
	\]
	Let $h\colon S\to \Spa B$ be the composition of $S\to \Spa B\times_{\Xcal} \Spa R_1$ and $\Spa B\times_{\Xcal} \Spa R_1\to \Spa B$.
	Since $\Spa R_1\to \Xcal$ is smooth, $\Spa B\times_{\Xcal} \Spa R_1\to \Spa B$ is also smooth.
	Therefore, $h\colon S\to \Spa B$ is also smooth as $S\to \Spa B\times_{\Xcal} \Spa R_1$ is \'{e}tale.
	By construction, there is a classical point $s\in S^{\cl}$ such that $h(s)=i(t)$.
	By Corollary \ref{cor:formally versal smooth local on the source}, it suffices to show that the composition $S\xrightarrow{h}\Spa B\to \Xcal$ is formally versal at $s$.
	Since $\Spa R_1\to \Xcal$ is smooth (and hence formally versal by Proposition \ref{prop:formally smooth implies smooth}), it is enough to show that the composition 
	\[S\to \Spa B\times_{\Xcal} \Spa R_1\to \Spa R_1\]
	is formally versal at $s$.
	Since $h\colon S\to \Spa B$ is smooth, it is also flat.
	Hence, $f\circ h\colon S\to \Spa A$ is also flat.
	By applying Lemma \ref{lem:formally smooth basechange} to the following diagram
	\[
	\xymatrix{
		(f\circ h)^{-1}(t)\ar[r]& f^{-1}(t)\times_{\Xcal} \Spa R_1 \ar[r]\ar[d]& \Spa R_1\ar[d]\\
		&f^{-1}(t)\ar[r]^g & \Xcal,
	}
	\]
	the morphism $(f\circ h)^{-1}(t)\to \Spa R_1$ is formally smooth (and hence smooth by Theorem \ref{thm:formally smooth implies smooth rigid space}) at $s$.
	Therefore, by Corollary \ref{cor:smooth fiber and base}, the composition \[S\to \Spa B\times_{\Xcal} \Spa R_1\to \Spa R_1\]
	is smooth (and thus formally versal) at $s$.
\end{proof}

%%%%%%%%%%%%%%%%%%%%%%%%%%%%%%%%%%%%%%%%%%%%%%%%%%%%%%%%%%%%%%%%%%%%%%%%%%%%%

\section{Representability of the analytic Emerton-Gee stack}
Throughout this section, ``$\Gamma$-equivariant vector bundle/torsor'' means ``continuous $\Gamma$-equivariant vector bundle/torsor'' for a profinite group $\Gamma$.
\subsection{Representability for $\GL_n$}

Let $K/\Qbb_p$ be a finite extension, and let $\Cbb_p$ be the completion of an algebraic closure of $K$.
Let $G_K$ be the absolute Galois group of $K$.
We fix a pseudo-uniformizer $p^{\flat}=(p,p^{1/p},\ldots)\in \Ocal_{\Cbb_p}^{\flat}$.
Let $Y_{\Cbb_p}=\Spa W(\Ocal_{\Cbb_p}^{\flat})\setminus \{p[p^{\flat}]=0\}$ and let $X_{\Cbb_p}=Y_{\Cbb_p}/\varphi^{\Zbb}$ be the Fargues-Fontaine curve.
By \cite[Theorem 3.2]{Ked16}, $X_{\Cbb_p}$ is a locally noetherian analytic adic space.
In particular, for any $A\in \Aff_{\Qbb_p}$, the adic space $X_{\Cbb_p}\times \Spa A$ is well-defined.
We set $X_{\Cbb_p,A}=X_{\Cbb_p}\times \Spa A$.

\begin{definition}
	For a positive integer $n$, let $\LS_{\GL_n(K),\FF}$ denote the \'{e}tale stack 
	\[\Aff_{\Qbb_p}\ni A\mapsto \{\text{$G_K$-equivariant vector bundles on $X_{\Cbb_p,A}$ of rank $n$}\},\]
	which we refer to as the \textit{analytic Emerton--Gee stack}. 
\end{definition}

For deformation-theoretic purposes, it is convenient to use the following
variant.
Let $K_{\infty}$ be the cyclotomic $\Zbb_p$-extension of $K$,\footnote{In \cite{EGH23, Mikami24}, the extension $K(\mu_{p^{\infty}})$ and the group $\Gal(K(\mu_{p^{\infty}})/K)$ are used. In this section, we use the cyclotomic $\Zbb_p$-extension because $\Gamma_K$ always admits a topological generator.} and set $\Gamma_K=\Gal(K_{\infty}/K)$.
Let $\gamma\in \Gamma_K$ be a topological generator of $\Gamma_K$.
We fix a pseudo-uniformizer $\pi^{\flat}\in \Ocal_{\hat{K}_{\infty}}^{\flat}$ such that $\pi^{\flat\sharp}/p\in \Ocal_{\hat{K}_{\infty}}^{\times}$.
Let $Y_{K_{\infty}}=\Spa W(\Ocal_{\hat{K}_{\infty}}^{\flat})\setminus \{p[\pi^{\flat}]=0\}$ and let $X_{K_{\infty}}=Y_{K_{\infty}}/\varphi^{\Zbb}$ be the Fargues-Fontaine curve.
By \cite[Theorem 3.2]{Ked16}, $X_{K_{\infty}}$ is also a locally noetherian analytic adic space.
In particular, for any $A\in \Aff_{\Qbb_p}$, the adic space $X_{K_{\infty}}\times \Spa A$ is well-defined.
We set $X_{K_{\infty},A}=X_{K_{\infty}}\times \Spa A$.
By \cite[Theorems 3.12 and 3.35]{Mikami24} or \cite[the discussion after Theorem 5.1.5]{EGH23}, the following categories are canonically equivalent as tensor categories:
\begin{enumerate}
	\item The category of $G_K$-equivariant vector bundles on $X_{\Cbb_p,A}$.
	\item The category of $\Gamma_K$-equivariant vector bundles on $X_{K_{\infty},A}$.
	\item The category of $(\varphi,\Gal(K(\mu_{p^{\infty}})/K))$-modules over the Robba ring $\Rcal_A$ defined in \cite[Definition 2.2.12]{KPX14}.
\end{enumerate}

In this section, we regard $\LS_{\GL_n(K),\FF}$ as the \'{e}tale stack 
\[\Aff_{\Qbb_p}\ni A\mapsto \{\text{$\Gamma_K$-equivariant vector bundles on $X_{K_{\infty},A}$ of rank $n$}\},\]
and in the next section, we regard it as the \'{e}tale stack 
\[\Aff_{\Qbb_p}\ni A\mapsto \{\text{$G_K$-equivariant vector bundles on $X_{\Cbb_p,A}$ of rank $n$}\}.\]

The goal of this subsection is to prove that $\LS_{\GL_n(K),\FF}$ is a rigid analytic Artin stack.
For $[r,s]\subset (0,\infty)$ such that $r,s\in \Qbb$, we set 
\begin{align*}
	\tilde{B}_{K_{\infty}}^{[r,s]}&=W(\Ocal_{\hat{K}_{\infty}}^{\flat})\left\langle \frac{p^s}{[\pi^{\flat}]},\frac{[\pi^{\flat}]}{p^r}\right\rangle[1/p].
\end{align*}  
This ring admits a natural $\Gamma_K$-action, and there is a Frobenius automorphism
\[\varphi \colon \tilde{B}_{K_{\infty}}^{[pr,ps]}\xrightarrow{\sim} \tilde{B}_{K_{\infty}}^{[r,s]}.\]
Note that the Fargues-Fontaine curve $X_{K_{\infty}}$ is obtained from $\Spa \tilde{B}_{K_{\infty}}^{[1,p]}$ by gluing $\Spa \tilde{B}_{K_{\infty}}^{[1,1]}$ and $\Spa \tilde{B}_{K_{\infty}}^{[p,p]}$ along $\varphi$.
Let $I=[r,s]\subset (0,\infty)$ such that $r,s\in \Qbb$.
For an affinoid $\Qbb_p$-algebra $A$, we set 
\[\tilde{B}_{K_{\infty},A}^{I}=\tilde{B}_{K_{\infty}}^{I}\hotimes_{\Qbb_p}A.\]
For a $\Gamma_K$-equivariant vector bundle $V$ on $X_{K_{\infty},A}$, let $V^I$ denote the pullback of $V$ under the morphism $\Spa (\tilde{B}_{K_{\infty},A}^{I}) \to X_{K_{\infty},A}$, and we regard $V^I$ as a finite projective $\tilde{B}_{K_{\infty},A}^{I}$-module with a continuous $\Gamma_K$-action.
By construction, there is a $\Gamma_K$-equivariant $\varphi$-semilinear isomorphism
\[\varphi\colon V^{[p,p]}\xrightarrow{\sim} V^{[1,1]}.\]
Conversely, a pair $(V^{[1,p]},\varphi)$ consisting of
\begin{itemize}
	\item a finite projective $\tilde{B}_{K_{\infty},A}^{[1,p]}$-module $V^{[1,p]}$ with a continuous semilinear $\Gamma_K$-action, and
	\item a $\Gamma_K$-equivariant $\varphi$-semilinear isomorphism
	\[\varphi\colon V^{[1,p]}\otimes_{\tilde{B}_{K_{\infty},A}^{[1,p]}}\tilde{B}_{K_{\infty},A}^{[p,p]}\xrightarrow{\sim}V^{[1,p]}\otimes_{\tilde{B}_{K_{\infty},A}^{[1,p]}}\tilde{B}_{K_{\infty},A}^{[1,1]}\]
\end{itemize}
defines a $\Gamma_K$-equivariant vector bundle $V$ on $X_{K_{\infty},A}$.

\begin{example}
	Let $V$ be a $\Gamma_K$-equivariant vector bundle on $X_{K_{\infty},A}$. Then $\underline{\End}(V)$ also admits a natural $\Gamma_K$-equivariant structure.
	Let us explain the description of the corresponding pair $(\End(V)^{[1,p]},\varphi)$.
	First, we have
	\[\underline{\End}(V)^{I}=\Hom_{\tilde{B}_{K_{\infty},A}^{I}}(V^{I},V^{I}).\]
	The $\Gamma_K$-action is given by
	\begin{align*}
		\gamma'\colon \Hom(V^{[1,p]},V^{[1,p]})\to \Hom(V^{[1,p]},V^{[1,p]}) ;\; f\mapsto \gamma'\circ f\circ\gamma'^{-1}
	\end{align*}
	for each $\gamma'\in \Gamma_K$.
	The Frobenius is given by
	\begin{align*}
		\varphi\colon \Hom(V^{[p,p]},V^{[p,p]})\to \Hom(V^{[1,1]},V^{[1,1]}) ;\; f\mapsto \varphi\circ f\circ\varphi^{-1}.
	\end{align*}
\end{example}

Let us check the four conditions in Theorem \ref{thm:criterion for representablity}.

\begin{lemma}\label{lem:EG stack Fin-homogeneous}
	The \'{e}tale stack $\LS_{\GL_n(K),\FF}$ is $\Fin$-homogeneous.
\end{lemma}
\begin{proof}
	This easily follows from the description of $\Gamma_K$-equivariant vector bundles in terms of the pairs $(V^{[1,p]},\varphi)$.
\end{proof}

The following lemma states that the Fargues-Fontaine curve $X_{K_{\infty}}$ behaves as a proper rigid space. 

\begin{lemma}\label{lem:FF curve closed map}
	Let $S$ be a rigid space over $\Qbb_p$.
	Then the projection $X_{K_{\infty}}\times S\to S$ is a closed map.
\end{lemma}
\begin{proof}
	We may assume that $S=\Spa A$ is an affinoid rigid space.
	Since $X_{K_{\infty},A}\to \Spa A$ is a spectral morphism of spectral spaces, it suffices to show that this morphism is specializing.
	Let $K$ be a complete non-archimedean field over $\Qbb_p$, and $K_1^+\subset K_2^+$ be open bounded valuation rings of $K$.
	Consider the following diagram:
	\[
	\xymatrix{
		X_{K_{\infty},A}\ar[r]& \Spa A\\
		\Spa (K,K_2^+) \ar[r] \ar[u]&\Spa (K,K_1^+)\ar[u]\ar@{-->}[lu].
	}
	\]
	It suffices to show that there exists a morphism $\Spa (K,K_1^+)\to X_{K_{\infty},A}$ (represented by the dotted arrow) making the above diagram commute.
	We choose a lift $\Spa (K,K_2^+)\to \Spa (\tilde{B}_{K_{\infty},A}^{[r,s]},\tilde{B}_{K_{\infty},A}^{[r,s]+})$ of $\Spa (K,K_2^+)\to X_{K_{\infty},A}$ for some $[r,s]\subset (0,\infty)$.
	Choose $0<r'<r\leq s <s'$.
	Then the image of $\tilde{B}_{K_{\infty},A}^{[r',s']+}$ under the morphism $\tilde{B}_{K_{\infty},A}^{[r',s']}\to \tilde{B}_{K_{\infty},A}^{[r,s]}$ is contained in the smallest ring of integral elements of $\tilde{B}_{K_{\infty},A}^{[r,s]}$.
	Therefore, the composition $\tilde{B}_{K_{\infty},A}^{[r',s']}\to \tilde{B}_{K_{\infty},A}^{[r,s]}\to K$ induces the morphism $\Spa (K,K_1^+)\to \Spa (\tilde{B}_{K_{\infty},A}^{[r',s']},\tilde{B}_{K_{\infty},A}^{[r',s']+})$.
	Then the composition
	\[\Spa (K,K_1^+)\to \Spa (\tilde{B}_{K_{\infty},A}^{[r',s']},\tilde{B}_{K_{\infty},A}^{[r',s']+})\to X_{K_{\infty},A}\]
	is the desired morphism.
\end{proof}

\begin{definition}
	Let $S$ be a rigid space over $\Qbb_p$, and let $V$ and $W$ be $\Gamma_K$-equivariant vector bundles on $X_{K_{\infty}}\times S$.
	We define an \'{e}tale sheaf $\intHom_{\Gamma_K}(V,W)$ over $S$ by the assignment
	\[(\Spa A\to S) \mapsto \Hom_{X_{K_{\infty},A},\Gamma_K}(V_A,W_A),\]
	where $V_A$ and $W_A$ are the pullbacks of $V$ and $W$ to $X_{K_{\infty},A}$, respectively.
	Let $\underline{\Surj}_{\Gamma_K}(V,W) \subset \intHom_{\Gamma_K}(V,W)$ (resp.\ $\underline{\Isom}_{\Gamma_K}(V,W)\subset \intHom_{\Gamma_K}(V,W)$) denote the \'{e}tale subsheaf consisting of surjective morphisms (resp.\ isomorphisms) $V_A\to W_A$.
\end{definition}

\begin{lemma}[{\cite[Theorems 4.4.3 and 4.4.5]{KPX14}}]\label{lem:finiteness of (phi, Gamma) cohomology}
	Let $A$ be an affinoid $\Qbb_p$-algebra, and let $V$ be a $\Gamma_K$-equivariant vector bundle on $X_{K_{\infty},A}$.
	Then $R\Gamma(\Gamma_K,R\Gamma(X_{K_{\infty},A},V))$ is quasi-isomorphic to a perfect complex of $A$-modules with Tor-amplitude in $[0,2]$.\footnote{For the comparison between $R\Gamma(\Gamma_K,R\Gamma(X_{K_{\infty},A},V))$ and the cohomology of $(\varphi,\Gal(K(\mu_{p^{\infty}})/K))$-modules over the Robba ring $\Rcal_A$, see \cite[Section 4]{Mikami24}.}
	Moreover, for an affinoid $A$-algebra $B$, the natural morphism
	\[R\Gamma(\Gamma_K,R\Gamma(X_{K_{\infty},A},V))\otimes^{\Lbb}_A B\to R\Gamma(\Gamma_K,R\Gamma(X_{K_{\infty},B},V_B))\]
	is an equivalence, where $V_B$ is the pullback of $V$ to $X_{K_{\infty},B}$.
\end{lemma}

\begin{lemma}\label{lem:Hom space representable EG}
	Let $S$ be a rigid space over $\Qbb_p$, and let $V$ and $W$ be $\Gamma_K$-equivariant vector bundles on $X_{K_{\infty}}\times S$.
	The \'{e}tale sheaves $\intHom_{\Gamma_K}(V,W)$, $\underline{\Surj}_{\Gamma_K}(V,W)$, and $\underline{\Isom}_{\Gamma_K}(V,W)$ are representable by rigid spaces.
\end{lemma}
\begin{proof}
	Since the claim is analytic local on $S$, we may assume that $S$ is an affinoid rigid space.
	We set $S=\Spa A$.
	By Lemma \ref{lem:finiteness of (phi, Gamma) cohomology}, there exists a morphism $f\colon E_0\to E_1$ of finite projective $A$-modules equipped with a functorial isomorphism
	\[\Hom_{X_{K_{\infty},B},\Gamma_K}(V_B,W_B)\cong \Ker(E_0\otimes_A B\to E_1\otimes_A B)\]
	for each affinoid $A$-algebra $B$, where $V_B$ and $W_B$ are the pullbacks of $V$ and $W$ to $X_{K_{\infty},B}$, respectively.
	Therefore, $\intHom_{\Gamma_K}(V,W)$ is represented by $\Vbb(E_0)\times_{\Vbb(E_1),e}\Spa A$, where $\Vbb(E_0)$ and $\Vbb(E_1)$ are the geometric vector bundles associated to $E_0$ and $E_1$, respectively, and $e\colon \Spa A\to \Vbb(E_1)$ is the zero section.

	Next, let us prove the representability of $\underline{\Surj}_{\Gamma_K}(V,W)$ and $\underline{\Isom}_{\Gamma_K}(V,W)$.
	We may assume that $V$ and $W$ are of constant rank.
	If the ranks of $V$ and $W$ are different, $\underline{\Isom}_{\Gamma_K}(V,W)=\emptyset$.
	If the ranks of $V$ and $W$ are the same, any surjective morphism $V_B\to W_B$ is an isomorphism, and hence $\underline{\Surj}_{\Gamma_K}(V,W)=\underline{\Isom}_{\Gamma_K}(V,W)$.
	Therefore, it suffices to show that $\underline{\Surj}_{\Gamma_K}(V,W)$ is representable by a rigid space.
	To this end, we prove that the morphism $\underline{\Surj}_{\Gamma_K}(V,W) \to \intHom_{\Gamma_K}(V,W)$ is representable and an open immersion.
	Let $n$ be the rank of $W$.
	Since a morphism $V_B\to W_B$ is surjective if and only if $\wedge^n V_B\to \wedge^n W_B$ is surjective, we have the following Cartesian diagram:
	\[
	\xymatrix{
		\underline{\Surj}_{\Gamma_K}(V,W)\ar[r]\ar[d]& \underline{\Surj}_{\Gamma_K}(\bigwedge^n V, \bigwedge^n W)\ar[d]\\
		\intHom_{\Gamma_K}(V,W)\ar[r]& \intHom_{\Gamma_K}(\bigwedge^n V, \bigwedge^n W).
	}
	\]
	Therefore, we may reduce to the case where $W$ is of rank $1$.
    Furthermore, by twisting by $W^{-1}$, we may assume that $W=\Ocal_{X_{K_{\infty}},A}$.
	Let $\Spa B\to \intHom_{\Gamma_K}(V,\Ocal)$ be a morphism corresponding to a morphism $f \colon V_B\to \Ocal_{X_{K_{\infty},B}}$.
	We need to show that $\Spa B\times_{\intHom_{\Gamma_K}(V,\Ocal)}\underline{\Surj}_{\Gamma_K}(V,\Ocal)$ is an open subspace of $\Spa B$.
	Let $\Ical$ be the image of the morphism $f \colon V_B\to \Ocal_{X_{K_{\infty}},B}$, and let $Z\subset X_{K_{\infty},B}$ be the zero locus of $\Ical$.
	Since $X_{K_{\infty},B}\to \Spa B$ is a closed map by Lemma \ref{lem:FF curve closed map}, the image $Z'$ of $Z$ under this morphism is closed in $\Spa B$.
	Let $U\subset \Spa B$ be the complement of $Z'$.
	Then, by construction, $\Spa B\times_{\intHom_{\Gamma_K}(V,\Ocal)}\underline{\Surj}_{\Gamma_K}(V,\Ocal)$ is represented by $U$.
\end{proof}

\begin{corollary}[{\cite[Theorem 5.1.20]{EGH23}}]\label{cor:EG stack diagonal}
	The diagonal morphism 
	\[\Delta \colon \LS_{\GL_n(K),\FF} \to \LS_{\GL_n(K),\FF} \times \LS_{\GL_n(K),\FF}\]
	is representable by rigid spaces.
\end{corollary}

Next, we consider the deformation theory of $\Gamma_K$-equivariant vector bundles on $X_{K_{\infty},A}$.

\begin{definition}
	For a $\Gamma_K$-equivariant vector bundle $V$ on $X_{K_{\infty},A}$, we define the \textit{Herr complex} $\Ccal_{\FF,\Gamma_K}^{\bullet}(V)$ as
	\[\Ccal_{\FF,\Gamma_K}^{\bullet}(V)=(V^{[1,p]}\xrightarrow{(\varphi-1,\gamma-1)}V^{[1,1]}\oplus V^{[1,p]}\xrightarrow{(\gamma-1)\oplus(1-\varphi)}V^{[1,1]}).\]
	We write $H^i_{\FF,\Gamma_K}(V)=H^i(\Ccal_{\FF,\Gamma_K}^{\bullet}(V)).$
	The complex $\Ccal_{\FF,\Gamma_K}^{\bullet}(V)$ is quasi-isomorphic to $R\Gamma(\Gamma_K,R\Gamma(X_{K_{\infty},A},V))$.
	In particular, $\Ccal_{\FF,\Gamma_K}^{\bullet}(V)$ is quasi-isomorphic to a perfect complex of $A$-modules by Lemma \ref{lem:finiteness of (phi, Gamma) cohomology}.
\end{definition}

\begin{remark}\label{rem:flatness}
	Let $I$ be one of $[1,1],[1,p],[p,p]$.
	Then any finite projective $\tilde{B}_{K_{\infty},A}^{I}$-module $V^I$ is a flat $A$-module.
	Indeed, it suffices to show that $\tilde{B}_{K_{\infty},A}^{I}$ is flat over $A$, which can be proved in the same way as in the proof of \cite[Corollary 2.1.5]{KPX14}.
	In particular, for $M\in \Coh(A)$, the complex $\Ccal_{\FF,\Gamma_K}^{\bullet}(V)\otimes_A M$ is quasi-isomorphic to $\Ccal_{\FF,\Gamma_K}^{\bullet}(V)\otimes_A^{\Lbb} M$.
\end{remark}

The following lemma is standard:
\begin{lemma}\label{lem:deformation Herr complex}
	Let $\Spa A\to \LS_{\GL_n(K),\FF}$ be a morphism corresponding to a $\Gamma_K$-equivariant vector bundle $V$ on $X_{K_{\infty},A}$ of rank $n$.
	Then, for $I\in \Coh(A)$, there are natural isomorphisms
	\begin{align*}
		\Aut_{\LS_{\GL_n(K),\FF}}(A,I)\cong H^0(\Ccal_{\FF,\Gamma_K}^{\bullet}(\underline{\End}(V))\otimes_{A} I),\\
		\Def_{\LS_{\GL_n(K),\FF}}(A,I)\cong H^1(\Ccal_{\FF,\Gamma_K}^{\bullet}(\underline{\End}(V))\otimes_{A} I).
	\end{align*}
\end{lemma}
\begin{proof}
	Although this result is well-known, we provide the explicit construction of the isomorphism
	\[\Def_{\LS_{\GL_n(K),\FF}}(A,I)\cong H^1(\Ccal_{\FF,\Gamma_K}^{\bullet}(\underline{\End}(V))\otimes_{A} I)\]
	for later use.
	Let $A\oplus I$ be the trivial square-zero extension of $A$ by $I$.
	Let $x\in  H^1(\Ccal_{\FF,\Gamma_K}^{\bullet}(\underline{\End}(V))\otimes_{A} I)$, and choose a representative
	\begin{align*}
		(\alpha,\beta)\in \Ker((\End(V^{[1,1]})\otimes_A I)\oplus (\End(V^{[1,p]})\otimes_A I)\xrightarrow{(\gamma-1)\oplus(1-\varphi)}\End(V^{[1,1]})\otimes_A I).
	\end{align*}
	We regard $\alpha$ as a $\tilde{B}_{K_{\infty},A}^{[1,1]}$-linear morphism $V^{[1,1]}\to V^{[1,1]}\otimes_A I$ and $\beta$ as a $\tilde{B}_{K_{\infty},A}^{[1,p]}$-linear morphism $V^{[1,p]}\to V^{[1,p]}\otimes_A I$.
	We set $V_1^{[1,p]}=V^{[1,p]}\oplus (V^{[1,p]}\otimes_A I)$, which is a finite projective $\tilde{B}_{K_{\infty},A\oplus I}^{[1,p]}$-module.
	For $i=1,p$, we write 
	\[V_1^{[i,i]}=V^{[i,i]}\oplus (V^{[i,i]}\otimes_A I)=V_1^{[1,p]}\otimes_{\tilde{B}_{K_{\infty},A\oplus I}^{[1,p]}}\tilde{B}_{K_{\infty},A\oplus I}^{[i,i]}.\]
	We define $\gamma\colon V_1^{[1,p]}\to V_1^{[1,p]}$ by
	\[(x,y)\mapsto (\gamma(x), \beta(\gamma(x))+\gamma(y)).\]
	It uniquely extends to the continuous semilinear action of $\Gamma_K$ on $V_1^{[1,p]}$.
	Similarly, we define $\varphi\colon V_1^{[p,p]}\to V_1^{[1,1]}$ by
	\[(x,y)\mapsto (\varphi(x), \alpha(\varphi(x))+\varphi(y)).\]
	Since $(\alpha,\beta)$ is a cocycle, the above morphism $\varphi\colon V_1^{[p,p]}\to V_1^{[1,1]}$ commutes with $\gamma$ (and hence is $\Gamma_K$-equivariant). 
	Therefore, we get a $\Gamma_K$-equivariant vector bundle $V_1$ on $X_{K_{\infty},A\oplus I}$, which is a deformation of $V$.
	In other words, it gives an element of $\Def_{\LS_{\GL_n(K),\FF}}(A,I)$.
	Since this element does not depend on the choice of the representative $(\alpha,\beta)$ of $x$, we obtain a well-defined morphism of $A$-modules
	\[H^1(\Ccal_{\FF,\Gamma_K}^{\bullet}(\underline{\End}(V))\otimes_{A} I)\to \Def_{\LS_{\GL_n(K),\FF}}(A,I).\]
    By a standard argument, this can be shown to be an isomorphism.
\end{proof}

Next, we consider the obstruction theory of $\Gamma_K$-equivariant vector bundles on $X_{K_{\infty},A}$.

\begin{lemma}\label{lem:obstruction Herr complex}
	Let $\Spa A\to \LS_{\GL_n(K),\FF}$ be a morphism corresponding to a $\Gamma_K$-equivariant vector bundle $V$ on $X_{K_{\infty},A}$ of rank $n$.
	Then, for each $I\in \Coh(A)$, there is a natural exact sequence
	\[\Exan_{\LS_{\GL_n(K),\FF}}(A,I) \to \Exan(A,I)\to  H^2(\Ccal_{\FF,\Gamma_K}^{\bullet}(\underline{\End}(V))\otimes_{A} I).\]
\end{lemma}
\begin{proof}
	Although this result is also well-known, we provide the explicit construction of the morphism
	\[\Exan(A,I)\to  H^2(\Ccal_{\FF,\Gamma_K}^{\bullet}(\underline{\End}(V))\otimes_{A} I)\]
	for later use.
	Let $B\to A$ be a $\Qbb_p$-extension of $A$ by $I$.
	First, we note that the finite projective $\tilde{B}_{K_{\infty},A}^{[1,p]}$-module $V^{[1,p]}$ lifts to a finite projective $\tilde{B}_{K_{\infty},B}^{[1,p]}$-module $V_B^{[1,p]}$, uniquely up to isomorphism.
	For $i=1,p$, we set
	\[V_B^{[i,i]}=V_B^{[1,p]}\otimes_{\tilde{B}_{K_{\infty},B}^{[1,p]}}\tilde{B}_{K_{\infty},B}^{[i,i]}.\]
	We choose a $\gamma$-semilinear lift $\gamma_B\colon V_B^{[1,p]}\to V_B^{[1,p]}$ of $\gamma\colon  V^{[1,p]}\to V^{[1,p]}$ and a $\varphi$-semilinear lift $\varphi_B\colon V_B^{[p,p]}\to V_B^{[1,1]}$ of $\varphi\colon  V^{[p,p]}\to V^{[1,1]}$.
	Then we obtain a morphism 
	\[\varphi_B\gamma_B\varphi_B^{-1}\gamma_B^{-1}-1\colon V_B^{[1,1]}\to V_B^{[1,1]},\]
	where the morphisms $V_B^{[i,i]}\to V_B^{[i,i]}$ (for $i=1,p$) induced by $\gamma_B\colon V_B^{[1,p]}\to V_B^{[1,p]}$ are denoted by the same symbol $\gamma_B$.
	By construction, the morphism $\varphi_B\gamma_B\varphi_B^{-1}\gamma_B^{-1}-1$ is zero on $V^{[1,1]}\otimes_A I\subset V_B^{[1,1]}$, and the image of this morphism is contained in $V^{[1,1]}\otimes_A I\subset V_B^{[1,1]}$.
	Therefore, we obtain a morphism
	\[\varphi_B\gamma_B\varphi_B^{-1}\gamma_B^{-1}-1\colon V^{[1,1]}\to V^{[1,1]}\otimes_A I.\]
	This defines a cohomology class of $H^2(\Ccal_{\FF,\Gamma_K}^{\bullet}(\underline{\End}(V))\otimes_{A}I)$, and it does not depend on the choice of lifts $\gamma_B$ and $\varphi_B$.
	Therefore, we obtain a well-defined morphism of $A$-modules
	\[\Exan(A,I)\to  H^2(\Ccal_{\FF,\Gamma_K}^{\bullet}(\underline{\End}(V))\otimes_{A} I).\]
	By a standard argument, we can show that 
	\[\Exan_{\LS_{\GL_n(K),\FF}}(A,I) \to \Exan(A,I)\to  H^2(\Ccal_{\FF,\Gamma_K}^{\bullet}(\underline{\End}(V))\otimes_{A} I)\]
	is an exact sequence.
\end{proof}

\begin{lemma}\label{lem:coherent deformation and obstruction Herr complex}
	The \'{e}tale stack $\LS_{\GL_n(K),\FF}$ has coherent deformations and obstructions.
\end{lemma}
\begin{proof}
	Let $\Spa A\to \LS_{\GL_n(K),\FF}$ be a morphism corresponding to a $\Gamma_K$-equivariant vector bundle $V_A$ on $X_{K_{\infty},A}$ of rank $n$.
	By Lemmas \ref{lem:deformation Herr complex} and \ref{lem:obstruction Herr complex}, we need to show that for $i=1,2$, the functor
	\[\Coh(A)\ni I\mapsto H^i(\Ccal_{\FF,\Gamma_K}^{\bullet}(\underline{\End}(V_A))\otimes_{A} I)\]
	is coherent.
	By \cite[Theorem 4.4.2]{KPX14}, $\Ccal_{\FF,\Gamma_K}^{\bullet}(\underline{\End}(V_A))$ is quasi-isomorphic to a perfect complex of $A$-modules.
	Therefore, the functor
	\[I\mapsto H^i(\Ccal_{\FF,\Gamma_K}^{\bullet}(\underline{\End}(V_A))\otimes_{A} I)\] 
	is coherent by \cite[Example 2.3 (a)]{Hartshorne90} and Remark \ref{rem:flatness}.
\end{proof}

Finally, we check the condition (4) in Theorem \ref{thm:criterion for representablity}.

\begin{theorem}\label{thm:EG stack smooth chart}
	Let $\Spa R\to \LS_{\GL_n(K),\FF}$ be a morphism corresponding to a $\Gamma_K$-equivariant vector bundle $V_R$ on $X_{K_{\infty},R}$ of rank $n$. 
	For a (not necessarily classical) point $x\in \Spa R$, the following hold:
		\begin{itemize}
			\item There exists a locally Zariski closed subspace $\Spa A$ of $\Spa R$ such that
			\begin{itemize}
				\item $x\in \Spa A$,
				\item $\Spa A$ is smooth over $\Qbb_p$,
				\item the natural morphism $A\to k(x)$ is injective.
			\end{itemize}
			\item Let $r_i=\dim_{k(x)}(H^i(\Ccal_{\FF,\Gamma_K}^{\bullet}(\underline{\End}(V_R))\otimes_{R} k(x)))$ for $i=1,2$.
			Let $m$ be a sufficiently large integer.
			Then there exist 
			\[F_1,\ldots, F_{r_2}\in (X_1,\ldots,X_{r_1})^2A\left\langle \frac{X_1,\ldots,X_{r_1}}{p^m}\right\rangle\] 
			and a morphism $g\colon \Spa \left(A\left\langle \frac{X_1,\ldots,X_{r_1}}{p^m}\right\rangle/(F_1,\ldots,F_{r_2})\right)\to \LS_{\GL_n(K),\FF}$ fitting into a commutative diagram
			\[
			\xymatrix{
				\Spa A \ar[r] \ar[d]_-{i} &\Spa R \ar[d]\\
				\Spa \left(A\left\langle \frac{X_1,\ldots,X_{r_1}}{p^m}\right\rangle/(F_1,\ldots,F_{r_2})\right)\ar[r]^-g & \LS_{\GL_n(K),\FF},
			}
			\]
			where $i\colon \Spa A\to \Spa \left(A\left\langle \frac{X_1,\ldots,X_{r_1}}{p^m}\right\rangle/(F_1,\ldots,F_{r_2})\right)$ is a closed immersion defined by $X_1=\cdots=X_{r_1}=0$.
			Moreover, for any classical point $t\in (\Spa A)^{\cl}$, the induced morphism
			\[g\colon \Spa \left(k(t)\left\langle \frac{X_1,\ldots,X_{r_1}}{p^m}\right\rangle/(F_1,\ldots,F_{r_2})\right)\to \LS_{\GL_n(K),\FF}\] 
			is formally versal at $i(t)$.
		\end{itemize}
\end{theorem}
\begin{proof}
	We set $A=R/\supp(x).$
	Since $\Spa A\to \Spa \Qbb_p$ is smooth over a non-empty Zariski open subset $U\subset \Spa A$ and $x\in U$, we may assume that $\Spa A\to \Spa \Qbb_p$ is smooth by replacing $\Spa A$ with an affinoid open neighborhood of $x$.
	We set $V_A=V_R\hotimes_R A$.
	Since $\Ccal_{\FF,\Gamma_K}^{\bullet}(\underline{\End}(V_A))$ is quasi-isomorphic to a perfect complex of $A$-modules, there exists $f\in A\setminus\{0\}$ such that $H^i(\Ccal_{\FF,\Gamma_K}^{\bullet}(\underline{\End}(V_A))\otimes_{A} A_f)$ is finite free of rank $r_i$ for $i=1,2$.
	By replacing $\Spa A$ with an affinoid open neighborhood of $x$ contained in $\Spa A\setminus \{f=0\}$, we may assume that $H^i(\Ccal_{\FF,\Gamma_K}^{\bullet}(\underline{\End}(V_A)))$ is finite free of rank $r_i$ for $i=1,2$.

	For $I\in\{[1,p],[1,1],[p,p]\}$, let $|\cdot|_{I}$ be a spectral norm on $\tilde{B}_{K_{\infty}}^{I}$ normalized so that $|p|_{I}=p^{-1}$.
	We also fix a norm $|\cdot|$ on $A$.
	Then we get a product norm $|\cdot|_{I}$ on $\tilde{B}_{K_{\infty},A}^{I}$.
	We also choose a $\Gamma_K$-invariant norm $|\cdot|_{[1,p]}$ on $V_A^{[1,p]}$, which induces a $\Gamma_K$-invariant norm $|\cdot|_{[i,i]}$ on $V_A^{[i,i]}$ for $i=1,p$.
	We endow $\End(V_A^I)$ with the operator norm.
	To simplify notation, we write $C^i=\Ccal_{\FF,\Gamma_K}^{i}(\underline{\End}(V_A))$ and use the standard notation $H^i, Z^i, B^i,\partial$ for its cohomology, cocycles, coboundaries, and differentials.
	Let $e_1,\ldots, e_{r_2}$ be a basis of the finite free $A$-module $H^2$.
	We endow $C^1=\End(V_A^{[1,1]})\oplus \End(V_A^{[1,p]})$ with the sup-norm, $H^2=A^{r_2}$ with the sup-norm, and $B^2$ with the induced norm.
	From now on, we drop the subscripts from the norms.
	Let 
	\[q_1\colon C^2\to H^2=A^{r_2}\]
	be the natural projection.
	We fix a splitting $\iota_1\colon H^2=A^{r_2}\to C^2$ of $q_1$ as Banach $A$-modules, which induces a continuous projection
	\[q_2\colon C^2\to B^2.\]
	Let us take integers $r,s,u>0$ such that:
	\begin{enumerate}
		\item For any $f\in V_A^{[p,p]}$, $|\varphi(f)|\leq p^r|f|$, and for any $f'\in V_A^{[1,1]}$, $|\varphi^{-1}(f')|\leq p^r|f'|$ (such $r$ exists since $\varphi$ and $\varphi^{-1}$ are continuous).
		\item For any $f\in C^2$, $|q_1(f)|\leq p^s|f|$ and $|q_2(f)|\leq p^s|f|$ (such $s$ exists since $q_1,q_2$ are continuous).
		\item For any $f\in B^2$, there exists $f'\in C^1$ such that $\partial(f')=f$ and $|f'|\leq p^u |f|$ (such $u$ exists by the open mapping theorem).
	\end{enumerate}
	We fix representatives $(\alpha_1,\beta_1),\ldots,(\alpha_{r_1},\beta_{r_1})\in Z^1$ for a basis of the finite free $A$-module $H^1$.
	We may assume $|(\alpha_i,\beta_i)|\leq 1$ for any $i$. 

	Let us construct 
	\[g\colon \Spa \left(A\left\langle \frac{X_1,\ldots,X_{r_1}}{p^m}\right\rangle/(F_1,\ldots,F_{r_2})\right)\to \LS_{\GL_n(K),\FF}.\]
	We set $V_n^{I}=V_A^{I}\otimes_A A\langle\underline{X}\rangle/(\underline{X})^{n+1}$ for $I=[1,p],[1,1],[p,p]$ and $n\geq 1$, where $\underline{X}=(X_1,\ldots,X_{r_1})$.
	We define 
	\begin{align*}
		\varphi_1&= (1+\sum_{i=1}^{r_1}\alpha_iX_i)\varphi \colon V_1^{[p,p]}\to V_1^{[1,1]}\\
		\gamma_1 &=(1+\sum_{i=1}^{r_1}\beta_iX_i)\gamma \colon V_1^{[1,p]}\to V_1^{[1,p]},
	\end{align*}
	where $\varphi$ in the definition is the induced morphism $\varphi \otimes \id\colon V_1^{[p,p]} \to V_1^{[1,1]}$ and similarly for $\gamma$.
	Since each $(\alpha_i,\beta_i)$ is a cocycle, it defines a $\Gamma_K$-equivariant vector bundle $V_1$ on $X_{K_{\infty},A\langle\underline{X}\rangle/(\underline{X})^2}$.

	Next, we regard $\varphi_1$ (resp.\ $\gamma_1$) as a morphism $V_2^{[p,p]}\to V_2^{[1,1]}$ (resp.\ $V_2^{[1,p]}\to V_2^{[1,p]}$).
	We set
	\begin{align*}
		\varphi_1^{-1}&= \varphi^{-1}(1+\sum_{i}\overline{\alpha_i}X_i+\sum_{i\leq j} \overline{\alpha_{ij}}X_i X_j) \colon V_2^{[1,1]}\to V_2^{[p,p]}\\
		\gamma_1^{-1} &=\gamma^{-1}(1+\sum_{i}\overline{\beta_i}X_i+\sum_{i\leq j}\overline{\beta_{i,j}}X_iX_j) \colon V_2^{[1,p]}\to V_2^{[1,p]},
	\end{align*}
	where we note that modulo $(\underline{X})^3$, $(1+\sum_{i}\overline{\alpha_i}X_i+\sum_{i\leq j} \overline{\alpha_{ij}}X_i X_j)\equiv(1+\sum_{i=1}^{r_1}\alpha_iX_i)^{-1}$ and $(1+\sum_{i}\overline{\beta_i}X_i+\sum_{i\leq j} \overline{\beta_{ij}}X_i X_j)\equiv(1+\sum_{i=1}^{r_1}\beta_iX_i)^{-1}$.
	Hence, we have $|(\overline{\alpha_i},\overline{\beta_i})|\leq 1$ and $|(\overline{\alpha_{ij}},\overline{\beta_{ij}})|\leq 1$.
	There is an obstruction 
	\[q_1(\varphi_1\gamma_1\varphi_1^{-1}\gamma_1^{-1}-1)\in H^2\otimes_A (\underline{X})^2/(\underline{X})^3.\]
	We define $\delta_{ij}^k \in A$ by
	\[q_1(\varphi_1\gamma_1\varphi_1^{-1}\gamma_1^{-1}-1)=\sum_{k=1}^{r_2} (\sum_{i\leq j}\delta_{ij}^kX_iX_j)e_k\in H^2\otimes_A (\underline{X})^2/(\underline{X})^3.\]
	Then we get $|\delta_{ij}^k| \leq p^{2r+s}$ by (1), (2), $|(\alpha_i,\beta_i)|\leq 1$, $|(\overline{\alpha_i},\overline{\beta_i})|\leq 1$, and $|(\overline{\alpha_{ij}},\overline{\beta_{ij}})|\leq 1$.
	We set $F_2^k=\sum_{i\leq j}\delta_{ij}^kX_iX_j.$
	Next, we take $(\alpha_{ij},\beta_{ij})\in C^1$ such that 
	\[\partial(\sum_{i\leq j}(\alpha_{ij},\beta_{ij})X_iX_j)=q_2(\varphi_1\gamma_1\varphi_1^{-1}\gamma_1^{-1}-1)\]
	in $B^2\otimes_A (\underline{X})^2/(\underline{X})^3$, where we note that $\varphi_1\gamma_1\varphi_1^{-1}\gamma_1^{-1}-1 \in C^2\otimes_A (\underline{X})^2/(\underline{X})^3.$
	By (2) and (3), we can choose such $(\alpha_{ij},\beta_{ij})\in C^1$ so that $|(\alpha_{ij},\beta_{ij})| \leq p^{2r+s+u}$.
	We define
	\begin{align*}
		\varphi_2&= (1+\sum_i\alpha_iX_i+\sum_{i\leq j}\alpha_{ij}X_iX_j)\varphi \colon V_2^{[p,p]}\to V_2^{[1,1]},\\
		\gamma_2&=(1+\sum_i\beta_iX_i+\sum_{i\leq j}\beta_{ij}X_iX_j)\gamma \colon V_2^{[1,p]}\to V_2^{[1,p]}.
	\end{align*}
	Then, we have
	\[\varphi_2\gamma_2\varphi_2^{-1}\gamma_2^{-1}-1\equiv \iota_1(\sum_{k=1}^{r_2} F_2^k e_k)\mod (\underline{X})^3.\]

	Next, for the induction step, suppose that for some $l \geq 2$, we have constructed
	\begin{align*}
		F_l^k&=\sum_{2\leq m\leq l, i_1\leq\cdots \leq i_m}\delta_{i_1\ldots i_m}^kX_{i_1}\cdots X_{i_m},\\
		\varphi_l&= (1+\sum_{1\leq m\leq l, i_1\leq\cdots\leq i_m}\alpha_{i_1\ldots i_m}X_{i_1}\cdots X_{i_m})\varphi \colon V_l^{[p,p]}\to V_l^{[1,1]},\\
		\gamma_l&=(1+\sum_{1\leq m\leq l, i_1\leq\cdots\leq i_m}\beta_{i_1\ldots i_m}X_{i_1}\cdots X_{i_m})\gamma \colon V_l^{[1,p]}\to V_l^{[1,p]}
	\end{align*}
	such that
	\begin{itemize}
		\item $|\delta_{i_1\ldots i_m}^k|\leq p^{(2r+s+u)(m-1)-u}$,
		\item $|(\alpha_{i_1\ldots i_m},\beta_{i_1\ldots i_m})|\leq p^{(2r+s+u)(m-1)}$,
		\item $\varphi_l\gamma_l\varphi_l^{-1}\gamma_l^{-1}-1\equiv \iota_1(\sum_{k=1}^{r_2} F_l^k e_k) \mod (\underline{X})^{l+1}$.
	\end{itemize}
	We regard $\varphi_l$ (resp.\ $\gamma_l$) as a morphism $V_{l+1}^{[p,p]}\to V_{l+1}^{[1,1]}$ (resp.\ $V_{l+1}^{[1,p]}\to V_{l+1}^{[1,p]}$).
	We set
	\begin{align*}
		\varphi_l^{-1}&=\varphi^{-1}(1+\sum_{1\leq m\leq l+1, i_1\leq\cdots\leq i_m}\overline{\alpha_{i_1\ldots i_m}}X_{i_1}\cdots X_{i_m}) \colon V_{l+1}^{[1,1]}\to V_{l+1}^{[p,p]}\\
		\gamma_l^{-1}&=\gamma^{-1}(1+\sum_{1\leq m\leq l+1, i_1\leq\cdots \leq i_m}\overline{\beta_{i_1\ldots i_m}}X_{i_1}\cdots X_{i_m}) \colon V_{l+1}^{[1,p]}\to V_{l+1}^{[1,p]}.
	\end{align*}
	Then for $m\leq l$, we have $|(\overline{\alpha_{i_1\ldots i_m}},\overline{\beta_{i_1\ldots i_m}})|\leq p^{(2r+s+u)(m-1)}$, and for $m=l+1$, we have $|(\overline{\alpha_{i_1\ldots i_{l+1}}},\overline{\beta_{i_1\ldots i_{l+1}}})|\leq p^{(2r+s+u)(l-1)}$.
	There is an obstruction
	\[q_1(\varphi_l\gamma_l\varphi_l^{-1}\gamma_l^{-1}-1-\iota_1(\sum_{k=1}^{r_2} F_l^k e_k))\in H^2\otimes_A (\underline{X})^{l+1}/(\underline{X})^{l+2}.\]
	We define $\delta_{i_1\ldots i_{l+1}}^k\in A$ by 
	\begin{align*}
	&q_1(\varphi_l\gamma_l\varphi_l^{-1}\gamma_l^{-1}-1-\iota_1(\sum_{k=1}^{r_2} F_l^k e_k))\\
	=&\sum_{k=1}^{r_2} (\sum_{i_1\leq \cdots \leq i_{l+1}}\delta_{i_1\ldots i_{l+1}}^kX_{i_1}\cdots X_{i_{l+1}})e_k\in H^2\otimes_A (\underline{X})^{l+1}/(\underline{X})^{l+2}.
    \end{align*}
	Then we have $|\delta_{i_1\ldots i_{l+1}}^k|\leq p^{(2r+s+u)l-u}$ by (1), (2), and the induction hypothesis.
	We set
	\[F_{l+1}^k=\sum_{2\leq m\leq l+1, i_1\leq\cdots\leq i_m}\delta_{i_1\ldots i_m}^kX_{i_1}\cdots X_{i_m}.\]
	Next, we take $(\alpha_{i_1\ldots i_{l+1}},\beta_{i_1\ldots i_{l+1}})\in C^1$ such that 
	\[\partial(\sum_{i_1\leq \cdots \leq i_{l+1}}(\alpha_{i_1\ldots i_{l+1}},\beta_{i_1\ldots i_{l+1}})X_{i_1}\cdots X_{i_{l+1}})=q_2(\varphi_l\gamma_l\varphi_l^{-1}\gamma_l^{-1}-1)\]
	in $B^2\otimes_A (\underline{X})^{l+1}/(\underline{X})^{l+2}$.
	By (2) and (3), we can choose such $(\alpha_{i_1\ldots i_{l+1}},\beta_{i_1\ldots i_{l+1}})\in C^1$ so that $|(\alpha_{i_1\ldots i_{l+1}},\beta_{i_1\ldots i_{l+1}})| \leq p^{(2r+s+u)l}$.
	We set 
	\begin{align*}
		\varphi_{l+1}&= (1+\sum_{1\leq m\leq l+1, i_1\leq\cdots\leq i_m}\alpha_{i_1\ldots i_m}X_{i_1}\cdots X_{i_m})\varphi \colon V_{l+1}^{[p,p]}\to V_{l+1}^{[1,1]}\\
		\gamma_{l+1}&=(1+\sum_{1\leq m\leq l+1, i_1\leq\cdots\leq i_m}\beta_{i_1\ldots i_m}X_{i_1}\cdots X_{i_m})\gamma \colon V_{l+1}^{[1,p]}\to V_{l+1}^{[1,p]}
	\end{align*}
	By construction, we have
	\[\varphi_{l+1}\gamma_{l+1}\varphi_{l+1}^{-1}\gamma_{l+1}^{-1}-1\equiv \iota_1(\sum_{k=1}^{r_2} F_{l+1}^k e_k) \mod (\underline{X})^{l+2}.\]

	By the above inductive argument, we formally obtain
	\begin{align*}
		F_{\infty}^k&=\sum_{2\leq m, i_1\leq\cdots\leq i_m}\delta_{i_1\ldots i_m}^kX_{i_1}\cdots X_{i_m},\\
		\varphi_{\infty}&= (1+\sum_{1\leq m, i_1\leq\cdots\leq i_m}\alpha_{i_1\ldots i_m}X_{i_1}\cdots X_{i_m})\varphi \colon V_A^{[p,p]}[[\underline{X}]]\to V_A^{[1,1]}[[\underline{X}]],\\
		\gamma_{\infty}&=(1+\sum_{1\leq m, i_1\leq\cdots\leq i_m}\beta_{i_1\ldots i_m}X_{i_1}\cdots X_{i_m})\gamma \colon V_A^{[1,p]}[[\underline{X}]]\to V_A^{[1,p]}[[\underline{X}]]
	\end{align*}
		such that
	\begin{itemize}
		\item $|\delta_{i_1\ldots i_m}^k|\leq p^{(2r+s+u)(m-1)-u}$,
		\item $|(\alpha_{i_1\ldots i_m},\beta_{i_1\ldots i_m})|\leq p^{(2r+s+u)(m-1)}$,
		\item $\varphi_{\infty}\gamma_{\infty}\varphi_{\infty}^{-1}\gamma_{\infty}^{-1}-1= \iota_1(\sum_{k=1}^{r_2} F_{\infty}^k e_k)$.
	\end{itemize}
	Let $\tilde{B}=A\left\langle \frac{\underline{X}}{p^{2r+s+u+1}}\right\rangle$ with the Gauss norm.
	Then we have $F_{\infty}^1,\ldots,F_{\infty}^{r_2}\in \tilde{B}$.
	Let $B=\tilde{B}/(F_{\infty}^1,\ldots,F_{\infty}^{r_2})$.
	For $I=[1,p],[1,1],[p,p]$, we set 
	\[V_{\tilde{B}}^{I}=V_A^{I}\otimes_{\tilde{B}_{K_{\infty},A}^{I}} \tilde{B}_{K_{\infty},\tilde{B}}^{I},\] 
	equipped with the induced norm.
	Then $\varphi_{\infty}$ (resp.\ $\gamma_{\infty}$) defines a $\varphi$-semilinear isomorphism $\varphi_{\infty}\colon V_{\tilde{B}}^{[p,p]}\xrightarrow{\sim}V_{\tilde{B}}^{[1,1]}$ (resp.\ a $\gamma$-semilinear isomorphism $\gamma_{\infty}\colon V_{\tilde{B}}^{[1,p]}\xrightarrow{\sim}V_{\tilde{B}}^{[1,p]}$). 
	Moreover, $\gamma_{\infty}$ defines a continuous $\Gamma_K$-action on $V_{\tilde{B}}^{[1,p]}$ since $\gamma_{\infty}=(1+u)\gamma$ with $|u|<1$, and the $\gamma$-invariance of the norm implies that $\gamma_{\infty}^{p^a}\to 1$ as $a\to \infty$.
	Since we have 
	\[\varphi_{\infty}\gamma_{\infty}\varphi_{\infty}^{-1}\gamma_{\infty}^{-1}-1= \iota_1(\sum_{k=1}^{r_2} F_{\infty}^k e_k),\]
	these actions define a $\Gamma_K$-equivariant vector bundle $V_B$ on $X_{K_{\infty},B}$. 
	In other words, it defines a morphism 
	\[g\colon \Spa B=\Spa \left(A\left\langle \frac{X_1,\ldots,X_{r_1}}{p^{2r+s+u+1}}\right\rangle/(F_{\infty}^1,\ldots,F_{\infty}^{r_2})\right)\to \LS_{\GL_n(K),\FF}.\]
	By comparing this construction with the proof of \cite[Expos\'{e} VI, Theorem 1.11]{SGA7} and noting Lemma \ref{lem:Hensel's lemma}, we find that for any classical point $t\in (\Spa A)^{\cl}$, the induced morphism
	\[g\colon \Spa \left(k(t)\left\langle \frac{X_1,\ldots,X_{r_1}}{p^{2r+s+u+1}}\right\rangle/(F_{\infty}^1,\ldots,F_{\infty}^{r_2})\right)\to \LS_{\GL_n(K),\FF}\] 
	is formally versal at $i(t)$ (see the following remark).
\end{proof}
\begin{remark}\label{rem:comparison with Rim-Schlessinger}
	We use the same notation as in the above proof.
	Let $t\in (\Spa A)^{\cl}$ be a classical point.
	For $l\geq 2$, we define an ideal $I_{l}=(F^1_l,\ldots,F^{r_2}_l)+(\underline{X})^{l+1}\subset k(t)[[\underline{X}]]$.
	Then we have $\bigcap_{l}I_l=(F_{\infty}^1,\ldots,F_{\infty}^{r_2})$.
	In the above proof, we also constructed a compatible system of vector bundles $V^l$ on $X_{K_{\infty}}\times \Spa (k(t)[[\underline{X}]]/I_l)$.
	This defines a compatible system 
	\[\Spa (k(t)[[\underline{X}]]/I_l)\to \LS_{\GL_n(K),\FF},\]
	which is a formal neighborhood at the origin of the morphism 
	\[g\colon \Spa \left(k(t)\left\langle \frac{X_1,\ldots,X_{r_1}}{p^{2r+s+u+1}}\right\rangle/(F_{\infty}^1,\ldots,F_{\infty}^{r_2})\right)\to \LS_{\GL_n(K),\FF}\] 
	constructed above.
	We set $I_{1}=(\underline{X})^2$.
	By construction, for each $l\geq 2$, $I_l$ is the smallest ideal satisfying the following properties.
	\begin{enumerate}
		\item $I_{l-1}\supset I_l \supset (\underline{X})I_{l-1}$.
		\item The $\Gamma_K$-equivariant vector bundle $V^{l-1}$ on $X_{K_{\infty}}\times \Spa (k(t)[[\underline{X}]]/I_{l-1})$ extends to a $\Gamma_K$-equivariant vector bundle $V^{l}$ on $X_{K_{\infty}}\times \Spa (k(t)[[\underline{X}]]/I_{l})$.
	\end{enumerate}
	This is the characterization appearing in \cite[Expos\'{e} VI, Proposition 1.8 and Theorem 1.11]{SGA7}.
\end{remark}

\begin{remark}
	In this construction, it is important that every $2$-cochain of $\Ccal_{\FF,\Gamma_K}(\underline{\End}(V_A))$ is a $2$-cocycle.
\end{remark}

\begin{theorem}
	The \'{e}tale stack $\LS_{\GL_n(K),\FF}$ on $\Aff_{\Qbb_p}^{\op}$ is a rigid analytic Artin stack.
\end{theorem}
\begin{proof}
	This follows from Theorem \ref{thm:criterion for representablity}, Lemma \ref{lem:EG stack Fin-homogeneous}, Corollary \ref{cor:EG stack diagonal}, Lemma \ref{lem:coherent deformation and obstruction Herr complex}, and Theorem \ref{thm:EG stack smooth chart}.
\end{proof}

For later use, we record the following Corollary of Theorem \ref{thm:EG stack smooth chart}.
\begin{corollary}\label{cor:smooth chart}
	Let $F$ be a finite extension of $\Qbb_p$, and let $x\colon \Spa F\to \LS_{\GL_n(K),\FF}$ be a classical point corresponding to a $\Gamma_K$-equivariant vector bundle $V_x$ on $X_{K_{\infty},F}$.
	Set $r_i=\dim_F H^i(\Ccal_{\FF,\Gamma_K}^{\bullet}(\underline{\End}(V_x)))$ for $i=0,1,2$.
	For a sufficiently large integer $m$, there exists a smooth morphism
	\[g\colon \Spa \left(F\left\langle\frac{X_1,\ldots,X_{r_1}}{p^m}\right\rangle/(F_1,\ldots,F_{r_2})\right)\to \LS_{\GL_n(K),\FF}\]
	of pure dimension $r_0$ whose image contains $x$, where 
	\[F_1,\ldots,F_{r_2}\in (X_1,\ldots,X_{r_1})^2F\left\langle\frac{X_1,\ldots,X_{r_1}}{p^m}\right\rangle.\]
\end{corollary}
\begin{proof}
	Let
	\[g\colon \Spa \left(F\left\langle\frac{X_1,\ldots,X_{r_1}}{p^m}\right\rangle/(F_1,\ldots,F_{r_2})\right)\to \LS_{\GL_n(K),\FF}\]
	be a morphism constructed in Theorem \ref{thm:EG stack smooth chart}.
	We need to show that after enlarging $m$, $g$ becomes a smooth morphism of pure dimension $r_0$.
	We set $U=\Spa \left(F\left\langle\frac{X_1,\ldots,X_{r_1}}{p^m}\right\rangle/(F_1,\ldots,F_{r_2})\right)$.
	Let $y$ be a classical point of $U$ defined by $X_1=\cdots=X_{r_1}=0$.
	By construction, we have $g(y)=x$.
	By Lemma \ref{lem:smooth of pure dimension}, the relative dimension of $g$ is constant on a neighborhood of $y$.
	Thus, after replacing $m$ by a sufficiently large integer, it suffices to show that the relative dimension at $y$ is $r_0$.
	Set $V=U\times_{\LS_{\GL_n(K),\FF},x}\Spa F$, which is a rigid space by Corollary \ref{cor:EG stack diagonal}.
	From the classical point $y\colon \Spa F\to U$, we obtain the classical point $z=(y,\id)\colon \Spa F\to U\times_{\LS_{\GL_n(K),\FF},x}\Spa F=V$.
	Let us prove that $\dim_z V=r_0$.
	The tangent space $T_{V,z}$ of $V$ at $z$ is the preimage of $z$ under the morphism $V(F[\epsilon]/(\epsilon^2))\to V(F)$, which is naturally equipped with an $F$-vector space structure.
	By construction and Lemma \ref{lem:deformation Herr complex}, this is isomorphic to $\Aut_{\LS_{\GL_n(K),\FF}}(F,F)\cong H^0(\Ccal_{\FF,\Gamma_K}^{\bullet}(\underline{\End}(V_x)))$, which proves the claim.
\end{proof}

\subsection{Representability for general $G$}

Let $G$ be a linear algebraic group over $K$.
For an affinoid $K$-algebra $A$, we regard $X_{K_{\infty},A}$ as an adic space over $\Spa K$ via $X_{K_{\infty},A}\to \Spa A\to \Spa K$.

In Appendix \ref{appendix:Gamma-equivariant torsor}, we define the notion of geometric (resp.\ Tannakian) $\Gamma_K$-equivariant $G$-torsors on $X_{K_{\infty},A}$.
Since these two notions are equivalent (Theorem \ref{thm:geometric vs Tannakian}), we simply refer to them as \textit{$\Gamma_K$-equivariant $G$-torsors}.

\begin{definition}\label{defn:EG stack general algebriac group}
	Let $\LS_{G(K),\FF}$ denote the \'{e}tale stack 
    \[\Aff_{K}\ni A\mapsto \{\text{$G_K$-equivariant $G$-torsors on $X_{\Cbb_p,A}$}\}.\]
\end{definition}

\begin{remark}
	By considering the Tannakian formalism (Definition \ref{defn:tannakian torsor}), we find that $\LS_{G(K),\FF}$ is equivalent to the \'{e}tale stack 
	\[\Aff_{K}\ni A\mapsto \{\text{$\Gamma_K$-equivariant $G$-torsors on $X_{K_{\infty},A}$}\}.\]
\end{remark}

We deduce the representability of $\LS_{G(K),\FF}$ from that of $\LS_{\GL_n(K),\FF}$.
We choose a closed embedding $G\hookrightarrow \GL_n$ of algebraic groups.

\begin{theorem}\label{thm:EG stack representable for general G}
	The natural morphism 
	\[\LS_{G(K),\FF} \to \LS_{\GL_n(K),\FF}\]
	induced by the closed embedding $G\hookrightarrow \GL_n$ is representable by separated rigid spaces.
	In particular, $\LS_{G(K),\FF}$ is a rigid analytic Artin stack (over $K$).
\end{theorem}

Before the proof, we establish the following statements.

\begin{construction}\label{construction:Chevalley}
By Chevalley's theorem (\cite[Theorem 11.2]{Humphreys75}), there exist a finite-dimensional algebraic representation $E$ of $\GL_n$ and a one-dimensional subspace $L\subset E$ such that the stabilizer of $L$ in $\GL_n$ is $G$.
Let $\Pbb(E)$ denote the projective bundle over $\Spec K$ associated to $E$, which classifies line subbundles of $E$, and let $\Pbb(E)^{\an}$ denote its analytification.
The one-dimensional subspace $L\subset E$ defines a $\Spa K$-valued point $x$ of $\Pbb(E)^{\an}$.
Then the quotient $\GL_n^{\an}/G^{\an}$ (as \'{e}tale sheaves) is represented by the analytification of the quasi-projective scheme $\GL_n/G$ over $K$, and the morphism $\GL_n^{\an}/G^{\an}\to \Pbb(E)^{\an};\; \alpha\mapsto \alpha x$ is a locally closed immersion by \cite[Proposition 8.3]{Humphreys75}.
\end{construction}

For an affinoid $K$-algebra $A$, the adic space $X_{K_{\infty},A}$ is locally noetherian by \cite[Theorem 3.2]{Ked16}.
Therefore, we can endow the category of adic spaces locally of finite type over $X_{K_{\infty},A}$ with the \'{e}tale topology as in \cite{Hub96}.
An \'{e}tale sheaf over $X_{K_{\infty},A}$ means a sheaf on this site.

\begin{lemma}\label{lem:quotient quasi-projective}
	Let $A$ be an affinoid $K$-algebra, and let $P\to X_{K_{\infty},A}$ be a geometric $\GL_n$-torsor.
	Let $\Ecal$ be a geometric vector bundle on $X_{K_{\infty},A}$ defined by $\Ecal=P\times^{\GL_n^{\an}} E^{\an}$, where $E^{\an}$ is the geometric vector bundle on $\Spa K$ associated to $E$.
	Let $\Pbb(\Ecal)$ be the projective bundle on $X_{K_{\infty},A}$ associated to $\Ecal$.
	Then the quotient $P/(G^{\an}\times X_{K_{\infty},A})$ as an \'{e}tale sheaf over $X_{K_{\infty},A}$ is representable by an adic space locally of finite type over $X_{K_{\infty},A}$, and there is a locally closed immersion
	\[P/(G^{\an}\times X_{K_{\infty},A})\hookrightarrow \Pbb(\Ecal).\]
\end{lemma}
\begin{proof}
	Since $P$ is a $\GL_n$-torsor, we have an isomorphism of \'{e}tale sheaves
	\[P\times^{\GL_n^{\an}} (\GL_n^{\an}/G^{\an}) \cong P/(G^{\an}\times X_{K_{\infty},A}).\]
	Moreover, by construction, we have an isomorphism of \'{e}tale sheaves
	\[P\times^{\GL_n^{\an}}\Pbb(E)^{\an}\cong \Pbb(\Ecal).\]
	In particular, $P\times^{\GL_n^{\an}}\Pbb(E)^{\an}$ is representable by an adic space locally of finite type over $X_{K_{\infty},A}$.
	By applying $P\times^{\GL_n^{\an}}(-)$ to the $\GL_n^{\an}$-equivariant morphism $\GL_n^{\an}/G^{\an}\to \Pbb(E)^{\an}$, we obtain a morphism
	\begin{align*}
	P\times^{\GL_n^{\an}} (\GL_n^{\an}/G^{\an})\to P\times^{\GL_n^{\an}}\Pbb(E)^{\an}.
    \end{align*}
	Since $P$ is a $\GL_n$-torsor, \'{e}tale locally on $X_{K_{\infty},A}$, this morphism is a locally closed immersion.
	Therefore, by Lemma \ref{lem:locally closed immersion smooth local on the target}, we obtain the claim.
\end{proof}

\begin{proposition}\label{prop:locally closed subspace representable}
	Let $S$ be a rigid space over $K$, and let $Z\subset X_{K_{\infty}}\times S$ be a locally Zariski closed subspace.
	Then there exists a locally Zariski closed subspace $Z'\subset S$ such that for any morphism $f\colon \Spa A \to S$, $f$ factors through $Z'$ if and only if the base change $X_{K_{\infty}}\times \Spa A\to X_{K_{\infty}}\times S$ of $f$ factors through $Z\subset X_{K_{\infty}}\times S$.
\end{proposition}
\begin{proof}
	Let $U\subset X_{K_{\infty}}\times S$ be an open subspace such that $Z\subset U$ is a Zariski closed subspace.
	Let $T\subset S$ be the image of $(X_{K_{\infty}}\times S)\setminus U$ in $S$, which is closed because $X_{K_{\infty}}\times S\to S$ is a closed map by Lemma \ref{lem:FF curve closed map}.
	Then for any morphism $f\colon \Spa A \to S$, $f$ factors through $S\setminus T$ if and only if the base change $X_{K_{\infty}}\times \Spa A\to X_{K_{\infty}}\times S$ of $f$ factors through $U\subset X_{K_{\infty}}\times S$.
	Therefore, by replacing $S$ with $S\setminus T$, we may assume that $Z\subset X_{K_{\infty}}\times S$ is a Zariski closed subspace.

	We may assume that $S=\Spa A$ is an affinoid rigid space.
	Let $\tilde{Z}$ be the preimage of $Z$ under the morphism $\Spa \tilde{B}_{K_{\infty},A}^{[1,p]}\to X_{K_{\infty},A}$.
	Note that $\Spa \tilde{B}_{K_{\infty},A}^{[1,p]}\to X_{K_{\infty},A}$ is surjective and a local isomorphism.
	Let $J\subset \tilde{B}_{K_{\infty},A}^{[1,p]}$ be the ideal defining $\tilde{Z}$.
	Let $\Lambda$ be the set of ideals $I$ of $A$ such that $J\subset I\tilde{B}_{K_{\infty},A}^{[1,p]}$.
	We prove that $\Lambda$ has a smallest element.
	It suffices to show that $\bigcap_{I\in \Lambda} I$ also lies in $\Lambda$.
	Since the Banach $K$-module $\tilde{B}_{K_{\infty}}^{[1,p]}$ is of countable type, there is an isomorphism
	$\tilde{B}_{K_{\infty}}^{[1,p]}\cong \hat{\oplus}K$ of Banach $K$-modules by \cite[2.7.2.Proposition 8]{BGR}, where $\hat{\oplus}K$ is the completion of $\oplus K$ with respect to the sup-norm.
	Then we have $\tilde{B}_{K_{\infty},A}^{[1,p]} \cong \hat{\oplus}A$, and for an ideal $I$ of $A$, we have $I\tilde{B}_{K_{\infty},A}^{[1,p]} \cong \hat{\oplus}I$.
	Since we have 
	\[\bigcap_{I\in \Lambda}\hat{\bigoplus}I=\hat{\bigoplus}\bigcap_{I\in \Lambda} I\] 
	as submodules of $\hat{\oplus}A$, we obtain 
	\[\bigcap_{I\in \Lambda}\left(I\tilde{B}_{K_{\infty},A}^{[1,p]}\right)=\left(\bigcap_{I\in \Lambda}I\right)\tilde{B}_{K_{\infty},A}^{[1,p]}\]
	as ideals of $\tilde{B}_{K_{\infty},A}^{[1,p]}$, which implies that $\bigcap_{I\in \Lambda} I$ also lies in $\Lambda$.

	Let $I_0\in \Lambda$ be the smallest element.
	We prove that $\Spa A/I_0\subset \Spa A$ is the desired Zariski closed subspace.
	Let $A\to B$ be a morphism of affinoid $K$-algebras, and let $I\subset A$ be the kernel of $A\to B$.
	Since the morphism $\Spa \tilde{B}_{K_{\infty},A}^{[1,p]}\to X_{K_{\infty},A}$ is surjective and a local isomorphism, the base change $X_{K_{\infty},B}\to X_{K_{\infty},A}$ of $f$ factors through $Z\subset X_{K_{\infty},A}$ if and only if the image of $J$ in $\tilde{B}_{K_{\infty},B}^{[1,p]}$ is zero.
	Since the morphism $\tilde{B}_{K_{\infty},A/I}^{[1,p]}\to \tilde{B}_{K_{\infty},B}^{[1,p]}$ is injective, the above condition is equivalent to the condition that the image of $J$ in $\tilde{B}_{K_{\infty},A/I}^{[1,p]}$ is zero. 
	By construction, this condition is equivalent to $I_0\subset I$, which proves the claim.
\end{proof}

\begin{proof}[Proof of Theorem \ref{thm:EG stack representable for general G}]
	We use the same notation as in Construction \ref{construction:Chevalley} and Lemma \ref{lem:quotient quasi-projective}.
	Then we have a $\Gamma_K$-equivariant locally closed immersion 
	\[P/(G^{\an}\times X_{K_{\infty},A})\hookrightarrow \Pbb(\Ecal).\] 
	We define the \'{e}tale sheaf over $\Spa A$
	\[\underline{\Map}_{X_{K_{\infty},A}}^{\Gamma_K}(X_{K_{\infty},A},P/(G^{\an}\times X_{K_{\infty},A}))\] 
	by
	\begin{align*}
		&(\Spa B\to \Spa A) \\
		\mapsto &\{\text{$\Gamma_K$-equivariant sections of $P/(G^{\an}\times X_{K_{\infty},A})\times_{\Spa A}\Spa B\to X_{K_{\infty},B}$}\}.
	\end{align*}
	Similarly, we define the \'{e}tale sheaf $\underline{\Map}_{X_{K_{\infty},A}}^{\Gamma_K}(X_{K_{\infty},A},\Pbb(\Ecal)).$
	\begin{description}
	\item[Claim] The morphism
	\[\underline{\Map}_{X_{K_{\infty},A}}^{\Gamma_K}(X_{K_{\infty},A},P/(G^{\an}\times X_{K_{\infty},A}))\to \underline{\Map}_{X_{K_{\infty},A}}^{\Gamma_K}(X_{K_{\infty},A},\Pbb(\Ecal))\]
	is a locally closed immersion.
	\end{description}
	Let $B$ be an affinoid $A$-algebra, and let $\Spa B\to \underline{\Map}_{X_{K_{\infty},A}}^{\Gamma_K}(X_{K_{\infty},A},\Pbb(\Ecal))$ be a morphism corresponding to a section $\alpha\colon X_{K_{\infty},B}\to \Pbb(\Ecal)\times_{\Spa A}\Spa B$.
	Let $Y'\subset X_{K_{\infty},B}$ be the preimage of $P/(G^{\an}\times X_{K_{\infty},A})\times_{\Spa A}\Spa B\subset \Pbb(\Ecal)\times_{\Spa A}\Spa B$ under $\alpha$.
	Then the fiber product of the morphisms $\Spa B\to \underline{\Map}_{X_{K_{\infty},A}}^{\Gamma_K}(X_{K_{\infty},A},\Pbb(\Ecal))$ and $\underline{\Map}_{X_{K_{\infty},A}}^{\Gamma_K}(X_{K_{\infty},A},P/(G^{\an}\times X_{K_{\infty},A}))\to \underline{\Map}_{X_{K_{\infty},A}}^{\Gamma_K}(X_{K_{\infty},A},\Pbb(\Ecal))$ classifies the morphisms $f\colon \Spa R\to \Spa B$ such that the base change $X_{K_{\infty},R}\to X_{K_{\infty},B}$ of $f$ factors through $Y'\subset X_{K_{\infty},B}$.
	This is representable by a locally Zariski closed subspace of $\Spa B$ by Proposition \ref{prop:locally closed subspace representable}.
	
	By the claim, it suffices to show that the \'{e}tale sheaf $\underline{\Map}_{X_{K_{\infty},A}}^{\Gamma_K}(X_{K_{\infty},A},\Pbb(\Ecal))$ is representable by a separated rigid space.
	We regard $\Ecal$ as a $\Gamma_K$-equivariant vector bundle on $X_{K_{\infty},A}$.
	Then, $\underline{\Map}_{X_{K_{\infty},A}}^{\Gamma_K}(X_{K_{\infty},A},\Pbb(\Ecal))$ is isomorphic to the \'{e}tale sheaf $\Quot_{\Ecal^{\vee}/X_{K_{\infty},A}/A}^{1,\Gamma_K}$ classifying $\Gamma_K$-equivariant surjections from $\Ecal^{\vee}$ to a $\Gamma_K$-equivariant line bundle. 
	Let $X^{\an}(K^{\times})$ denote the rigid analytic moduli space over $K$ of continuous characters of $K^{\times}$.
	By \cite[Theorem 6.2.14]{KPX14}, we have
	\[\LS_{\GL_1(K),\FF}\cong X^{\an}(K^{\times})/\Gbb_m^{\an},\]
	where the action of $\Gbb_m^{\an}$ on $X^{\an}(K^{\times})$ is trivial.
	In particular, we have a natural morphism $\LS_{\GL_1(K),\FF}\to X^{\an}(K^{\times})$.
	By composing with the natural morphism $\Quot_{\Ecal^{\vee}/X_{K_{\infty},A}/A}^{1,\Gamma_K}\to \LS_{\GL_1(K),\FF}\times \Spa A$, we obtain a morphism
	\begin{align}\label{eq:projection line bundle}
		\Quot_{\Ecal^{\vee}/X_{K_{\infty},A}/A}^{1,\Gamma_K}\to X^{\an}(K^{\times})\times \Spa A.
	\end{align}
	We now prove that this morphism is representable by separated rigid spaces.
	Let $B$ be an affinoid $A$-algebra, and let $\alpha\colon \Spa B\to X^{\an}(K^{\times})\times \Spa A$ be a morphism corresponding to a continuous character $\delta\colon K^{\times}\to B^{\times}$.
	Let $\Ocal(\delta)$ be the $\Gamma_K$-equivariant line bundle on $X_{K_{\infty},B}$ corresponding to $\delta$.
	Then the fiber product of \eqref{eq:projection line bundle} and $\alpha$ is isomorphic to the \'{e}tale sheaf $\underline{\Surj}_{\Gamma_K}(\Ecal_B^{\vee},\Ocal(\delta))/\Gbb_{m,B}^{\an}$, where $\Gbb_{m,B}^{\an}$ acts on $\underline{\Surj}_{\Gamma_K}(\Ecal_B^{\vee},\Ocal(\delta))$ by scalar multiplication.
	Hence, it suffices to show that the \'{e}tale sheaf $\underline{\Surj}_{\Gamma_K}(\Ecal_B^{\vee},\Ocal(\delta))/\Gbb_{m,B}^{\an}$ is representable by a separated rigid space.
	By Lemma \ref{lem:finiteness of (phi, Gamma) cohomology}, there exists a morphism $f\colon M\to N$ of finite projective $B$-modules such that $\Vbb(M)\times_{\Vbb(N),e}\Spa B$ represents $\underline{\Hom}_{\Gamma_K}(\Ecal_B^{\vee},\Ocal(\delta))$, where $\Vbb(M)$ and $\Vbb(N)$ are the geometric vector bundles associated to $M$ and $N$, respectively, and $e\colon \Spa B\to \Vbb(N)$ is the zero section.
	Let $\Vbb(M)^{\circ}\subset \Vbb(M)$ denote the complement of the zero section.
	Then $\underline{\Surj}_{\Gamma_K}(\Ecal_B^{\vee},\Ocal(\delta))$ is a locally Zariski closed subspace of $\Vbb(M)^{\circ}$.
	Therefore, $\underline{\Surj}_{\Gamma_K}(\Ecal_B^{\vee},\Ocal(\delta))/\Gbb_{m,B}^{\an}$ is also a locally Zariski closed subspace of the projective bundle $\Vbb(M)^{\circ}/\Gbb_{m,B}^{\an}=\Pbb(M)$.
	Consequently, the \'{e}tale sheaf $\underline{\Surj}_{\Gamma_K}(\Ecal_B^{\vee},\Ocal(\delta))/\Gbb_{m,B}^{\an}$ is representable by a separated rigid space.
\end{proof}

%%%%%%%%%%%%%%%%%%%%%%%%%%%%%%%%%%%%%%%%%%%%%%%%%%%%%%%%%%%%%%%%%%%%%%%%%%%%%%%%%%%%%%%%%%
\section{Geometric properties of the analytic Emerton--Gee stacks}
We use the same notation as in the previous section.

\subsection{Main theorem and reduction to dimension bounds}
We recall the conjecture on the geometric properties of $\LS_{\GL_n(K),\FF}$.
We set $d=[K:\Qbb_p]$.
The goal of this section is to prove the following theorem.

\begin{theorem}\label{thm:EGH conjecture}
The analytic Emerton--Gee stack $\LS_{\GL_n(K),\FF}$ is normal, a local complete intersection, and equidimensional of dimension $dn^2$.
\end{theorem}

For an affinoid $\Qbb_p$-algebra $A$ and a $G_K$-equivariant vector bundle $V$ on $X_{\Cbb_p,A}$, we write $H^i_{\FF,G_K}(V)=H^i(G_K,R\Gamma(X_{\Cbb_p,A},V))$.
By \cite[Theorem 4.4.5]{KPX14}, for $i\neq 0,1,2$, $H^i_{\FF,G_K}(V)$ vanishes, and for $i=0,1,2$, $H^i_{\FF,G_K}(V)$ is a finitely generated $A$-module.
We also write $H^i_{\FF,G_K}(\intHom(V,W))=\Ext^i_{\FF,G_K}(V,W)$ for $G_K$-equivariant vector bundles $V$ and $W$ on $X_{\Cbb_p,A}$.

For a rigid analytic Artin stack $\Xcal$, we define a $G_K$-equivariant vector bundle $V$ on $X_{\Cbb_p}\times \Xcal$ of rank $n$ as a morphism $\Xcal\to \LS_{\GL_n(K),\FF}$.
By \cite[Theorem 4.4.3]{KPX14}, we can define the coherent sheaf $H^2_{\FF,G_K}(V)$ on $\Xcal$ by
\[(\Spa A\to \Xcal)\mapsto H^2_{\FF,G_K}(V_A),\]
where $V_A$ is the $G_K$-equivariant vector bundle $V$ on $X_{\Cbb_p,A}$ of rank $n$ corresponding to the composition $\Spa A\to \Xcal \to \LS_{\GL_n(K),\FF}$.
If the coherent sheaf $H^2_{\FF,G_K}(V)$ is locally free, then, by Tate local duality (\cite[Theorem 4.4.5]{KPX14}), we can define the vector bundle $H^0_{\FF,G_K}(V^{\vee}(1))$ on $\Xcal$ by
\[(\Spa A\to \Xcal)\mapsto H^0_{\FF,G_K}(V^{\vee}_A(1)),\]
which is isomorphic to the dual of $H^2_{\FF,G_K}(V)$.
Here $(1)$ is the Tate twist.

Theorem \ref{thm:EGH conjecture} follows from the following upper bound on the dimension of the locus where $\Ext^2$ does not vanish.

\begin{theorem}\label{thm:bad locus dimension upper bound}
	Let $V_n$ be the universal $G_K$-equivariant vector bundle $V$ on $X_{\Cbb_p}\times \LS_{\GL_n(K),\FF}$, and let $Z\subset \LS_{\GL_n(K),\FF}$ be the reduced Zariski closed substack defined by the $0$-th Fitting ideal $\Fit_0(\Ext^2_{\FF,G_K}(V_n,V_n))$.
	Then we have $\dim Z\leq dn^2-2.$
\end{theorem}

\begin{proof}[Proof of ``Theorem \ref{thm:bad locus dimension upper bound} $\implies$ Theorem \ref{thm:EGH conjecture}'']
Let $x\colon \Spa F\to \LS_{\GL_n(K),\FF}$ be a classical point.
Set $r_{i,x}=\dim_F \Ext^i_{\FF,G_K}(V_x,V_x)$ for $i=0,1,2$.

First, assume that $x$ is a classical point of $\LS_{\GL_n(K),\FF}\setminus Z$.
Then, by definition, we have $\Ext^2_{\FF,G_K}(V_x,V_x)=0$.
Therefore, by Corollary \ref{cor:smooth chart}, we obtain a smooth morphism
\[g\colon \Spa F\left\langle\frac{X_1,\ldots,X_{r_{1,x}}}{p^m}\right\rangle\to \LS_{\GL_n(K),\FF}\]
of pure dimension $r_{0,x}$ whose image contains $x$.
Thus, we find that $\LS_{\GL_n(K),\FF}$ is smooth (in particular, normal) at $x$ of dimension $r_{1,x}-r_{0,x}=dn^2$ by the Euler characteristic formula \cite[Theorem 1.2]{Liu08}.
In particular, we obtain 
\[\dim (\LS_{\GL_n(K),\FF}\setminus Z)=dn^2.\]
By Theorem \ref{thm:bad locus dimension upper bound}, we have $\dim Z\leq dn^2-2$, so we obtain $\dim \LS_{\GL_n(K),\FF}=dn^2$ by Corollary \ref{cor:dimension stack closed and complement}.

Next, let $x\colon \Spa F\to \LS_{\GL_n(K),\FF}$ be a general classical point.
As in Corollary \ref{cor:smooth chart}, let us take a smooth morphism
\[g\colon U=\Spa \left(F\left\langle\frac{X_1,\ldots,X_{r_{1,x}}}{p^m}\right\rangle/(F_1,\ldots,F_{r_{2,x}})\right)\to \LS_{\GL_n(K),\FF}\]
of pure dimension $r_{0,x}$ whose image contains $x$.
By Lemma \ref{lem:stack fiber dimension} and the Euler characteristic formula, we have 
\[\dim U\leq r_{0,x}+\dim\LS_{\GL_n(K),\FF}=r_{0,x}+dn^2=r_{1,x}-r_{2,x}.\]
On the other hand, by Krull's height theorem, we have $\dim U\geq r_{1,x}-r_{2,x}$.
Therefore, by Lemma \ref{lem:c.i.}, $U$ is a local complete intersection and equidimensional of dimension $r_{1,x}-r_{2,x}$, which implies that $\LS_{\GL_n(K),\FF}$ is a local complete intersection and equidimensional of dimension $dn^2$.
Moreover, by Theorem \ref{thm:bad locus dimension upper bound}, we have $\dim g^{-1}(Z)\leq r_{0,x}+dn^2-2=r_{1,x}-r_{2,x}-2$.
By the above argument, $U\setminus g^{-1}(Z)$ is normal.
Therefore, by Lemma \ref{lem:normal outside codimension 2}, $U$ is also normal, which implies that $\LS_{\GL_n(K),\FF}$ is normal.
\end{proof}

Let $E$ be a finite extension of $K$ such that $[K:\Qbb_p]=\#\Hom_{\Alg_{\Qbb_p}}(K,E)$, and set $\Sigma_K=\Hom_{\Alg_{\Qbb_p}}(K,E)$.
To prove Theorem \ref{thm:bad locus dimension upper bound}, we may replace $\LS_{\GL_n(K),\FF}$ with $\LS_{\GL_n(K),\FF}\times_{\Spa \Qbb_p} \Spa E$.
From now on, we will always work with the test category $\Aff_{E}$.

We outline the proof of Theorem \ref{thm:bad locus dimension upper bound}.
For integers $0<m<n$, let $P_{m,n-m}$ be the parabolic subgroup of $\GL_{n}$ of the form 
\[
\begin{pmatrix}
    \GL_m & * \\
    0 & \GL_{n-m}
\end{pmatrix}.
\]

Then $\LS_{P_{m,n-m}(K),\FF}$ in Definition \ref{defn:EG stack general algebriac group} is equivalent to the \'{e}tale stack over $\Aff_{E}$ that associates to each $A \in \Aff_{E}$ the groupoid of $G_K$-equivariant vector bundles $V$ on $X_{\Cbb_p,A}$ of rank $n$ equipped with an equivariant subbundle $W$ of rank $m$.

We will give an upper bound on the dimension of $\LS_{P_{m,n-m}(K),\FF}$ inductively, and by using this, we will give an upper bound on the dimension of the locus $Z\subset \LS_{\GL_n(K),\FF}$ in Theorem \ref{thm:bad locus dimension upper bound}.
More precisely, for a positive integer $n$, we will consider the following hypothesis:
\begin{description}
    \item[H$(n)$] Let $Z, Z'\subset \LS_{\GL_n(K),\FF}$ be the reduced Zariski closed substacks defined by the $0$-th Fitting ideals $\Fit_0(\Ext^2_{\FF,G_K}(V_n,V_n))$, $\Fit_0(H^2_{\FF,G_K}(V_n))$, respectively, where $V_n$ is the universal $G_K$-equivariant vector bundle on $X_{\Cbb_p}\times \LS_{\GL_n(K),\FF}$.
    Then we have $\dim Z, \dim Z'\leq dn^2-2$.
\end{description}
By the proof of ``Theorem \ref{thm:bad locus dimension upper bound} $\implies$ Theorem \ref{thm:EGH conjecture}'', Hypothesis \textbf{H}$(n)$ implies that the dimension of $\LS_{\GL_n(K),\FF}$ is $dn^2$.

\begin{theorem}\label{thm:rank 1}
	Hypothesis \textbf{H}$(1)$ is true.
\end{theorem}

\begin{theorem}\label{thm:parabolic dimension upper bound}
	Let $0<m<n$ be integers, and assume Hypothesis \textbf{H}$(l)$ for all positive integers $l<n$.
	Let $V_m\subset V_n$ be the universal $G_K$-equivariant vector bundles on $X_{\Cbb_p}\times \LS_{P_{m,n-m}(K),\FF}$.
	Let $Z, Z'\subset \LS_{P_{m,n-m}(K),\FF}$ be the reduced Zariski closed substacks defined by the $0$-th Fitting ideals $\Fit_0(\Ext^2_{\FF,G_K}(V_n,V_n))$, $\Fit_0(H^2_{\FF,G_K}(V_n))$, respectively. 
	Then we have $\dim Z, \dim Z'\leq dn^2-2$.
\end{theorem}

\begin{remark}
	This upper bound is far from optimal, but is enough for our purpose.
\end{remark}

\begin{remark}
	If $n\neq 2$ or $K\neq \Qbb_p$, then we can prove $\dim \LS_{P_{m,n-m}(K),\FF} \leq dn^2-2.$
\end{remark}

\begin{theorem}\label{thm:bad locus dimension induction}
	Let $n$ be a positive integer, and assume Hypothesis \textbf{H}$(m)$ for all positive integers $m<n$.
	Then Hypothesis \textbf{H}$(n)$ is also true.
\end{theorem}

We will prove these three theorems in the following subsections, which completes the proofs of Theorems \ref{thm:bad locus dimension upper bound} and \ref{thm:EGH conjecture}.

\subsection{The case $n=1$}\label{subsection:n=1}

Let $X^{\an}(K^{\times})$ denote the rigid analytic moduli space over $E$ of continuous characters of $K^{\times}$.
By \cite[Theorem 6.2.14]{KPX14}, we have a canonical equivalence
\[\LS_{\GL_1(K),\FF}\cong X^{\an}(K^{\times})/\Gbb_m^{\an},\]
where the action of $\Gbb_m^{\an}$ on $X^{\an}(K^{\times})$ is trivial.

First, let us consider the coherent sheaf $\Ext^2_{\FF,G_K}(V_1,V_1)$.
This is isomorphic to $H^2_{\FF,G_K}(\Ocal)$, where $\Ocal$ is the trivial $G_K$-equivariant line bundle on $X_{\Cbb_p}\times \LS_{\GL_1(K),\FF}$.
By \cite[Proposition 6.2.8 (2)]{KPX14}, we have $H^2_{\FF,G_K}(\Ocal)=0$, so the Zariski closed substack $Z\subset \LS_{\GL_1(K),\FF}$ defined by $\Fit_0(\Ext^2_{\FF,G_K}(V_1,V_1))$ is empty.

Next, let us consider the reduced Zariski closed substack $Z'\subset \LS_{\GL_1(K),\FF}$ defined by $\Fit_0(H^2_{\FF,G_K}(V_1))$.

\begin{definition}
	Let $\varepsilon\colon K^{\times}\to \Qbb_p^{\times}$ denote the character corresponding to the $p$-adic cyclotomic character under local class field theory.
	For $\sigma\in \Sigma_K$, let $x_{\sigma}\colon K^{\times}\to E^{\times}$ denote the character obtained from $\sigma\colon K\to E$.
	Let $X^{\an}(K^{\times})^{\irreg} \subset X^{\an}(K^{\times})$ denote the subspace parametrizing characters of the form $\varepsilon \prod_{\sigma\in \Sigma_K}x_{\sigma}^{k_{\sigma}}$ for $k_{\sigma}\in \Zbb_{\geq 0}$.
	Since this subspace is locally finite, $X^{\an}(K^{\times})^{\irreg} \subset X^{\an}(K^{\times})$ is a $0$-dimensional reduced Zariski closed subspace.
\end{definition}

\begin{proposition}\label{prop:irregular locus n=1}
	The reduced Zariski closed substack $Z'$ is equal to $X^{\an}(K^{\times})^{\irreg}/\Gbb_m^{\an}$ under the equivalence $\LS_{\GL_1(K),\FF}\cong X^{\an}(K^{\times})/\Gbb_m^{\an}$.
	In particular, we obtain $\dim Z'=-1$, which implies Theorem \ref{thm:rank 1}.
\end{proposition}
\begin{proof}
	This follows immediately from \cite[Proposition 6.2.8 (2)]{KPX14}.
\end{proof}

\subsection{Proof of Theorem \ref{thm:parabolic dimension upper bound}}
Let $x_{\infty}\colon \Spa \Cbb_p\to X_{\Cbb_p}$ denote the distinguished point, and let $t\Ocal_{X_{\Cbb_p}}$ denote the defining ideal of $x_{\infty}$.
Then $t\Ocal_{X_{\Cbb_p}}$ becomes a $G_K$-equivariant line bundle on $X_{\Cbb_p}$ corresponding to the character $\Nm_{K/\Qbb_p}\colon K^{\times}\to \Qbb_p^{\times}$.
For $\sigma\in \Sigma_K$, we define the divisor $x_{\infty,\sigma}$ of $X_{\Cbb_p,E}$ by
\[x_{\infty,\sigma} \colon \Spa (\Cbb_p\otimes_{K,\sigma}E)\to \Spa (\Cbb_p\otimes_{\Qbb_p}E)\xrightarrow{x_{\infty}\otimes \id} X_{\Cbb_p,E}.\]
Let $t_{\sigma}\Ocal_{X_{\Cbb_p,E}}$ denote the defining ideal of $x_{\infty,\sigma}$.
Then $t_{\sigma}\Ocal_{X_{\Cbb_p,E}}$ becomes a $G_K$-equivariant line bundle on $X_{\Cbb_p,E}$ corresponding to the character $x_{\sigma} \colon K^{\times}\to E^{\times}.$
By construction, we have $t\Ocal_{X_{\Cbb_p,E}}=\prod_{\sigma\in \Sigma_K}t_{\sigma}\Ocal_{X_{\Cbb_p,E}}.$

\begin{remark}
	The above construction is a translation of \cite[Notation 6.2.7]{KPX14} to our setting.
\end{remark}

Let $F$ be a finite extension of $E$.
By \cite[Theorem 7.11]{Ked16}, the Fargues-Fontaine curve $X_{\Cbb_p,F}$ is, locally, the adic spectrum of a principal ideal domain.
Therefore, we can define the saturation of a submodule of a vector bundle on $X_{\Cbb_p,F}$.

\begin{lemma}\label{lem:H^2 bound}
	Let $V$ be a $G_K$-equivariant vector bundle on $X_{\Cbb_p,F}$ of rank $n$.
	Then we have $\dim_F H^0_{\FF,G_K}(V), \dim_F H^2_{\FF,G_K}(V)\leq n$.
\end{lemma}
\begin{proof}
	By Tate local duality (\cite[Theorem 1.2]{Liu08}), it is enough to show that $\dim_F H^0_{\FF,G_K}(V)\leq n$.
	We proceed by induction on $n$.
	When $n=1$, the claim follows from \cite[Proposition 6.2.8 (1)]{KPX14}.
	Assume $n>1$ and $H^0_{\FF,G_K}(V)\neq 0$.
	Let $f\colon \Ocal_{X_{\Cbb_p,F}}\to V$ be a non-zero morphism, and let $W$ be the saturation of $f(\Ocal_{X_{\Cbb_p,F}})$ in $V$.
	Then, there is an exact sequence of $G_K$-equivariant vector bundles
	\[0\to W \to V \to V/W \to 0.\]
	By the induction hypothesis, we have $\dim_F H^0_{\FF,G_K}(W)\leq 1$ and $\dim_F H^0_{\FF,G_K}(V/W)\leq n-1$, which implies $\dim_F H^0_{\FF,G_K}(V)\leq n$.
\end{proof}

\begin{definition}
	Let $V$ be a non-zero $G_K$-equivariant vector bundle on $X_{\Cbb_p,F}$.
	We say that $V$ is \textit{absolutely irreducible} (resp.\ \textit{absolutely non-trianguline}) if for any finite extension $F'/F$, the $G_K$-equivariant vector bundle $V\otimes_F F'$ on $X_{\Cbb_p,F'}$ has no non-zero proper subbundle (resp.\ is non-trianguline).
\end{definition}

\begin{lemma}\label{lem:submodule of the same rank t-torsion}
	Let $V$ be a $G_K$-equivariant vector bundle on $X_{\Cbb_p,F}$ and $W\subset V$ be an equivariant submodule of $V$ whose saturation is $V$.
	Then the support of $V/W$ lies in $V(t)=x_{\infty}\times \Spa F\subset X_{\Cbb_p,F}$.
\end{lemma}
\begin{proof}
	Since the saturation of $W$ is $V$, the module $V/W$ is torsion.
	Then the claim follows from \cite[Proposition 10.1.1]{FF18} and \cite[Proposition II.2.9]{FS24}.
\end{proof}

\begin{definition}
	Let $\Ocal_{X_{\Cbb_p,F}}[1/t]$ denote the sheaf of rings $\varinjlim_{n}t^{-n}\Ocal_{X_{\Cbb_p,F}}$ on $X_{\Cbb_p,F}$.
	Similarly, for a $G_K$-equivariant vector bundle $V$ on $X_{\Cbb_p,F}$, we define a $G_K$-equivariant $\Ocal_{X_{\Cbb_p,F}}[1/t]$-module $V[1/t]$ as $\varinjlim_{n}t^{-n} V$.
	For $G_K$-equivariant vector bundles $V$ and $W$ on $X_{\Cbb_p,F}$, let $\Hom_{\FF,G_K}(V[1/t],W[1/t])$ denote the $F$-vector space of $G_K$-equivariant $\Ocal_{X_{\Cbb_p,F}}[1/t]$-linear morphisms from $V[1/t]$ to $W[1/t]$.
\end{definition}

\begin{remark}\label{rem:add 1/t injective}
	By definition, for $G_K$-equivariant vector bundles $V$ and $W$ on $X_{\Cbb_p,F}$, the natural morphism 
	\[\Hom_{\FF,G_K}(V,W)\to \Hom_{\FF,G_K}(V[1/t],W[1/t])\]
	is injective, and the natural morphism
	\[\varinjlim_{n\geq 0}\Hom_{\FF,G_K}(V,t^{-n}W)\to \Hom_{\FF,G_K}(V[1/t],W[1/t])\]
	is an isomorphism.
\end{remark}

\begin{example}\label{ex:killing t-torsion}
	Let $V$ be a $G_K$-equivariant vector bundle on $X_{\Cbb_p,F}$ and $W\subset V$ be a submodule of $V$ whose saturation is $V$.
	By Lemma \ref{lem:submodule of the same rank t-torsion}, we have $V[1/t]\cong W[1/t]$.
\end{example}

\begin{lemma}\label{lem:[1/t] finiteness}
	Let $V$ and $W$ be $G_K$-equivariant vector bundles on $X_{\Cbb_p,F}$.
	Then $\Hom_{\FF,G_K}(V[1/t],W[1/t])$ is a finite-dimensional $F$-vector space.
\end{lemma}
\begin{proof}
Since the transition morphisms of the filtered colimit $\varinjlim_{n\geq 0}\Hom_{\FF,G_K}(V,t^{-n}W)$ are injective and we have $\varinjlim_{n\geq 0}\Hom_{\FF,G_K}(V,t^{-n}W)\cong \Hom_{\FF,G_K}(V[1/t],W[1/t])$, it suffices to show that the dimensions $\dim_F\Hom_{\FF,G_K}(V,t^{-n}W)$ are uniformly bounded independently of $n$.
By applying Lemma \ref{lem:H^2 bound} to $\intHom(V,t^{-n}W)$, we obtain $\dim_F\Hom_{\FF,G_K}(V,t^{-n}W)\leq \rank (V)\cdot \rank (t^{-n} W)=\rank (V)\cdot \rank (W)$.
This completes the proof.
\end{proof}

\begin{lemma}\label{lem:Schur}
	Let $V$ be an absolutely irreducible $G_K$-equivariant vector bundle on $X_{\Cbb_p,F}$.
	Then $\dim_F\End_{\FF,G_K}(V[1/t])=1$.
\end{lemma}
\begin{proof}
	Let $f\colon V[1/t]\to V[1/t]$ be a non-zero morphism.
	Then $\Ker f \cap V$ is a proper subbundle of $V$.
	Since $V$ is absolutely irreducible, we obtain $\Ker f \cap V=0$, which implies $\Ker f=0$.
	Let $n\geq 0$ be an integer such that $f(V)\subset t^{-n} V$.
	Since $f(V)\neq 0$ and $t^{-n} V$ is absolutely irreducible, the saturation of $f(V)$ is $t^{-n} V$.
	By Lemma \ref{lem:submodule of the same rank t-torsion}, $t^{-n}V/f(V)$ is $t$-torsion, which implies that $f\colon V[1/t]\to V[1/t]$ is surjective.
	Thus, $f\colon V[1/t]\to V[1/t]$ is an isomorphism.
	By combining the above argument with Lemma \ref{lem:[1/t] finiteness}, $\End_{\FF,G_K}(V[1/t])$ is a finite division $F$-algebra.
	Since $V$ is absolutely irreducible, for any finite extension $F'/F$, $\End_{\FF,G_K}(V[1/t])\otimes_F F'$ is also a finite division $F'$-algebra, which implies $\End_{\FF,G_K}(V[1/t])\cong F$.
\end{proof}

\begin{lemma}\label{lemma:abs irr Hom space dimension}
	Let $V$ and $W$ be $G_K$-equivariant vector bundles on $X_{\Cbb_p,F}$ of rank $m$ and $n$, respectively.
	Assume that $V$ is absolutely irreducible.
	Then we have $\dim_F \Hom_{\FF,G_K}(V,W), \dim_F \Hom_{\FF,G_K}(W,V)\leq \left\lfloor \frac{n}{m}\right\rfloor$.
\end{lemma}
\begin{proof}
	The two bounds can be proved in a similar way, so we only show that $\dim_F \Hom_{\FF,G_K}(V,W)\leq \left\lfloor \frac{n}{m}\right\rfloor$.
	We proceed by induction on $n$.
	Assume that $\Hom_{\FF,G_K}(V,W)\neq 0$, and let $f\colon V\to W$ be a non-zero morphism.
	Since $V$ is absolutely irreducible and $\Ker f$ is a proper subbundle of $V$, we obtain $\Ker f=0$.
	Let $U\subset W$ be the saturation of $\Im f$. 
	Then there exists an exact sequence of $G_K$-equivariant vector bundles on $X_{\Cbb_p,F}$
	\[0\to U\to W\to W/U\to 0.\]
	By the induction hypothesis, we have $\dim_F \Hom_{\FF,G_K}(V,W/U)\leq \left\lfloor \frac{n-m}{m}\right\rfloor$.
	It remains to prove $\dim_F \Hom_{\FF,G_K}(V,U)\leq 1$.
	By Example \ref{ex:killing t-torsion}, we have $V[1/t]\cong U[1/t]$.
	Therefore, by Lemma \ref{lem:Schur}, we obtain $\dim_F \Hom_{\FF,G_K}(V[1/t],U[1/t])=1$.
	Hence, by Remark \ref{rem:add 1/t injective}, we obtain $\dim_F \Hom_{\FF,G_K}(V,U)\leq 1$.
\end{proof}

\begin{corollary}\label{cor:abs irr Ext^2 space dimension}
	Let $V$ and $W$ be $G_K$-equivariant vector bundles on $X_{\Cbb_p,F}$ of rank $m$ and $n$, respectively.
	Assume that $V$ is absolutely irreducible.
	Then we have $\dim_F \Ext^2_{\FF,G_K}(V,W), \dim_F \Ext^2_{\FF,G_K}(W,V)\leq \left\lfloor \frac{n}{m}\right\rfloor$.
	In particular, if $m\geq 2$, then we have $\dim_F \Ext^2_{\FF,G_K}(V,W), \dim_F \Ext^2_{\FF,G_K}(W,V)\leq mn-2$.
\end{corollary}
\begin{proof}
	This follows immediately from Lemma \ref{lemma:abs irr Hom space dimension} and Tate local duality.
\end{proof}

\begin{corollary}\label{cor:abs non-trianguline Ext^2 space dimension}
	Let $V$ and $W$ be $G_K$-equivariant vector bundles on $X_{\Cbb_p,F}$ of rank $m$ and $n$, respectively.
	Assume that $V$ is absolutely non-trianguline.
	Then we have $\dim_F \Ext^2_{\FF,G_K}(V,W), \dim_F \Ext^2_{\FF,G_K}(W,V)\leq mn-2$.
\end{corollary}
\begin{proof}
	The two bounds can be proved in a similar manner, so we only show that $\dim_F \Ext^2_{\FF,G_K}(V,W)\leq mn-2$.
	We proceed by induction on $m$.
	Since $V$ is absolutely non-trianguline, we have $m\geq 2$.
	If $V$ is absolutely irreducible, then the claim follows from Corollary \ref{cor:abs irr Ext^2 space dimension}.
	Assume that $V$ is not absolutely irreducible.
	By replacing $F$ with a sufficiently large finite extension of $F$, we can take a non-zero proper subbundle $V_1\subset V$.
	Since $V$ is absolutely non-trianguline, so is either $V_1$ or $V/V_1$.
	Since we have $\dim_F \Ext^2_{\FF,G_K}(V,W) \leq \dim_F \Ext^2_{\FF,G_K}(V_1,W)+\dim_F \Ext^2_{\FF,G_K}(V/V_1,W)$, the claim follows from the induction hypothesis and Lemma \ref{lem:H^2 bound}.
\end{proof}

\begin{proof}[Proof of Theorem \ref{thm:parabolic dimension upper bound}]
	We set $m_1=m$ and $m_2=n-m$.
	Let 
	\[f\colon \LS_{P_{m_1,m_2}(K),\FF}\to \LS_{\GL_{m_1}(K),\FF}\times \LS_{\GL_{m_2}(K),\FF};\; (V_{m_1}\subset V_n)\mapsto (V_{m_1}, V_n/V_{m_1}).\]
	For a classical point $x\colon \Spa F\to \LS_{\GL_{m_1}(K),\FF}\times \LS_{\GL_{m_2}(K),\FF}$, the fiber $f^{-1}(x)$ is equivalent to
	\[[\Ext^1_{\FF,G_K}(V_{m_2,x},V_{m_1,x})/\Hom_{\FF,G_K}(V_{m_2,x},V_{m_1,x})],\]
	where the action of $\Hom_{\FF,G_K}(V_{m_2,x},V_{m_1,x})$ is trivial.
	Therefore, by the Euler characteristic formula (\cite[Theorem 1.2]{Liu08}), we have 
	\begin{align}\label{eq:fiber dimension}
	\dim f^{-1}(x)=m_1m_2d+\dim_F \Ext^2_{\FF,G_K}(V_{m_2,x},V_{m_1,x}).
    \end{align}

	First, we assume that $n>2$ and prove that $\dim \LS_{P_{m_1,m_2}(K),\FF}\leq dn^2-2$.
	In this case, we have $m_1>1$ or $m_2>1$.
	Both cases can be proved in a similar way, so we only consider the case of $m_1>1$.
	Let $Z_{m_1m_2-2}\subset \LS_{\GL_{m_1}(K),\FF}\times \LS_{\GL_{m_2}(K),\FF}$ denote the Zariski closed substack defined by the $(m_1m_2-2)$-th Fitting ideal $\Fit_{m_1m_2-2}(\Ext^2_{\FF,G_K}(V_{m_2},V_{m_1}))$, where $(V_{m_1},V_{m_2})$ is the universal pair of $G_K$-equivariant vector bundles on $\LS_{\GL_{m_1}(K),\FF}\times \LS_{\GL_{m_2}(K),\FF}$.
	Let $Z_{m_1m_2-2}^{c}$ denote the open complement of $Z_{m_1m_2-2}$.
	It suffices to show that 
	\[\dim f^{-1}(Z_{m_1m_2-2}), \dim f^{-1}(Z_{m_1m_2-2}^c)\leq dn^2-2.\]
	For any classical point $x\in Z_{m_1m_2-2}^c$, by definition, we have $\dim_F\Ext^2_{\FF,G_K}(V_{m_2,x},V_{m_1,x})\leq m_1m_2-2$.
	Therefore, we have 
	\begin{align*}
	\dim f^{-1}(x)&=m_1m_2d+\dim_F \Ext^2_{\FF,G_K}(V_{m_2,x},V_{m_1,x})\\
	&\leq m_1m_2d+m_1m_2-2\\
	&\leq 2m_1m_2d-2.	
	\end{align*}
	By Hypotheses \textbf{H}$(m_1)$ and \textbf{H}$(m_2)$, we have
	\begin{align*}
		\dim Z_{m_1m_2-2}^c \leq \dim (\LS_{\GL_{m_1}(K),\FF}\times \LS_{\GL_{m_2}(K),\FF})=d(m_1^2+m_2^2).
	\end{align*}
	Therefore, by Lemma \ref{lem:stack fiber dimension}, we obtain
	\begin{align*}
	\dim f^{-1}(Z_{m_1m_2-2}^c) &\leq \dim Z_{m_1m_2-2}^c+\sup_{x\in (Z_{m_1m_2-2}^c)^{\cl}}\dim f^{-1}(x)\\
	&\leq d(m_1^2+m_2^2)+2m_1m_2d-2\\
	&\leq dn^2-2.
	\end{align*}
	Next, we prove $\dim f^{-1}(Z_{m_1m_2-2})\leq dn^2-2$.
	By \eqref{eq:fiber dimension} and Lemma \ref{lem:H^2 bound}, we have 
	\[\sup_{x\in (Z_{m_1m_2-2})^{\cl}}\dim f^{-1}(x) \leq m_1m_2d+m_1m_2\leq 2m_1m_2d.\]
	Therefore, it suffices to show the following claim. 
	For later use, we also consider the case $n=2$.
	\begin{description}
		\item[Claim] We have $\dim Z_{m_1m_2-2}\leq d(m_1^2+m_2^2)-2$.
		When $n=2$ and $m_1=m_2=1$, let $\tilde{Z}\subset \LS_{\GL_1(K),\FF}\times\LS_{\GL_1(K),\FF}$ denote the Zariski closed substack defined by the $0$-th Fitting ideal $\Fit_{0}(\Ext^2_{\FF,G_K}(L_1,L_2))$, where $(L_1,L_2)$ is the universal pair of $G_K$-equivariant line bundles on $\LS_{\GL_1(K),\FF}\times \LS_{\GL_1(K),\FF}$.
		Then we have $\dim\tilde{Z}\leq 2d-2$.
	\end{description}
	Both claims can be proved in the same way, so we only prove the former one.
	Let 
	\[g\colon Z_{m_1m_2-2}\subset \LS_{\GL_{m_1}(K),\FF}\times \LS_{\GL_{m_2}(K),\FF}\to \LS_{\GL_{m_1}(K),\FF}\] 
	denote the restriction of the first projection. 
	Let $x\colon \Spa F\to \LS_{\GL_{m_1}(K),\FF}$ be a classical point.
	If $V_{m_1,x}$ is absolutely non-trianguline, then the fiber $g^{-1}(x)$ is empty by Corollary \ref{cor:abs non-trianguline Ext^2 space dimension}.
	Assume that $V_{m_1,x}$ is not absolutely non-trianguline.
	By replacing $F$ with a sufficiently large finite extension of $F$, we may assume that $V_{m_1,x}$ is trianguline.
	Fix a triangulation of $V_{m_1,x}$, and let $\delta_1,\ldots, \delta_{m_1}$ be the characters appearing in the graded pieces.
	For $i=1,\ldots, m_1$, let $Z'_i\subset \Spa F\times \LS_{\GL_{m_2}(K),\FF}$ be the reduced Zariski closed substack defined by the $0$-th Fitting ideal $\Fit_0(\Ext^2_{\FF,G_K}(V_{m_2},\Ocal(\delta_i)))$.
	For a classical point $y\in g^{-1}(x)^{\cl}$, if $y\notin Z'_i$ for any $i$, then we have $\Ext^2_{\FF,G_K}(V_{m_2,y},\Ocal(\delta_i))=0$, which implies $\Ext^2_{\FF,G_K}(V_{m_2,y},V_{m_1,x})=0$.
	Therefore, we obtain
	\[g^{-1}(x)^{\red}\subset Z'_1\cup \cdots \cup Z'_{m_1}.\]
	By Hypothesis \textbf{H}$(m_2)$ (and considering the isomorphism on $\LS_{\GL_{m_2}(K),\FF}$ defined by taking the dual and twisting by a character), we obtain $\dim Z'_i \leq dm_2^2-2$.
	Therefore, we obtain $\dim g^{-1}(x)\leq dm_2^2-2$.
	Consequently, by Hypothesis \textbf{H}$(m_1)$, we obtain
	\begin{align*}
		\dim Z_{m_1m_2-2}&\leq \dim \LS_{\GL_{m_1}(K),\FF} +\sup_{x\in (\LS_{\GL_{m_1}(K),\FF})^{\cl}} \dim g^{-1}(x)\\
		&\leq dm_1^2+dm_2^2-2,
	\end{align*}
	which completes the proof in the case of $n>2$.

	Finally, we consider the case of $n=2$.
	Let $Z, Z'\subset \LS_{P_{1,1}(K),\FF}$ be the reduced Zariski closed substacks defined by the $0$-th Fitting ideals $\Fit_0(\Ext^2_{\FF,G_K}(V_2,V_2))$, $\Fit_0(H^2_{\FF,G_K}(V_2))$, respectively. 
	Let $(L_1,L_2)$ be the universal pair of $G_K$-equivariant line bundles on $\LS_{\GL_{1}(K),\FF}\times \LS_{\GL_{1}(K),\FF}$, and let $\tilde{Z}, \tilde{Z}'\subset \LS_{\GL_1(K),\FF}\times \LS_{\GL_1(K),\FF}$ be the reduced Zariski closed substacks defined by ideals $\Fit_0(\Ext^2_{\FF,G_K}(L_1,L_2)) \cap \Fit_0(\Ext^2_{\FF,G_K}(L_2,L_1))$, $\Fit_0(H^2_{\FF,G_K}(L_1))\cap \Fit_0(H^2_{\FF,G_K}(L_2))$, respectively.
	Let us prove that $Z\subset f^{-1}(\tilde{Z})$ and $Z'\subset f^{-1}(\tilde{Z}')$. 
	Both claims can be proved in a similar way, so we only prove the former one.
	Let $x\colon \Spa F\to Z$ be a classical point.
	Let $(L_{1,x},L_{2,x})$ be the pair of $G_K$-equivariant line bundles corresponding to $f(x)$.
	Then, by definition, we have an exact sequence
	\[0\to L_{1,x}\to V_{2,x} \to L_{2,x}\to 0.\]
	Since $x$ is a classical point of $Z$, $\Ext^2_{\FF,G_K}(V_{2,x},V_{2,x})\neq 0$.
	By \cite[Proposition 6.2.8 (2)]{KPX14}, we have $\Ext^2_{\FF,G_K}(L_{i,x},L_{i,x})= 0$ for $i=1,2$.
	Therefore, $\Ext^2_{\FF,G_K}(V_{2,x},V_{2,x})\neq 0$ implies that either $\Ext^2_{\FF,G_K}(L_{1,x},L_{2,x})$ or $\Ext^2_{\FF,G_K}(L_{2,x},L_{1,x})$ is non-zero.
	In other words, we have $f(x)\in \tilde{Z}$.
	Thus, we obtain $Z\subset f^{-1}(\tilde{Z})$.

	We prove that $\dim \tilde{Z}, \dim \tilde{Z}'\leq 2d-2$.
	Note that by \eqref{eq:fiber dimension} and Lemma \ref{lem:H^2 bound}, these give the desired upper bound $\dim Z, \dim Z' \leq (2d-2)+(d+1)\leq 4d-2$.
	The inequality $\dim \tilde{Z}\leq 2d-2$ was already proved in Claim in the proof of the case $n>2$.
	By definition, we have 
	\[\tilde{Z}'= ((X^{\an}(K^{\times})^{\irreg}/\Gbb_m^{\an}) \times \LS_{\GL_1(K),\FF})\cup (\LS_{\GL_1(K),\FF}\times (X^{\an}(K^{\times})^{\irreg}/\Gbb_m^{\an})),\]
	where we use the same notation as in Subsection \ref{subsection:n=1}.
	Therefore, by Proposition \ref{prop:irregular locus n=1}, we obtain $\dim \tilde{Z}'=d-1\leq 2d-2$.
\end{proof}

\subsection{Proof of Theorem \ref{thm:bad locus dimension induction}}
Let $n\geq 2$ be an integer.
In this subsection, we always assume the hypothesis \textbf{H}$(m)$ for all integers $m<n$.
We use the same notation as in the previous subsection.

Before the proof, we prove the following lemmas.
For $\kbf=(k_{\sigma})_{\sigma\in \Sigma_K} \in \Zbb_{\geq 0}^{\Sigma_K}$, we write $\tbf^{-\kbf}=\prod_{\sigma\in \Sigma_K}t_{\sigma}^{-k_{\sigma}}$.

\begin{lemma}\label{lem:k large cohomology vanishing}
	Let $F$ be a finite extension of $E$, and let $V$ be a $G_K$-equivariant vector bundle on $X_{\Cbb_p,F}$.
	Then for a sequence $\kbf_1<\kbf_2<\cdots$ in $\Zbb_{\geq 0}^{\Sigma_K}$, we have $H^2_{\FF,G_K}(\tbf^{-\kbf_i}V)=0$ for $i$ sufficiently large.
\end{lemma}
\begin{proof}
	By Tate local duality, it suffices to show that $H^0_{\FF,G_K}(\tbf^{\kbf_i}V^{\vee}(1))=0$ for $i$ sufficiently large.
	Since for some $\sigma\in \Sigma_K$, the sequence $k_{1,\sigma}\leq k_{2,\sigma}\leq \cdots$ tends to $\infty$, we obtain $0=\bigcap_{i} H^0(X_{\Cbb_p,F},\tbf^{\kbf_i}V^{\vee}(1))$, and thus $0=\bigcap_{i} H^0_{\FF,G_K}(\tbf^{\kbf_i}V^{\vee}(1))$.
	Since $H^0_{\FF,G_K}(V^{\vee}(1))$ is finite-dimensional, we get the claim.
\end{proof}

\begin{lemma}\label{lem:quotient vector bundle}
	Let $A$ be a reduced affinoid $\Qbb_p$-algebra, and let $f\colon V\to W$ be a morphism of $G_K$-equivariant vector bundles on $X_{\Cbb_p,A}$. 
	Let $n$ be a non-negative integer.
	Assume that for any classical point $x\in \Spa A$, $\Coker(f_x\colon V_x\to W_x)$ is a ($G_K$-equivariant) vector bundle of rank $n$ on $X_{\Cbb_p,k(x)}$, where $f_x\colon V_x\to W_x$ is the restriction of $f\colon V\to W$ to $X_{\Cbb_p,k(x)}$.
	Then $\Coker(f\colon V\to W)$ is also a vector bundle of rank $n$ on $X_{\Cbb_p,A}$.
\end{lemma}
\begin{proof}
	Let $\tilde{B}_{\Cbb_p}^{[1,p]}$ and $B_{K,m}^{[1,p]}$ be as in \cite[Sections 3.1, 3.2]{Mikami24}.\footnote{In \cite{Mikami24}, $\tilde{B}_{\Cbb_p}^{[1,p]}$ is written as $\tilde{B}_{\bar{K}}^{[1,p]}$.}
	Note that the natural morphism $\Spa(\tilde{B}_{\Cbb_p}^{[1,p]})\to X_{\Cbb_p}$ is surjective and a local isomorphism.
	Let $V^{[1,p]}$ and $W^{[1,p]}$ be the pullbacks of $V$ and $W$ to $\Spa \tilde{B}_{\Cbb_p,A}^{[1,p]}=\Spa(\tilde{B}_{\Cbb_p}^{[1,p]}\hotimes A)$, respectively.
	Set $M^{[1,p]}=\Coker(f\colon V^{[1,p]}\to W^{[1,p]})$.
	It suffices to show that $M^{[1,p]}$ is a finite projective $\tilde{B}_{\Cbb_p,A}^{[1,p]}$-module of rank $n$.
	By \cite[Theorems 3.12 and 3.37]{Mikami24}, for a sufficiently large integer $m$, there exists a morphism $f_m\colon V_m^{[1,p]}\to W_m^{[1,p]}$ of finite projective $B_{K,m,A}^{[1,p]}$-modules equipped with an isomorphism of morphisms
	\[(f_m\colon V_m^{[1,p]}\to W_m^{[1,p]})\otimes_{B_{K,m,A}^{[1,p]}}\tilde{B}_{\Cbb_p,A}^{[1,p]} \cong (f\colon V^{[1,p]}\to W^{[1,p]}).\]
	Set $M_{m}^{[1,p]}=\Coker(f_{m})$.
	We need to show that $M_{m}^{[1,p]}$ is a finite projective $B_{K,m,A}^{[1,p]}$-module of rank $n$.
	By assumption, for any maximal ideal $\mfrak\subset A$, 
	\[M_{m}^{[1,p]}/\mfrak\otimes_{B_{K,m,A/\mfrak}^{[1,p]}}\tilde{B}_{\Cbb_p,A/\mfrak}^{[1,p]}\cong M^{[1,p]}/\mfrak\] 
	is a finite projective $\tilde{B}_{\Cbb_p,A/\mfrak}^{[1,p]}$-module of rank $n$.
	Therefore, by faithfully flat descent (\cite[Propositions 1.21, 3.13, and Lemma 3.39]{Mikami24}), $M_{m}^{[1,p]}/\mfrak$ is a finite projective $B_{K,m,A/\mfrak}^{[1,p]}$-module of rank $n$.
	Note that $A\to B_{K,m,A}^{[1,p]}$ is a smooth morphism of affinoid $\Qbb_p$-algebras.
	In particular, $B_{K,m,A}^{[1,p]}$ is Jacobson and reduced as $A$ is also reduced.
	Moreover, for any maximal ideal $\nfrak\subset B_{K,m,A}^{[1,p]}$, $\nfrak\cap A$ is a maximal ideal of $A$.
	Therefore, for every maximal ideal $\nfrak\subset B_{K,m,A}^{[1,p]}$, $M_{m}^{[1,p]}/\nfrak$ is a finite projective $B_{K,m,A}^{[1,p]}/\nfrak$-module of rank $n$, which implies that $M_{m}^{[1,p]}$ is a finite projective $B_{K,m,A}^{[1,p]}$-module of rank $n$.
\end{proof} 

Let $Z, Z'\subset \LS_{\GL_n(K),\FF}$ be the reduced Zariski closed substacks defined by the $0$-th Fitting ideals $\Fit_0(\Ext^2_{\FF,G_K}(V_n,V_n))$, $\Fit_0(H^2_{\FF,G_K}(V_n))$, respectively, where $V_n$ is the universal $G_K$-equivariant vector bundle on $X_{\Cbb_p}\times \LS_{\GL_n(K),\FF}$.

\begin{proof}[Proof of $\dim Z'\leq dn^2-2$.]
	For $\kbf\in \Zbb_{\geq 0}^{\Sigma_K}$, let $Z'_{\kbf}\subset Z'$ be the reduced Zariski closed substack defined by the $0$-th Fitting ideal $\Fit_0(H^2_{\FF,G_K}(\tbf^{-\kbf}V_n))$.
	Set $Z'^{\circ}_{\kbf}=Z'_{\kbf}\setminus \bigcup_{\kbf'>\kbf}Z'_{\kbf'}$.
	By Lemma \ref{lem:k large cohomology vanishing} and Example \ref{ex:locally Zariski closed partition rigid stack}, $\{Z'^{\circ}_{\kbf}\}_{\kbf\in \Zbb_{\geq 0}^{\Sigma_K}}$ is a locally Zariski closed covering of $Z'$.
	Therefore, by Lemma \ref{lem:locally Zariski closed partition dimension}, it suffices to show that $\dim Z'^{\circ}_{\kbf}\leq dn^2-2$.
	For $i\geq 0$, let $Z'^{\circ}_{\kbf,i}\subset Z'^{\circ}_{\kbf}$ be the reduced Zariski closed substack defined by the $i$-th Fitting ideal $\Fit_i(H^2_{\FF,G_K}(\tbf^{-\kbf}V_n))$, and set $Z'^{\circ\circ}_{\kbf,i}=Z'^{\circ}_{\kbf,i-1}\setminus Z'^{\circ}_{\kbf,i}$ for $i\geq 1$.
	Then, by Example \ref{ex:Fitting ideal locally Zariski closed partition}, $\bigsqcup_{i\geq 1}Z'^{\circ\circ}_{\kbf,i}$ is a locally Zariski closed partition of $Z'^{\circ}_{\kbf}$.
	Therefore, it suffices to show that $\dim Z'^{\circ\circ}_{\kbf,i}\leq dn^2-2$.
	To lighten the notation, we write $Y_{\kbf,i}=Z'^{\circ\circ}_{\kbf,i}$.
	Summarizing the above construction, the universal $G_K$-equivariant vector bundle $V_n$ on $X_{\Cbb_p}\times Y_{\kbf,i}$ satisfies the following properties:
	\begin{enumerate}
		\item $H^2_{\FF,G_K}(\tbf^{-\kbf}V_n)$ is a vector bundle of rank $i$ on $Y_{\kbf,i}$, which implies that $\Hom_{\FF,G_K}(V_n,\tbf^{\kbf}\Ocal(1))$ is also a vector bundle of rank $i$ on $Y_{\kbf,i}$ by Tate local duality;
		\item\label{maximal k} For any classical point $x\colon \Spa F\to Y_{\kbf,i}$ and any $\kbf'>\kbf$, $H^2_{\FF,G_K}(\tbf^{-\kbf'}V_{n,x})=0$, which implies that $\Hom_{\FF,G_K}(V_{n,x},\tbf^{\kbf'}\Ocal(1))=0$ by Tate local duality. 
	\end{enumerate}
	Let $\Spa A\to Y_{\kbf,i}$ be a smooth morphism of pure dimension $m$.
	It suffices to show that $\dim A\leq dn^2-2+m$.
	Since $Y_{\kbf,i}$ is reduced by definition, $A$ is also reduced.
	By taking an affinoid open cover of $\Spa A$, we may assume that $\Hom_{\FF,G_K}(V_{n,A},\tbf^{\kbf}\Ocal(1))$ is a finite free $A$-module of rank $i$.
	Let $\bar{Y}_{\kbf,i,A}$ be the fiber product of the natural morphism $q\colon \LS_{P_{n-1,1}(K),\FF}\to \LS_{\GL_n(K),\FF}$ and $\Spa A\to Y_{\kbf,i}\subset \LS_{\GL_n(K),\FF}$.
	Then, by Theorem \ref{thm:EG stack representable for general G}, $\bar{Y}_{\kbf,i,A}$ is a rigid space.
	Moreover, we have 
	\[\dim \bar{Y}_{\kbf,i,A} \leq m+\dim q^{-1}(Z') \leq dn^2-2+m\]
	by Theorem \ref{thm:parabolic dimension upper bound}.
	Therefore, by Lemma \ref{lem:section dimension}, it is enough to show that the projection $\bar{Y}_{\kbf,i,A} \to\Spa A$ has a section.
	Equivalently, it suffices to construct a subbundle $W$ of $V_{n,A}$ on $X_{\Cbb_p}\times \Spa A$ of codimension $1$.
	Let $f\colon V_{n,A}\to \tbf^{\kbf}\Ocal(1)$ be an element of a basis of the finite free $A$-module $\Hom_{\FF,G_K}(V_{n,A},\tbf^{\kbf}\Ocal(1))$.
	For any classical point $x\in \Spa A$, the morphism $f_x\colon V_{n,x}\to \tbf^{\kbf}\Ocal(1)$ is non-zero and does not factor through $\tbf^{\kbf'}\Ocal(1)$ for any $\kbf'>\kbf$ by Property \eqref{maximal k}.
	By \cite[Corollary 6.2.9]{KPX14}, this implies that $f_x\colon V_{n,x}\to \tbf^{\kbf}\Ocal(1)$ is surjective.
	Therefore, by Lemma \ref{lem:quotient vector bundle}, $f\colon V_{n,A}\to \tbf^{\kbf}\Ocal(1)$ is also surjective, and thus, $W=\Ker(f)$ is the desired subbundle.
\end{proof}

\begin{lemma}\label{lem:Plucker relation}
	Let $m$ be an integer with $0<m<n$, and let $\Spa A\to \LS_{\GL_n(K),\FF}$ be a morphism corresponding to a $G_K$-equivariant vector bundle $V_{n,A}$ on $X_{\Cbb_p}\times \Spa A$.
	Let $\Pbb(\wedge^m V_{n,A})$ be a rigid space over $\Spa A$ classifying $G_K$-equivariant subbundles of $\wedge^m V_{n,A}$ of codimension $1$.
	Then the natural morphism 
	\[\LS_{P_{n-m,m}(K),\FF}\times_{\LS_{\GL_n(K),\FF}} \Spa A\to \Pbb(\wedge^m V_{n,A})\]
	defined by taking $m$-th wedge products is a closed immersion.
\end{lemma}
\begin{proof}
	The above morphism is isomorphic to the closed immersion defined by the Pl\"{u}cker relations.
\end{proof}

\begin{proof}[Proof of $\dim Z\leq dn^2-2$.]
	For $i\geq 0$, let $Z_i\subset Z$ be a reduced Zariski closed substack defined by the $i$-th Fitting ideal $\Fit_i(\Ext^2_{\FF,G_K}(V_n,V_n))$, and set $Z_i^{\circ}=Z_{i-1}\setminus Z_{i}$ for $i\geq 1$.
	Then, by Example \ref{ex:Fitting ideal locally Zariski closed partition}, $\bigsqcup_{i\geq 1}Z_i^{\circ}$ is a locally Zariski closed partition of $Z$.
	Therefore, by Lemma \ref{lem:locally Zariski closed partition dimension}, it suffices to show that $\dim Z_i^{\circ}\leq dn^2-2$.
	Then $\Ext^2_{\FF,G_K}(V_n,V_n)$ over $Z_i^{\circ}$ is a vector bundle of rank $i$.
	By replacing $Z_i^{\circ}$ with a reduced substack of the same dimension as in Corollary \ref{cor:coherent sheaf to vector bundle}, we may assume that $\Ext^2_{\FF,G_K}((\wedge^j V_n)(j-1),\wedge^j V_n)$ is also a vector bundle on $Z_i^{\circ}$ for $j=1,\ldots,n$.
	By Tate local duality, $\Hom_{\FF,G_K}(\wedge^j V_n,(\wedge^j V_n)(j))$ is  also a vector bundle on $Z_i^{\circ}$ for $j=1,\ldots,n$.
	Let $\tilde{Z}_i^{\circ}$ be a rigid analytic Artin stack over $Z_i^{\circ}$ classifying surjections $\Ext^2_{\FF,G_K}(V_n,V_n)\to \Ocal$.
	Then $\tilde{Z}_i^{\circ}\to Z_i^{\circ}$ is smooth of pure dimension $i$.
	Therefore, it suffices to show that $\dim \tilde{Z}_i^{\circ}\leq dn^2-2+i$.
	Let $f^{\vee}\colon \Ext^2_{\FF,G_K}(V_n,V_n)\to \Ocal_{\tilde{Z}_i^{\circ}}$ be the universal morphism on $\tilde{Z}_i^{\circ}$.
	By Tate local duality, the dual of the universal morphism $f^{\vee}$ is isomorphic to $f\colon \Ocal_{\tilde{Z}_i^{\circ}}\to \Hom_{\FF,G_K}(V_n,V_n(1))$.
	Equivalently, we have a universal morphism $f\colon V_n\to V_n(1)$ on $X_{\Cbb_p}\times \tilde{Z}_i^{\circ}$.
	For each $j=1,\ldots,n$, we obtain a morphism $\wedge^j f\colon \wedge^j V_n\to \wedge^j(V_n(1))=(\wedge^j V_n)(j)$.
	This defines a morphism $\wedge^j f\colon \Ocal_{\tilde{Z}_i^{\circ}}\to \Hom_{\FF,G_K}(\wedge^j V_n,(\wedge^j V_n)(j))$ on $\tilde{Z}_i^{\circ}$.
	By taking the dual, we obtain a morphism
	\[(\wedge^j f)^{\vee}\colon \Ext^2_{\FF,G_K}((\wedge^j V_n)(j-1),\wedge^j V_n)\to \Ocal_{\tilde{Z}_i^{\circ}}.\]
	Let $\tilde{Z}_{i,j}^{\circ}\subset \tilde{Z}_i^{\circ}$ be the reduced Zariski closed substack defined by the ideal sheaf $\Im((\wedge^j f)^{\vee})$.
	\begin{description}
		\item[Claim] $\tilde{Z}_{i}^{\circ}=\tilde{Z}_{i,n}^{\circ}\supset \tilde{Z}_{i,n-1}^{\circ}\supset\cdots \supset \tilde{Z}_{i,1}^{\circ}=\emptyset$.
	\end{description}
	By construction, $f^{\vee}$ is a surjection. Therefore, we have $\tilde{Z}_{i,1}^{\circ}=\emptyset$.
	Let $x\colon \Spa F\to \tilde{Z}_{i}^{\circ}$ be a classical point.
	Then $x\in \tilde{Z}_{i,j}^{\circ}$ if and only if 
	\[(\wedge^j f_x)^{\vee}\colon \Ext^2_{\FF,G_K}((\wedge^j V_{n,x})(j-1),\wedge^j V_{n,x})\to F\]
	is zero.
	In other words, $x\in \tilde{Z}_{i,j}^{\circ}$ if and only if $\wedge^j f_x\colon \wedge^j V_{n,x}\to \wedge^j(V_{n,x}(1))$ is zero.
	Note that $\wedge^n f_x$ is always zero by \cite[Proposition 6.2.8 (1)]{KPX14} as $\wedge^n V_{n,x}$ is of rank $1$.
	Therefore, the claim holds for classical points, which implies the claim itself by Lemma \ref{lem:Zariski closed inclusion classical points}.

	Set $\tilde{Z}_{i,j}^{\circ\circ}=\tilde{Z}_{i,j+1}^{\circ}\setminus \tilde{Z}_{i,j}^{\circ}$ for $j=1,\ldots,n-1$.
	Then $\bigsqcup_{j=1}^{n-1} \tilde{Z}_{i,j}^{\circ\circ}$ is a locally Zariski closed partition of $\tilde{Z}_i^{\circ}$.
	Therefore, it suffices to show that $\dim \tilde{Z}_{i,j}^{\circ\circ}\leq dn^2-2+i$.
	By construction, for every classical point $x\colon \Spa F\to \tilde{Z}_{i,j}^{\circ\circ}$, $\wedge^{j+1} f_x\colon \wedge^{j+1}V_{n,x}\to \wedge^{j+1}(V_{n,x}(1))$ is zero, but $\wedge^{j} f_x\colon \wedge^{j}V_{n,x}\to \wedge^{j}(V_{n,x}(1))$ is non-zero.
	Therefore, the codimension of $\Ker(f_x\colon V_{n,x}\to V_{n,x}(1))$ is $j$, and the codimension of $\Ker(\wedge^{j} f_x\colon \wedge^{j}V_{n,x}\to \wedge^{j}(V_{n,x}(1)))$ is $1$.

	Since \[(\wedge^j f)^{\vee}\colon \Ext^2_{\FF,G_K}((\wedge^j V_n)(j-1),\wedge^j V_n)\to \Ocal_{\tilde{Z}_i^{\circ}}\]
	is surjective over $\tilde{Z}_{i,j}^{\circ\circ}$, $\Ker((\wedge^j f)^{\vee})\subset \Ext^2_{\FF,G_K}((\wedge^j V_n)(j-1),\wedge^j V_n)$ is a vector bundle on $\tilde{Z}_{i,j}^{\circ\circ}$.
	For $\kbf\in \Zbb_{\geq 0}^{\Sigma_K}$, we have a universal surjection
	\[\Ext^2_{\FF,G_K}((\wedge^j V_n)(j-1),\wedge^j V_n)\twoheadrightarrow \Ext^2_{\FF,G_K}(\tbf^{\kbf}(\wedge^j V_n)(j-1),\wedge^j V_n).\]
	Let $W_{i,j,\kbf}$ be the cokernel of the composition
	\[\Ker((\wedge^j f)^{\vee})\hookrightarrow \Ext^2_{\FF,G_K}((\wedge^j V_n)(j-1),\wedge^j V_n)\twoheadrightarrow \Ext^2_{\FF,G_K}(\tbf^{\kbf}(\wedge^j V_n)(j-1),\wedge^j V_n).\]
	Let $\tilde{Z}_{i,j,\kbf}^{\circ\circ}\subset \tilde{Z}_{i,j}^{\circ\circ}$ be the reduced Zariski closed substack defined by the $0$-th Fitting ideal $\Fit_0(W_{i,j,\kbf})$, and set $Y_{i,j,\kbf}=\tilde{Z}_{i,j,\kbf}^{\circ\circ} \setminus \bigcup_{\kbf'>\kbf}\tilde{Z}_{i,j,\kbf'}^{\circ\circ}$.
	By Lemma \ref{lem:k large cohomology vanishing} and Example \ref{ex:locally Zariski closed partition rigid stack}, $\{Y_{i,j,\kbf}\}_{\kbf\in \Zbb_{\geq 0}^{\Sigma_K}}$ is a locally Zariski closed covering of $\tilde{Z}_{i,j}^{\circ\circ}$.
	Therefore, by Lemma \ref{lem:locally Zariski closed partition dimension}, it suffices to show that $\dim Y_{i,j,\kbf} \leq dn^2-2+i$.
	Summarizing the above constructions, for every classical point $x\colon \Spa F\to Y_{i,j,\kbf}$, the following properties hold:
	\begin{enumerate}
		\item The kernel $\Ker(f_x\colon V_{n,x}\to V_{n,x}(1))$ has codimension $j$;
		\item The image of the morphism $\wedge^{j} f_x\colon \wedge^{j}V_{n,x}\to \wedge^{j}(V_{n,x}(1))$ is a vector bundle of rank $1$;
		\item\label{largest k} The morphism $\wedge^{j} f_x\colon \wedge^{j}V_{n,x}\to \wedge^{j}(V_{n,x}(1))$ factors through $\tbf^{\kbf}\wedge^{j}(V_{n,x}(1))$, but does not factor through $\tbf^{\kbf'}\wedge^{j}(V_{n,x}(1))$ for any $\kbf'>\kbf$.
	\end{enumerate}
	In fact, Property \eqref{largest k} can be checked as follows: For a classical point $x\colon \Spa F\to \tilde{Z}_{i,j,\kbf}^{\circ\circ}$, by definition, $x$ is a classical point of $\tilde{Z}_{i,j,\kbf}^{\circ\circ}$ if and only if the composition
	\[\Ker((\wedge^j f_x)^{\vee})\hookrightarrow \Ext^2_{\FF,G_K}((\wedge^j V_{n,x})(j-1),\wedge^j V_{n,x})\twoheadrightarrow \Ext^2_{\FF,G_K}(\tbf^{\kbf}(\wedge^j V_{n,x})(j-1),\wedge^j V_{n,x})\]
	is not surjective.
	By taking the duals and applying Tate local duality, this is equivalent to the assertion
	\[\Hom_{\FF,G_K}(\wedge^j V_{n,x}, \tbf^{\kbf}(\wedge^j V_{n,x})(j))\to \Hom_{\FF,G_K}(\wedge^j V_{n,x}, (\wedge^j V_{n,x})(j))/\Im(\wedge^j f_x)\]
	is not injective.
	Therefore, $x$ is a classical point of $\tilde{Z}_{i,j,\kbf}^{\circ\circ}$ if and only if $\wedge^{j} f_x\colon \wedge^{j}V_{n,x}\to \wedge^{j}(V_{n,x}(1))$ factors through $\tbf^{\kbf}\wedge^{j}(V_{n,x}(1))$.

	Let $\Spa A\to Y_{i,j,\kbf}$ be a smooth morphism of pure dimension $m$.
	It suffices to show that $\dim A\leq dn^2-2+i+m$.
	Note that $A$ is reduced as $Y_{i,j,\kbf}$ is reduced by definition.
	Let $x\in (\Spa A)^{\cl}$ be a classical point.
	By applying \cite[Corollary 6.2.9]{KPX14} to the image of
	\[\wedge^j f_x:\wedge^jV_{n,x}\to \tbf^{\kbf}\wedge^j(V_{n,x}(1))\]
	inside its saturation, it follows from Property \eqref{largest k} that the image is saturated.
	Moreover, the cokernel $\Coker(\wedge^{j} f_x\colon \wedge^{j}V_{n,x}\to \tbf^{\kbf}\wedge^{j}(V_{n,x}(1)))$ is a vector bundle of rank $\binom{n}{j}-1$.
	Therefore, by Lemma \ref{lem:quotient vector bundle}, 
	\[\Coker(\wedge^{j} f_A\colon \wedge^{j}V_{n,A}\to \tbf^{\kbf}\wedge^{j}(V_{n,A}(1)))\]
	is a vector bundle on $X_{\Cbb_p}\times \Spa A$ of rank $\binom{n}{j}-1$.
	Hence, $\Ker (\wedge^{j} f_A)\subset \wedge^{j}V_{n,A}$ is a subbundle of codimension $1$.
	Set $\tilde{Y}_{i,j,\kbf}=\LS_{P_{n-j,j}(K),\FF}\times_{\LS_{\GL_n(K),\FF}} \Spa A$, which is a Zariski closed subspace of the rigid space $\Pbb(\wedge^{j}V_{n,A})$ as in Lemma \ref{lem:Plucker relation}. 
	The subbundle $\Ker (\wedge^{j} f_A)\subset \wedge^{j}V_{n,A}$ of codimension $1$ defines a section $\Spa A\to \Pbb(\wedge^{j}V_{n,A})$.
	By construction, for any classical point $x\in \Spa A$, the subbundle $\Ker (\wedge^{j} f_x)\subset \wedge^{j}V_{n,x}$ satisfies the Pl\"{u}cker relation.
	In other words, for any classical point $x\in \Spa A$, $\Spa k(x)\to \Spa A\to \Pbb(\wedge^{j}V_{n,A})$ factors through $\tilde{Y}_{i,j,\kbf}\subset \Pbb(\wedge^{j}V_{n,A})$.
	Since $A$ is reduced, we find that $\Spa A\to \Pbb(\wedge^{j}V_{n,A})$ also factors through $\tilde{Y}_{i,j,\kbf}\subset \Pbb(\wedge^{j}V_{n,A})$.
	Consequently, we obtain the section $\Spa A\to \tilde{Y}_{i,j,\kbf}$.
	By construction, we have 
	\[\dim\tilde{Y}_{i,j,\kbf}\leq i+m+\dim q^{-1}(Z) \leq dn^2-2+i+m\]
	by Theorem \ref{thm:parabolic dimension upper bound}, where $q\colon \LS_{P_{n-j,j}(K),\FF}\to \LS_{\GL_n(K),\FF}$.
	Therefore, by Lemma \ref{lem:section dimension}, we obtain $\dim \Spa A\leq \dim\tilde{Y}_{i,j,\kbf}\leq dn^2-2+i+m$, which completes the proof.
\end{proof}

%%%%%%%%%%%%%%%%%%%%%%%%%%%%%%%%%%%%%%%%%%%%%%%%%%%%%%%%%%%%%%%%%%%%%%%%%%%%%%%%%%
\subsection{Global sections on $\LS_{\GL_n(K),\FF}$}
In this subsection, we regard $\LS_{\GL_n(K),\FF}$ as a rigid analytic Artin stack over $E$ via base change.
Let $\mathfrak{gl}_n/\!/\GL_n$ be the GIT quotient of the Lie algebra $\mathfrak{gl}_n$ (over $K$) by the adjoint action of $\GL_n$.
We regard $\mathfrak{gl}_n/\!/\GL_n$ as a rigid space over $K$ by analytification.
For $M\in\mathfrak{gl}_n$, we write
\[\det(T-M)=T^n+a_1(M)T^{n-1}+\cdots+a_n(M).\]
The invariant functions $a_1,\ldots,a_n$ induce an isomorphism
\[
\mathfrak{gl}_n/\!/\GL_n\xrightarrow{\sim}\Abb_K^{n,\an}.
\]
For each $\sigma \in \Sigma_K$, set $\WT_{n,E,\sigma}=(\mathfrak{gl}_n/\!/\GL_n)\times_{\Spa K,\sigma}\Spa E$.
We define a rigid space $\WT_{n,E}$ over $E$ by 
\[\WT_{n,E}=\prod_{\sigma\in \Sigma_K}\WT_{n,E,\sigma}.\]
For an affinoid $E$-algebra $A$, a $G_K$-equivariant vector bundle $V_A$ on $X_{\Cbb_p,A}$, and $\sigma\in \Sigma_K$, we can define the Sen polynomial $P_{V_A,\sigma}(T)\in A[T]$ (\cite[Subsection 5.1.33]{EGH23} or \cite[Lemma and Definition 1.33]{Mikami25m}).
Since $P_{V_A,\sigma}(T)$ is a monic polynomial of degree $n$, its coefficients naturally determine an $A$-valued point of $\WT_{n,E,\sigma}$.
Consequently, we obtain a morphism 
\[\omega_n \colon \LS_{\GL_n(K),\FF}\to \WT_{n,E};\; V_A\mapsto (P_{V_A,\sigma}(T))_{\sigma\in \Sigma_K}.\]
On the other hand, by taking determinants, we obtain a morphism
\[\det \colon \LS_{\GL_n(K),\FF}\to \LS_{\GL_1(K),\FF}.\]
There is a natural commutative diagram
\[
\xymatrix{
	\LS_{\GL_n(K),\FF}\ar[r]^-{\omega_n}\ar[d]_-{\det} & \WT_{n,E}\ar[d]^-{\Tr}\\
	\LS_{\GL_1(K),\FF}\ar[r]^-{\omega_1} & \WT_{1,E},
}
\]
where the right vertical morphism $\Tr\colon \WT_{n,E}\to \WT_{1,E}$ is induced from the trace morphism $\Tr\colon \mathfrak{gl}_n\to \mathfrak{gl}_1$.

For a rigid analytic Artin stack $\Xcal$, we write $\Gamma(\Xcal)=\Gamma(\Xcal,\Ocal_{\Xcal})$ for the ring of global sections on $\Xcal$.
\begin{conjecture}[{\cite[Conjecture 5.1.36]{EGH23}}]\label{conj:EGH global sections}
	The morphism 
	\[\LS_{\GL_n(K),\FF}\xrightarrow{(\omega_n,\det)} \WT_{n,E}\times_{\WT_{1,E}}\LS_{\GL_1(K),\FF}\]
	induces an isomorphism on the global sections.
\end{conjecture}

\begin{remark}
	Under the identifications $\WT_{n,E}\cong \prod_{\sigma\in\Sigma_K}\Abb_E^{n,\an}$ and $\WT_{1,E}\cong \prod_{\sigma\in\Sigma_K}\Abb_E^{1,\an}$,
	the morphism $\Tr\colon\WT_{n,E}\to\WT_{1,E}$
	identifies with the product of the projections onto the first coordinates.
	On the other hand, since $\LS_{\GL_1(K),\FF}\cong X^{\an}(K^{\times})/\Gbb_m^{\an}$, where the action of $\Gbb_m^{\an}$ is trivial, we have
	\[\Gamma(\LS_{\GL_1(K),\FF})\cong \Gamma(X^{\an}(K^{\times})).\]
	Therefore, the above conjecture implies that
	\[\Gamma(\LS_{\GL_n(K),\FF})\cong \Gamma(\WT_{n,E})\hotimes_{\Gamma(\WT_{1,E})}\Gamma(X^{\an}(K^{\times})),\]
	which is non-canonically isomorphic to
	\[\Gamma(\prod_{\sigma\in \Sigma_K}\Abb^{n-1,\an}_E)\hotimes_{E}\Gamma(X^{\an}(K^{\times})).\]
\end{remark}

Le, Strauch, and Wang proved this conjecture under the assumption that $\LS_{\GL_n(K),\FF}$ is reduced (\cite[Proposition 5.1.2, Theorem 5.1.3]{LSW26}).
By Theorem \ref{thm:EGH conjecture}, $\LS_{\GL_n(K),\FF}$ is normal, and thus reduced.
Therefore, we obtain the following.
\begin{theorem}
	Conjecture \ref{conj:EGH global sections} holds.
\end{theorem}

%%%%%%%%%%%%%%%%%%%%%%%%%%%%%%%%%%%%%%%%%%%%%%%%%%%%%%%%%%%%%%%%%%%%%%%%%%%%%%%%%
\appendix

\section{Smoothness implies formal smoothness}

Let $K$ be a complete non-archimedean field.
The goal of this appendix is to prove the following:
\begin{theorem}\label{thm:smooth implies formally smooth}
	Let $f\colon X\to Y$ be a smooth morphism of quasi-analytic spaces over $K$.
	Then $f$ is formally smooth.
	In other words, for any square-zero closed immersion $i\colon \Spa A\to \Spa B$ of affinoid rigid spaces, and any commutative diagram
	\[
	\xymatrix{
		\Spa A \ar[r]^i\ar[d] & \Spa B\ar[d]\ar@{-->}[ld]\\
		X\ar[r]_f & Y,
	}
	\]
	there exists a dotted arrow making the diagram commute.
\end{theorem}

For a quasi-analytic space $X$ over $K$, let $X_{\et}$ denote the small \'{e}tale site of rigid spaces over $X$.	
We define the sheaf of rings $\Ocal_{X_{et}}$ on $X_{\et}$ by $\Gamma(U,\Ocal_{X_{et}})=\Ocal_U(U)$ for $U\in X_{\et}$.
For $F\in \Coh(X)$, we define the associated $\Ocal_{X_{et}}$-module, also denoted by $F$, by $\Gamma(U,F)=\Gamma(U,h^*F)$ for $(h\colon U\to X)\in X_{\et}$.
Then the functor
\[\Coh(X)\to \Mod_{\Ocal_{X_{\et}}};\; F\mapsto F\] 
is fully faithful. 

\begin{definition}
	Let $f\colon X\to Y$ be a morphism of quasi-analytic spaces over $K$.
	We define the coherent sheaf $\Omega_{X/Y}$ of continuous differentials as follows:
	\begin{itemize}
		\item If $X$ and $Y$ are rigid spaces, $\Omega_{X/Y}$ is already defined.
		\item When $X$ is general and $Y$ is a rigid space, we choose an \'{e}tale surjection $U\twoheadrightarrow X$.
		Then we have a coherent sheaf $\Omega_{U/Y}$ and natural isomorphisms
		\[p_1^*\Omega_{U/Y}\cong \Omega_{(U\times_X U)/Y} \cong p_2^*\Omega_{U/Y}\]
		satisfying the cocycle condition.
		We define $\Omega_{X/Y}$ as the coherent sheaf corresponding to the above descent datum.
		It does not depend on the choice of $U\twoheadrightarrow X$.
		\item When $X$ and $Y$ are general, we choose an \'{e}tale surjection $V\twoheadrightarrow Y$.
		Then we have a coherent sheaf $\Omega_{(X\times_Y V)/V}$ and natural isomorphisms
		\[p_1^*\Omega_{(X\times_Y V)/V}\cong \Omega_{((X\times_Y V)\times_X(X\times_Y V))/(V\times_Y V)} \cong p_2^*\Omega_{(X\times_Y V)/V}\]
		satisfying the cocycle condition.
		We define $\Omega_{X/Y}$ as the coherent sheaf corresponding to the above descent datum.
		It does not depend on the choice of $V\twoheadrightarrow Y$.
	\end{itemize}
\end{definition}

\begin{lemma}\label{lem:differential module sheaf theoretic}
	Let $f\colon X\to Y$ be a morphism of quasi-analytic spaces over $K$, and let $\Delta\colon X\to X\times_Y X$ be the diagonal morphism of $f$.
	Let $J$ be the $\Ocal_{(X\times_Y X)_{\et}}$-module defined by $J=\Ker(\Ocal_{(X\times_Y X)_{\et}}\to \Delta_*\Ocal_{X_{\et}})$.
	Then the $\Ocal_{X_{\et}}$-module $\Delta^*J$ is naturally isomorphic to $\Omega_{X/Y}$.
\end{lemma}
\begin{proof}
	We may assume that $Y$ is a rigid space.
	Let us take an \'{e}tale morphism $h\colon U\to X$ from an affinoid rigid space $U$.
	Let $\Delta_{h}\colon U\to U\times_X U$ be the diagonal morphism of $h$.
	Since $h$ is \'{e}tale, $\Delta_{h}$ is an open immersion.
	Moreover, as $U$ is an affinoid rigid space, the diagonal morphism $U\to U\times U$ is a closed immersion.
	Since this closed immersion factors through $U\times_X U\subset U\times U$, $\Delta_{h}\colon U\to U\times_X U$ is also a closed immersion.
	Therefore, $U$ is an open and closed subspace of $U\times_X U$.
	We need to show that $h^*\Delta^*J$ is naturally isomorphic to $\Omega_{U/Y}$.
	Let us consider the diagram
	\[
	\xymatrix{
	 U\ar[rr]^-{\Delta_h}\ar[rrd]_-{h}\ar@/^25pt/[rrrr]^-{\Delta_{fh}} & &U\times_X U\ar[rr]^-{i}\ar[d]& & U\times_Y U\ar[d]^-{h\times h}\\
	& &X \ar[rr]_-{\Delta}& & X\times_Y X.
	}
	\]
	We write $\Delta_{fh}\colon U\to U\times_Y U$ and $i\colon U\times_XU \to U\times_Y U$ as above.
	Then we have $h^*\Delta^*J \cong \Delta_{fh}^*(h\times h)^*J$.
	We set $J_U=(h\times h)^*J$.
	Since $h\times h\colon U\times_Y U\to X\times_Y X$ is \'{e}tale, we have
	\[J_U\cong \Ker(\Ocal_{(U\times_Y U)_{\et}}\to i_*\Ocal_{(U\times_X U)_{\et}}).\]
	We set $I_U=\Ker(\Ocal_{(U\times_Y U)_{\et}}\to (\Delta_{fh})_*\Ocal_{U_{\et}})$ and $C=\Ker(i_*\Ocal_{(U\times_X U)_{\et}}\to  (\Delta_{fh})_*\Ocal_{U_{\et}})$.
	Then, there exists a natural left exact sequence
	\[0\to J_U \to I_U \to C.\]
	Since $U$ is an open and closed subspace of $U\times_X U$, the morphism 
	\[i_*\Ocal_{(U\times_X U)_{\et}}\to  (\Delta_{fh})_*\Ocal_{U_{\et}}\] 
	becomes an isomorphism after applying the exact functor $\Delta_{fh}^{-1}$.
	Hence $\Delta_{fh}^{-1}C=0$, and consequently $\Delta_{fh}^{-1}J_U\cong \Delta_{fh}^{-1}I_U$.
	It follows that $\Delta_{fh}^*J_U\cong \Delta_{fh}^*I_U$.
	It remains to show that $\Delta_{fh}^*I_U$ is naturally isomorphic to $\Omega_{U/Y}$.
	This immediately follows from the definition before Lemma 1.6.4 in \cite{Hub96}.
\end{proof}

\begin{proof}[Proof of Theorem \ref{thm:smooth implies formally smooth}]
Consider a commutative diagram 
\[
\xymatrix{
        \Spa A \ar[r]^-i\ar[d]_-{g} & \Spa B\ar[d]\ar@{-->}[ld]\\
        X\ar[r]_f & Y
}
\]
as in Theorem \ref{thm:smooth implies formally smooth}, and let $g\colon \Spa A\to X$ be the left vertical morphism.
We need to show that there exists a morphism $\Spa B\to X$ (represented by the dotted arrow) making the above diagram commute.
Let $I=\Ker(B\to A)$, viewed as an $A$-module.
We define a sheaf $F$ of sets on $(\Spa A)_{\et}$ as follows:
Let $h\colon U\to \Spa A$ be an \'{e}tale morphism of rigid spaces.
Since $\Spa A\to \Spa B$ is a square-zero closed immersion, $h$ uniquely lifts to an \'{e}tale morphism $V\to \Spa B$ by \cite[Proposition 2.3.7]{Hub96}.
Let $F(U)$ denote the set of morphisms $V\to X$ (represented by the dotted arrow) making the following diagram commute:
\begin{align}\label{eq:infinitesimal lifting}
	\xymatrix{
	    U\ar[r]\ar[d]_-{h} & V\ar[d]\ar@{-->}[ldd]\\
		\Spa A \ar[d]_-{g} & \Spa B\ar[d]\\
		X\ar[r]_-f & Y.
}
\end{align}
We need to show that $F(\Spa A)$ is non-empty.
It suffices to show that $F$ is a $\Hom(g^*\Omega_{X/Y},I)$-torsor.
In fact, since we have $H^1_{\et}(\Spa A, \Hom(g^*\Omega_{X/Y},I))=0$, every $\Hom(g^*\Omega_{X/Y},I)$-torsor is trivial and thus admits a global section.
First, we define a morphism 
\[F\times F\to F\times \Hom(g^*\Omega_{X/Y},I).\]
Let $h\colon U\to \Spa A$ be an \'{e}tale morphism, and let $a,b\colon V\to X$ be morphisms making the diagram \eqref{eq:infinitesimal lifting} commute.
Since $U\to V$ is a square-zero closed immersion, the morphism $U_{\et}\to V_{\et}$ of sites is an equivalence.
Using this equivalence we regard $\Ocal_{V_{\et}}$ as a sheaf on $U_{\et}$.
The morphisms $a$ and $b$ define the morphism $(a,b)\colon V\to X\times_Y X$.
From this, we obtain a morphism $\Ocal_{(X\times_Y X)_{\et}}\to (\Delta gh)_*\Ocal_{V_{\et}}$ of sheaves of rings on $(X\times_Y X)_{\et}$, where $\Delta\colon X\to X\times_Y X$ is the diagonal morphism.
We set $J=\Ker(\Ocal_{(X\times_Y X)_{\et}}\to \Delta_*\Ocal_{X_{\et}})$.
Since the restrictions of $a$ and $b$ to $U$ coincide, the morphism 
\[J\hookrightarrow \Ocal_{(X\times_Y X)_{\et}}\to (\Delta gh)_*\Ocal_{V_{\et}}\]
factors through $(\Delta gh)_*h^*I\subset (\Delta gh)_*\Ocal_{V_{\et}}$.
Therefore, by Lemma \ref{lem:differential module sheaf theoretic}, we get a morphism
\[\Omega_{X/Y}=\Delta^*J\to \Delta^*(\Delta gh)_*h^*I\to (gh)_*h^*I.\]
By adjunction, we obtain a morphism $h^*g^*\Omega_{X/Y}\to h^*I$, which defines an element $b-a \in \Gamma(U,\Hom(g^*\Omega_{X/Y},I))$.
Thus, we obtain a morphism \[F\times F\to F\times \Hom(g^*\Omega_{X/Y},I);\; (a,b)\mapsto (a,b-a).\]

Finally, we prove that the above morphism is an isomorphism locally on the site $(\Spa A)_{\et}$.
Let us take an \'{e}tale surjection $Z\twoheadrightarrow X$ from a rigid space $Z$.
We may assume that $gh\colon U\to X$ lifts to a morphism $j\colon U\to Z$.
Then the set $F(U)$ is naturally in bijection with the set of morphisms $V\to Z$ (represented by the dotted arrow) making the following diagram commute:
\[
\xymatrix{
U \ar[rr]\ar[d]_-{j} & & V\ar[d]\ar@{-->}[lld]\\
Z\ar[r] & X\ar[r]_-{f} &Y.
}
\]
Since both $Z$ and $Y$ are rigid spaces, the claim follows from standard arguments.
\end{proof}

%%%%%%%%%%%%%%%%%%%%%%%%%%%%%%%%%%%%%%%%%%%%%%%%%%%%%%%%%%%%%%%%%%%%%%%%%%%%%%%%%%%%%
\section{A smooth surjection is a surjection with respect to the \'{e}tale topology}

In this appendix, we prove the following theorem.

\begin{theorem}\label{thm:smooth surjection etale topology}
	Let $f\colon X\to Y$ be a smooth surjection of locally noetherian analytic adic spaces.
	Then there exist an \'{e}tale surjection $g\colon Z\twoheadrightarrow Y$ and a morphism $h\colon Z\to X$ such that $f\circ h=g$.
\end{theorem}

This theorem is stated in \cite[Proposition 3.1.11]{War17}.
However, since the author is unable to follow a lifting argument therein (see the lemma below), we present a modified proof here.

\begin{lemma}\label{lem:lifting of points}
	Let $K$ be an algebraically closed complete non-archimedean field, and let $K^+\subset K$ be an open bounded valuation ring.
	Let $n$ be a positive integer, and let $U\subset \Bbb^n_{\Spa(K,K^+)}$ be an open subspace.
	If the projection $U\to \Spa(K,K^+)$ is surjective, then it admits a section.
\end{lemma}
\begin{proof}
	We may assume that $U$ is a rational open subspace of the form
    \[U=\Bbb^n_{\Spa(K,K^+)}\left(\frac{f_1,\ldots,f_m}{g}\right),\] 
    where $f_1,\ldots,f_m,g\in K\langle T_1,\ldots,T_n\rangle$ generate the unit ideal.
	By \cite[Lemma 3.10]{Huber93}, we may assume that $f_1,\ldots,f_m,g$ are polynomials in $T_1,\ldots,T_n$.
	Since the projection $U\to \Spa(K,K^+)$ is surjective, there exist an extension of algebraically closed valued fields $(L,L^+)/(K,K^+)$ and elements $a_1,\ldots, a_n\in L^+$ such that 
	\[|f_i(a_1,\ldots,a_n)|_{L^+}\leq |g(a_1,\ldots,a_n)|_{L^+}\neq 0\] 
	for each $i=1,\ldots,m$, where $|\cdot|_{L^+}$ is the (multiplicative) valuation corresponding to $(L,L^+)$.
	Since the theory of algebraically closed valued fields has elimination of quantifiers, there exist elements $b_1,\ldots, b_n\in K^+$ such that $|f_i(b_1,\ldots,b_n)|_{K^+}\leq |g(b_1,\ldots,b_n)|_{K^+}\neq 0$ for each $i=1,\ldots,m$, where $|\cdot|_{K^+}$ is the valuation corresponding to $(K,K^+)$.
	These elements $b_1,\ldots, b_n\in K^+$ define the desired section.
\end{proof}

\begin{corollary}\label{cor:etale local section point}
	Let $K$ be a complete non-archimedean field, and let $K^+\subset K$ be an open bounded valuation ring.
	Let $n$ be a positive integer, and let $U\subset \Bbb^n_{\Spa(K,K^+)}$ be an open subspace.
	Assume that the projection $U\to \Spa(K,K^+)$ is surjective. 
	Then there exist a finite separable extension $L/K$ and a non-empty open subspace $\Sigma\subset (L^+)^n$, where $L^+$ is the integral closure of $K^+$ in $L$, such that for any $(a_1,\ldots,a_n)\in \Sigma$, the morphism $\Spa(L,L^+)\to \Bbb^n_{\Spa(K,K^+)}$ defined by $(a_1,\ldots,a_n)$ factors through $U\subset \Bbb^n_{\Spa(K,K^+)}$.
\end{corollary}
\begin{proof}
	Let $(C,C^+)$ be the completion of the algebraic closure of $(K,K^+)$, and let $U_C$ be the preimage of $U$ under the morphism $\Bbb^n_{\Spa(C,C^+)}\to\Bbb^n_{\Spa(K,K^+)}$.
	By Lemma \ref{lem:lifting of points}, there exist elements $a_1,\ldots,a_n\in C^+$ defining a section $i\colon \Spa(C,C^+)\to U_C\subset \Bbb^n_{\Spa(C,C^+)}$.
	Let $\varpi$ be a pseudo-uniformizer of $K$, and let 
	\[V_m=\Bbb^n_{\Spa(C,C^+)}\left(\frac{T_1-a_1,\ldots,T_n-a_n}{\varpi^m}\right)\] 
	for each positive integer $m$.
	Since $i(\Spa(C,C^+))=\bigcap_m V_m$ is contained in $U_C$, there exists a positive integer $m$ such that $V_m\subset U_C$.
	Since the separable closure of $K$ is dense in $C$ by \cite[3.4.1 Proposition 6]{BGR}, there exists a finite separable extension $L/K$ in $C$ such that $\Sigma=L^n\cap \{(b_1,\ldots,b_n)\in (C^+)^n\mid b_i-a_i\in \varpi^m C^+\}$ is non-empty.
	This extension $L$ and the open subspace $\Sigma\subset L^n\cap(C^+)^n=(L^+)^n$ satisfy the desired property.
\end{proof}

For a locally noetherian analytic adic space $X$, let $X_{\fet}$ denote the category of finite \'{e}tale adic spaces over $X$.
For a ring $A$, let $A_{\fet}$ denote the category of finite \'{e}tale $A$-algebras.
Note that for a strongly noetherian Tate affinoid pair $(A,A^+)$, we have $\Spa(A,A^+)_{\fet}\simeq (A_{\fet})^{\op}$.

\begin{lemma}\label{lem:fet limit}
	Let $X$ be a locally noetherian analytic adic space, and let $x\in X$ be a point.
	Then the natural functor 
	\[\varinjlim_{x\in U}U_{\fet}\to \Spa(k(x),k(x)^+)_{\fet}\]
	is an equivalence, where the colimit is a $2$-colimit taken over all open neighborhoods of $x$.
\end{lemma}
\begin{proof}
	After replacing $X$ with an affinoid open neighborhood of $x$, we can choose a pseudo-uniformizer $\varpi\in \Ocal(X)$.
	For an affinoid open neighborhood $\Spa(A,A^+)\subset X$ of $x$, the pair $(A^+,\varpi A^+)$ is henselian by \cite[Lemma 3.1.1]{Hub96}.
	Let $\tilde{k(x)}^+=\varinjlim_{x\in \Spa (A,A^+)}A^+$ (resp.\ $\tilde{k(x)}=\varinjlim_{x\in \Spa (A,A^+)}A$), where the colimit is taken over all affinoid open neighborhoods of $x$.
	Then the $\varpi$-adic completion of $\tilde{k(x)}^+$ is isomorphic to $k(x)^+$.
	By \cite[0FWT]{stacks-project}, the pair $(\tilde{k(x)}^+,\varpi \tilde{k(x)}^+)$ is also henselian.
	By \cite[Proposition 5.4.53]{GR03}, we have $k(x)_{\fet}\simeq \tilde{k(x)}_{\fet}$.
	Therefore, we obtain 
	\[k(x)_{\fet}\simeq \tilde{k(x)}_{\fet}\simeq \varinjlim_{x\in \Spa (A,A^+)}A_{\fet},\]
	which proves the claim.
\end{proof}

\begin{proof}[Proof of Theorem \ref{thm:smooth surjection etale topology}]
	Let $x\in X$ be a point, and let $y=f(x)$ be the image.
	It suffices to show	that there exist an \'{e}tale morphism $g\colon Z\to Y$ whose image contains $y$ and a morphism $h\colon Z\to X$ such that $f\circ h=g$.
	By \cite[Corollary 1.6.10]{Hub96}, after replacing $X$ and $Y$ with open neighborhoods of $x$ and $y$, respectively, we may assume that there is a factorization
	\[
	\xymatrix{
		X \ar[r]\ar[rd]_-{f} & \Bbb^n_Y\ar[d]\\
		 & Y,
	}
	\]
	where the morphism $X\to \Bbb^n_Y$ is \'{e}tale.
	By \cite[Proposition 1.7.8]{Hub96}, the morphism $X\to \Bbb^n_Y$ is an open map. 
	Let $U\subset \Bbb^n_Y$ be the image of this morphism, and let $U_y =U\cap \Bbb^n_{\Spa(k(y),k(y)^+)}$.
	By applying Corollary \ref{cor:etale local section point} to $U_y\subset \Bbb^n_{\Spa (k(y),k(y)^+)}\to \Spa(k(y),k(y)^+)$, we get a finite separable extension $L/k(y)$ and an open subspace $\Sigma\subset (L^+)^n$ satisfying the condition in Corollary \ref{cor:etale local section point}.
	By Lemma \ref{lem:fet limit}, after replacing $Y$ with an open neighborhood of $y$, we can take a finite \'{e}tale morphism $Z\to Y$ which is an extension of $\Spa (L,L^+)\to \Spa(k(y),k(y)^+)$.
	Let $z\in \Spa(L,L^+)$ be the closed point, and we regard it as a point of $Z$.
	By replacing $Z$ with an open neighborhood of $z$, we may assume that $\Im((\Ocal(Z)^+)^n\to (L^+)^n)\cap \Sigma\neq\emptyset$. 
	Let us take $(a_1,\ldots,a_n)\in (\Ocal(Z)^+)^n$ whose image in $(L^+)^n$ lies in $\Sigma$.
	These elements define a morphism $Z\to \Bbb^n_Y$.
	By replacing $Z$ with a preimage of $U\subset \Bbb^n_Y$, which contains $z$ by construction, we may assume that the morphism $Z\to \Bbb^n_Y$ factors through $U\subset \Bbb^n_Y$.
	By replacing $Z$ with $X\times_{U}Z$ (which is \'{e}tale over $Y$ and whose image contains $y$), we obtain the desired lift $Z\to X$ of $Z\to Y$ along $f\colon X\to Y$.
\end{proof}

%%%%%%%%%%%%%%%%%%%%%%%%%%%%%%%%%%%%%%%%%%%%%%%%%%%%%%%%%%%%%%%%%%%%%%%%%%%%%%%%%%%%%

\section{$\Gamma$-equivariant $G$-torsors}\label{appendix:Gamma-equivariant torsor}

Let $\Gamma$ be a profinite group.
First, we recall the notion of continuous actions of $\Gamma$ on (analytic) adic spaces.

\begin{definition}
	Let $X$ be an analytic adic space equipped with a right action of $\Gamma$.
	The action of $\Gamma$ on $X$ is \textit{continuous} if the following conditions are satisfied:
	\begin{enumerate}
		\item The action of $\Gamma$ on the underlying topological space of $X$ is continuous, that is, $X\times \Gamma \to X$ is continuous.
		\item For any affinoid open subspace $U=\Spa A$ of $X$, let $\Gamma_U\subset \Gamma$ be the stabilizer of $U$.
		Then the induced action of $\Gamma_U$ on $A$ is continuous, that is, $\Gamma_U\times A\to A$ is continuous.
		(Note that $\Gamma_U$ is an open subgroup of $\Gamma$, as explained in the remark below).
	\end{enumerate}
\end{definition}

\begin{remark}
	Let us prove that $\Gamma_U$ is open in $\Gamma$.
	Since $X\times \Gamma \to X$ is continuous, for each $x\in U$, there exist an open subgroup $\Gamma_x\subset \Gamma$ and an open neighborhood $U_x$ of $x$ in $U$ such that $U_x \Gamma_x\subset U$.
	Since $U$ is quasi-compact, there exist finitely many points $x_1,\ldots,x_n\in U$ such that $\bigcup_{i=1}^n U_{x_i}=U$.
	Then the open subgroup $\bigcap_{i=1}^n \Gamma_{x_i}\subset \Gamma$ stabilizes $U$. 
\end{remark}

For a complete Tate affinoid pair $(A,A^+)$ and a profinite set $S$, let $C(S,A)$ (resp.\ $C(S,A^+)$) denote the ring of continuous functions on $S$ with values in $A$ (resp.\ $A^+$).
Endowed with the compact-open topology, $C(S,A)$ also becomes a complete Tate $f$-adic ring.
Moreover, $C(S,A^+)$ is a ring of integral elements in $C(S,A)$.

\begin{lemma}\label{lem:profinite set as an adic space}
	Let $(A,A^+)$ be a complete Tate affinoid pair.
	Then the topological space $\Spa(C(S,A),C(S,A^+))$ is homeomorphic to $S\times \Spa (A,A^+)$.
\end{lemma}
\begin{proof}
	Let $A_0$ be a ring of definition of $A$, and let $\varpi\in A_0$ be a pseudo-uniformizer.
	Let $C^{\lc}(S,A)$ (resp.\ $C^{\lc}(S,A^+)$, $C^{\lc}(S,A_0)$) denote the ring of locally constant functions on $S$ with values in $A$ (resp.\ $A^+$, $A_0$).
	We endow $C^{\lc}(S,A)$ with a Tate $f$-adic ring structure so that $C^{\lc}(S,A_0)$ is a ring of definition and $\varpi$ is a pseudo-uniformizer.
	Then $(C^{\lc}(S,A), C^{\lc}(S,A^+))$ is a (non-complete) Tate affinoid pair, and its completion is isomorphic to $(C(S,A),C(S,A^+))$.
	Therefore, by \cite[Proposition 3.9]{Huber93}, $\Spa(C(S,A),C(S,A^+))$ is homeomorphic to $\Spa (C^{\lc}(S,A), C^{\lc}(S,A^+))$.
	We write $S$ as a cofiltered limit $S=\varprojlim_i S_i$ of finite sets $S_i$.
	Then, by construction, we have a homeomorphism
	\[\Spa (C^{\lc}(S,A), C^{\lc}(S,A^+)) \cong \varprojlim_i \Spa (C(S_i,A),C(S_i,A^+)).\]
	Since $S_i$ is a finite set, $\Spa (C(S_i,A),C(S_i,A^+))$ is homeomorphic to $S_i\times \Spa (A,A^+)$.
	Thus, we obtain
	\begin{align*}
		\varprojlim_i \Spa (C(S_i,A),C(S_i,A^+)) \cong \varprojlim_i(S_i\times \Spa (A,A^+))\cong S\times \Spa (A,A^+),
	\end{align*}
	which proves the claim.
\end{proof}

\begin{lemma}\label{lem:action on affinoid pair and adic space}
	Let $(A,A^+)$ be a sheafy complete Tate affinoid pair equipped with a continuous left action of $\Gamma$, where ``continuous'' means that the map $\Gamma\times A\to A$ is continuous.
	Then the induced right action of $\Gamma$ on the adic space $\Spa (A,A^+)$ is also continuous.
\end{lemma}
\begin{proof}
	Since $\Gamma\times A\to A$ is continuous, the morphism $A\to C(\Gamma,A);\; a\mapsto (\gamma\mapsto \gamma a)$ is well-defined and continuous.
	Therefore, we obtain a morphism $(A,A^+)\to (C(\Gamma,A),C(\Gamma,A^+))$ of complete Tate affinoid pairs.
	Thus, by Lemma \ref{lem:profinite set as an adic space}, we get a continuous morphism
	\[\Gamma \times \Spa (A,A^+)\cong \Spa (C(\Gamma,A),C(\Gamma,A^+))\to \Spa (A,A^+),\]
	which implies that the action of $\Gamma$ on the underlying topological space of $\Spa (A,A^+)$ is continuous.

	Next, let $U=\Spa (B,B^+)$ be a rational open subspace of $\Spa (A,A^+)$, and let $\Gamma_U\subset \Gamma$ be the stabilizer of $U$.
	Then, the morphism
	\[\Gamma_U \times \Spa (B,B^+)\cong \Spa (C(\Gamma_U,B),C(\Gamma_U,B^+))\to \Spa (A,A^+)\]
	factors through $U=\Spa(B,B^+)\subset \Spa (A,A^+)$.
	Therefore, by \cite[Proposition 1.3]{Huber94}, we obtain a morphism
	\[(B,B^+)\to (C(\Gamma_U,B),C(\Gamma_U,B^+))\]
	of complete Tate affinoid pairs.
	In particular, we get a continuous morphism $B\to C(\Gamma_U,B)$, which implies that the action of $\Gamma_U$ on $B$ is continuous.

	Finally, let $V=\Spa (C,C^+)$ be an affinoid open subspace of $\Spa (A,A^+)$, and let $\Gamma_V\subset \Gamma$ be the stabilizer of $V$.
	We prove that the action of $\Gamma_V$ on $C$ is continuous.
	We write $V=\bigcup_{i=1}^n U_i$ as a union of rational open subspaces $U_i=\Spa(B_i,B_i^+)$ of $\Spa (A,A^+)$.
	Let $\Gamma'\subset \Gamma_V$ be the intersection of the stabilizers of $U_i$.
	We need to show that the action of $\Gamma'$ on $C$ is continuous.
	Since $(A,A^+)$ is sheafy, the sequence
	\[0\to C\to \prod_{i=1}^n B_i \to \prod_{i<j}B_i\hotimes_A B_j\]
	is strictly exact.
	In particular, $C\to \prod_{i=1}^n B_i$ is a closed embedding.
	Since the action of $\Gamma'$ on each $B_i$ is continuous, the action of $\Gamma'$ on $C$ is also continuous.
\end{proof}

\begin{lemma}\label{lem:continuity affinoid}
	Let $X$ be an analytic adic space equipped with a right action of $\Gamma$.
	The action of $\Gamma$ on $X$ is continuous if and only if there exists an affinoid open cover $\{U_i\}$ of $X$ such that the stabilizer $\Gamma_{U_i}\subset \Gamma$ of $U_i$ is an open subgroup and the induced action of $\Gamma_{U_i}$ on $\Ocal_X(U_i)$ is continuous.
\end{lemma}
\begin{proof}
	This immediately follows from Lemma \ref{lem:action on affinoid pair and adic space}.
\end{proof}

\begin{corollary}\label{cor:continuous fiber product}
	Let $S$ be a locally noetherian analytic adic space, let $X$ be a locally noetherian analytic adic space over $S$ endowed with a continuous right action of $\Gamma$ over $S$, and let $Y$ be an adic space locally of weakly finite type over $S$.
	Then the induced action of $\Gamma$ on $X\times_S Y$ (where $\Gamma$ acts trivially on $Y$) is continuous.
\end{corollary}
\begin{proof}
	By Lemma \ref{lem:continuity affinoid}, we may assume that $X,Y,S$ are affinoid adic spaces.
	Then the claim follows easily from Lemma \ref{lem:action on affinoid pair and adic space}.
\end{proof}

\begin{lemma}\label{lem:continuous smooth surjection}
	Let $X$ and $Y$ be locally noetherian analytic adic spaces with a right action of $\Gamma$.
	Let $f\colon X\to Y$ be a $\Gamma$-equivariant smooth surjective morphism.
	Assume that the action of $\Gamma$ on $X$ is continuous.
	Then the action of $\Gamma$ on $Y$ is also continuous.
\end{lemma}
\begin{proof}
	Since the morphism $f\colon X\to Y$ is an open map, the action of $\Gamma$ on the underlying topological space of $Y$ is continuous.
	Therefore, by Lemma \ref{lem:continuity affinoid}, after replacing $\Gamma$ with an open subgroup (the stabilizer), we may assume that $X=\Spa (A,A^+)$ and $Y=\Spa(B,B^+)$ are affinoid adic spaces.
	By \'{e}tale descent and Theorem \ref{thm:smooth surjection etale topology}, the sequence
	\[0\to B\to A \to A\hotimes_B A\]
	is strictly exact.
	In particular, the morphism $B\to A$ is a closed embedding.
	Since the action of $\Gamma$ on $A$ is continuous, the action of $\Gamma$ on $B$ is also continuous.
\end{proof}

\begin{lemma}\label{lem:continuous immersion}
	Let $X$ and $Y$ be locally noetherian analytic adic spaces with a right action of $\Gamma$.
	Let $i\colon X\to Y$ be a $\Gamma$-equivariant locally closed immersion.
	Assume that the action of $\Gamma$ on $Y$ is continuous.
	Then the action of $\Gamma$ on $X$ is also continuous.
\end{lemma}
\begin{proof}
	This immediately follows from Lemma \ref{lem:action on affinoid pair and adic space}.
\end{proof}

Let $K$ be a complete non-archimedean field, and $G$ be a smooth linear algebraic group over $K$.
Before introducing the notion of equivariant $G$-torsors, let us recall the non-equivariant version.
Let $X$ be a locally noetherian analytic adic space over $K$.
We endow the category of adic spaces locally of finite type over $X$ with the \'{e}tale topology as in \cite{Hub96}.
An \'{e}tale sheaf over $X$ means a sheaf on this site.

\begin{definition}\label{def:torsors}
Let $G^{\an}$ be the analytification of $G$, which is a rigid analytic group over $K$, and let $\Rep(G)$ denote the category of finite-dimensional algebraic representations of $G$.
\begin{enumerate}
	\item A \textit{geometric $G$-torsor} on $X$ is an adic space $P$ locally of finite type over $X$ with a right action of $G^{\an}$ over $X$ such that, \'{e}tale locally on $X$, there is a $G^{\an}$-equivariant isomorphism $P\cong X\times G^{\an}$.
	\item A \textit{cohomological $G$-torsor} on $X$ is an \'{e}tale sheaf $\Pcal$ over $X$ with a right action of $G^{\an}$ over $X$ such that, \'{e}tale locally on $X$, there is a $G^{\an}$-equivariant isomorphism $\Pcal \cong X\times G^{\an}$.
	\item A \textit{Tannakian $G$-torsor} on $X$ is an exact $K$-linear $\otimes$-functor $\Rep(G)\to \Vect(X)$, where $\Vect(X)$ is the category of vector bundles on $X$.
\end{enumerate}
The following categories are naturally equivalent:
\begin{enumerate}
	\item The category of geometric $G$-torsors on $X$.
	\item The category of cohomological $G$-torsors on $X$.
	\item The category of Tannakian $G$-torsors on $X$.
\end{enumerate}
These equivalences can be proved in the same way as in \cite[Theorem 19.5.2]{SW20}.
Using these equivalences, we identify these three notions and simply refer to them as \textit{$G$-torsors} on $X$.
\end{definition}

\begin{remark}
	When $G=\GL_n$, the notion of $G$-torsors is equivalent to that of vector bundles of rank $n$.
	In fact, for a vector bundle $E$ of rank $n$ on $X$, we get a (cohomological) $G$-torsor $\underline{\Isom}_X(\Ocal_X^n,E)$.
\end{remark}

\begin{definition}
	Let $\Gcal$ be an \'{e}tale sheaf of groups over $X$.
	Let $T_1\to X$ (resp.\ $T_2\to X$) be an \'{e}tale sheaf over $X$ with a right (resp.\ left) action of $\Gcal$.
	Then we define the \'{e}tale sheaf $T_1\times^{\Gcal} T_2$ over $X$ as the \'{e}tale sheafification of the presheaf given by
	\[(\Spa A\to X)\mapsto T_1(A)\times T_2(A)/\sim,\]
	where $\sim$ is defined by $(xg,y)\sim (x,gy)$ for $x\in T_1(A)$, $y\in T_2(A)$, and $g\in \Gcal(A).$
\end{definition}

\begin{example}\label{ex:torsor functorial}
	Let $G\to H$ be a morphism of linear algebraic groups over $K$.
	Let $\Pcal$ be a cohomological $G$-torsor on $X$.
	Then $\Pcal\times^{G^{\an}} H^{\an}=\Pcal\times^{G^{\an}\times X} (H^{\an}\times X)$ is a cohomological $H$-torsor.
	In particular, we obtain a functor from the category of $G$-torsors on $X$ to that of $H$-torsors on $X$.
\end{example}

Let us introduce the notion of $\Gamma$-equivariant $G$-torsors.
Let $X$ be a locally noetherian analytic adic space over $K$ equipped with a continuous right action of $\Gamma$.

\begin{definition}
	A \textit{geometric $\Gamma$-equivariant $G$-torsor} on $X$ is a geometric $G$-torsor $P$ over $X$ with a continuous right action of $\Gamma$ over $X$ such that the actions of $G^{\an}$ and $\Gamma$ commute.
\end{definition}

Let $G\to H$ be a morphism of smooth linear algebraic groups over $K$.
For a geometric $\Gamma$-equivariant $G$-torsor $P$ over $X$, we obtain a geometric $H$-torsor $P_H=P\times^{G^{\an}} H^{\an}$ as in Example \ref{ex:torsor functorial}.
By the functoriality of this construction, we obtain a natural action of $\Gamma$ on $P_H$ over $X$.

\begin{lemma}\label{lem:P_H continuous}
	The action of $\Gamma$ on $P_H$ is continuous.
\end{lemma}
\begin{proof}
	By construction, we have a $\Gamma$-equivariant morphism
	\[P\times_{\Spa K} H^{\an}\to P_H,\]
	where the action of $\Gamma$ on $H^{\an}$ is trivial.
	By Corollary \ref{cor:continuous fiber product}, the action of $\Gamma$ on $P\times_{\Spa K} H^{\an}$ is continuous.
	Since the morphism $P\times_{\Spa K} H^{\an}\to P_H$ is a geometric $G$-torsor (where the action of $G^{\an}$ is given by $(x,y)\cdot g=(xg,g^{-1}y)$), it is a smooth surjection.
	Therefore, by Lemma \ref{lem:continuous smooth surjection}, the action of $\Gamma$ on $P_H$ is also continuous.
\end{proof}

From the above lemma, we obtain a functor from the category of geometric $\Gamma$-equivariant $G$-torsors on $X$ to that of such $H$-torsors.

Next, we define the notion of $\Gamma$-equivariant $G$-torsors using the Tannakian formalism and compare it with the geometric one.

\begin{definition}
	Let $\Ecal$ be a $\Gamma$-equivariant vector bundle on $X$.
	We say that the action of $\Gamma$ on $\Ecal$ is continuous if for any affinoid open subspace $U=\Spa (A,A^+)$ of $X$, the action of the stabilizer $\Gamma_U$ of $U$ on the finite projective $A$-module $H^0(U,\Ecal)$ is continuous.
	Let $\Vect^{\Gamma}(X)$ denote the category of continuous $\Gamma$-equivariant vector bundles on $X$.
\end{definition}

\begin{remark}
    Let $\Ecal$ be a $\Gamma$-equivariant vector bundle on $X$.
    Similarly to Lemma \ref{lem:action on affinoid pair and adic space}, we can show that the action of $\Gamma$ on $\Ecal$ is continuous if and only if there exists an affinoid open covering $\{U_i\}$ of $X$ such that the action of $\Gamma_{U_i}$ on $H^0(U_i,\Ecal)$ is continuous.
\end{remark}

\begin{lemma}\label{lem:geometric vector bundle vs vector bundle}
	Let $n$ be a positive integer.
	Then the following categories are naturally equivalent:
	\begin{enumerate}
		\item The category of geometric vector bundles on $X$ of rank $n$ with a continuous $\Gamma$-action over $X$.
		\item The category of continuous $\Gamma$-equivariant vector bundles on $X$ of rank $n$.
	\end{enumerate}
\end{lemma}
\begin{proof}
	First, we have a natural categorical equivalence between
	\begin{enumerate}
		\item[(i)] the category of geometric vector bundles on $X$ of rank $n$ with a (not necessarily continuous) $\Gamma$-action over $X$, and
		\item[(ii)] the category of (not necessarily continuous) $\Gamma$-equivariant vector bundles on $X$ of rank $n$.
	\end{enumerate}
	Therefore, it suffices to show that the above equivalence preserves the continuity.

	Let $\Ecal$ be a (not necessarily continuous) $\Gamma$-equivariant vector bundle on $X$ of rank $n$, and let $\Vbb(\Ecal)$ be the geometric vector bundle on $X$ of rank $n$ with a (not necessarily continuous) $\Gamma$-action over $X$ corresponding to $\Ecal$.
	We prove that the action of $\Gamma$ on $\Ecal$ is continuous if and only if the action of $\Gamma$ on $\Vbb(\Ecal)$ is continuous.
	To prove this, we may assume that $X=\Spa (A,A^+)$ is an affinoid adic space and $\Ecal\cong A^n$.
	First, we assume that the action of $\Gamma$ on $A^n$ is continuous.
	We may assume that $(A^+)^n$ is stable under the action of $\Gamma$ by replacing $\Gamma$ with the stabilizer of $(A^+)^n$.
	Note that $(A^+)^n$ defines the affinoid open subspace $\Vbb(\Ecal^+)$ of $\Vbb(\Ecal)$.
	By the scaling argument, it suffices to show that the action of $\Gamma$ on $\Vbb(\Ecal^+)$ is continuous.
	Let $e_1,\ldots,e_n$ be the standard basis of $A^n$.
	For $i=1,\ldots,n$ and $\gamma\in \Gamma$, we set $\gamma e_i= a_{1i}(\gamma)e_1+\cdots+a_{ni}(\gamma)e_n.$
	Then for $i=1,\ldots,n$, the morphism
	\[\mu_i\colon \Gamma\to A^+\langle T_1,\ldots,T_n\rangle;\; \gamma \mapsto a_{1i}(\gamma)T_1+\cdots+a_{ni}(\gamma) T_n\]
	is continuous.
	Therefore, we get the element $\mu_i\in C(\Gamma,A^+\langle T_1,\ldots,T_n\rangle).$
	We regard $C(\Gamma,A\langle T_1,\ldots,T_n\rangle)$ as an $A$-algebra via
	\[A\to C(\Gamma,A\langle T_1,\ldots,T_n\rangle) ;\; a\mapsto (\gamma\mapsto \gamma a).\]
	Then we get a continuous morphism 
	\[A\langle T_1,\ldots,T_n\rangle\to C(\Gamma,A\langle T_1,\ldots,T_n\rangle) ;\;T_i\mapsto \mu_i,\]
	which proves that the action of $\Gamma$ on $\Vbb(\Ecal^+)$ is continuous.

	Next, we assume that the action of $\Gamma$ on $\Vbb(\Ecal)$ is continuous.
	In this case, by reversing the above argument, we can prove that the action of $\Gamma$ on $\Ecal$ is continuous.
\end{proof}

\begin{lemma}\label{lem:vector bundle vs GL_n-torsor}
	Let $n$ be a positive integer.
	Then the following categories are naturally equivalent:
	\begin{enumerate}
		\item The category of geometric $\Gamma$-equivariant $\GL_n$-torsors on $X$.
		\item The category of continuous $\Gamma$-equivariant vector bundles on $X$ of rank $n$.
	\end{enumerate}
\end{lemma}
\begin{proof}
	First, we have a natural categorical equivalence between
	\begin{enumerate}
		\item[{(i)}] the category of (not necessarily continuous) geometric $\Gamma$-equivariant $\GL_n$-torsors on $X$, and
		\item[{(ii)}] the category of (not necessarily continuous) $\Gamma$-equivariant vector bundles on $X$ of rank $n$.
	\end{enumerate}
	Therefore, it suffices to show that the above equivalence preserves the continuity.
	To prove this, we may assume that $X=\Spa (A,A^+)$ is an affinoid adic space.
	
	Let $P$ be a geometric $\Gamma$-equivariant $\GL_n$-torsor on $X$, and let $\Ecal_P$ be the (not necessarily continuous) $\Gamma$-equivariant vector bundle corresponding to $P$.
	We prove that the action of $\Gamma$ on $\Ecal_P$ is continuous.
	As in Lemma \ref{lem:P_H continuous}, we can show that the action of $\Gamma$ on the geometric vector bundle $\Vbb(\Ecal_P)=P\times^{\GL_n^{\an}} \Vbb(K^n)^{\an}$ is continuous, where $K^n$ is the standard representation of $\GL_n$.
	Therefore, by Lemma \ref{lem:geometric vector bundle vs vector bundle}, the action of $\Gamma$ on $\Ecal_P$ is also continuous.

	Next, let $\Ecal$ be a continuous $\Gamma$-equivariant vector bundle on $X$ of rank $n$, and let $P_{\Ecal}$ be the (not necessarily continuous) $\Gamma$-equivariant geometric $\GL_n$-torsor corresponding to $\Ecal$.
	We prove that the action of $\Gamma$ on $P_{\Ecal}$ is continuous.
	We have a $\Gamma$-equivariant open immersion 
	\[P_{\Ecal}\to P_{\Ecal} \times^{\GL_n^{\an}} M_n^{\an},\]
	where $\GL_n^{\an}$ acts on $M_n^{\an}$ by left multiplication.
	Since $P_{\Ecal} \times^{\GL_n^{\an}} M_n^{\an}$ is a geometric vector bundle associated with $\Ecal^{\oplus n}$, the action of $\Gamma$ on it is continuous by Lemma \ref{lem:geometric vector bundle vs vector bundle}.
	Therefore, the action of $\Gamma$ on $P_{\Ecal}$ is also continuous.
\end{proof}

\begin{definition}\label{defn:tannakian torsor}
	Let $G$ be a smooth linear algebraic group over $K$.
	A \textit{Tannakian $\Gamma$-equivariant $G$-torsor} is an exact $K$-linear $\otimes$-functor $\Rep(G)\to \Vect^{\Gamma}(X)$.
\end{definition}

\begin{theorem}\label{thm:geometric vs Tannakian}
	Let $G$ be a smooth linear algebraic group over $K$.
	Then the following categories are naturally equivalent:
	\begin{enumerate}
		\item The category of geometric $\Gamma$-equivariant $G$-torsors on $X$.
		\item The category of Tannakian $\Gamma$-equivariant $G$-torsors on $X$.
	\end{enumerate}
\end{theorem}
\begin{proof}
	Let $P$ be a (continuous) geometric $\Gamma$-equivariant $G$-torsor on $X$.
	Then, for an algebraic representation $E$ of $G$, as in Lemma \ref{lem:P_H continuous}, we can show that $P\times^{G^{\an}} \Vbb(E)^{\an}$ is a geometric vector bundle with a continuous $\Gamma$-action over $X$.
	Therefore, we obtain a continuous $\Gamma$-equivariant vector bundle on $X$ via Lemma \ref{lem:geometric vector bundle vs vector bundle}.
	This construction defines an exact $\otimes$-functor $T_P\colon \Rep(G)\to \Vect^{\Gamma}(X)$.

	On the other hand, let $T \colon \Rep(G)\to \Vect^{\Gamma}(X)$ be an exact $\otimes$-functor.
	Then, we obtain a (not necessarily continuous) geometric $\Gamma$-equivariant $G$-torsor $P_T$ on $X$ (see Definition \ref{def:torsors}).
	We need to show that the action of $\Gamma$ on $P_T$ is continuous.
	We choose a closed embedding $G\hookrightarrow \GL_n$ of algebraic groups.
	By Lemma \ref{lem:continuous immersion}, it suffices to show that the action of $\Gamma$ on $P_T\times^{G^{\an}}\GL_n^{\an}$ is continuous.
	This follows from Lemma \ref{lem:vector bundle vs GL_n-torsor}.
\end{proof}
 
%%%%%%%%%%%%%%%%%%%%%%%%%%%%%%%%%%%%%%%%%%%%%%%%%%%%%%%%%%%%%%%%%%%%%%%%%%%%%%%%
\section{Rigid analytic Artin stacks and quasi-analytic spaces}\label{section:rigid analytic Artin stack and quasi-analytic spaces}
In algebraic geometry, it is well-known that if an \'{e}tale sheaf is an Artin stack, then it is an algebraic space.
However, the naive rigid analytic analogue of the above statement is false.
In other words, there is an \'{e}tale sheaf on $\Aff_K^{\op}$ which is a rigid analytic Artin stack but not a quasi-analytic space.
\begin{example}\label{ex:rigid analytic Artin stack and quasi-analytic space}
	Let $C$ be an algebraically closed complete non-archimedean field.
	Let $\varpi\in C$ be a pseudo-uniformizer, and set $\Bbb^1=\Spa C\langle T\rangle$ and $\Bbb^1(\varpi)=\Spa C\langle T/\varpi\rangle$.
	Then the quotient \'{e}tale sheaf $\Bbb^1/\Bbb^1(\varpi)$ is a rigid analytic Artin stack.
	In fact, $\Bbb^1\to \Bbb^1/\Bbb^1(\varpi)$ is a smooth surjection, and the diagonal morphism 
	\[\Bbb^1/\Bbb^1(\varpi) \to \Bbb^1/\Bbb^1(\varpi)\times \Bbb^1/\Bbb^1(\varpi)\] 
	is an open immersion since its pullback along $\Bbb^1\times\Bbb^1\to \Bbb^1/\Bbb^1(\varpi)\times \Bbb^1/\Bbb^1(\varpi)$ is isomorphic to the rational open immersion $(\Bbb^1\times\Bbb^1)(\frac{T_1-T_2}{\varpi})\to \Bbb^1\times\Bbb^1$.
	We prove that $\Bbb^1/\Bbb^1(\varpi)$ is \textbf{not} a quasi-analytic space.
	Assume that $\Bbb^1/\Bbb^1(\varpi)$ is a quasi-analytic space.
	Then we can take a rigid space $X$ and an \'{e}tale surjection $X\twoheadrightarrow \Bbb^1/\Bbb^1(\varpi)$.
	Since $\Bbb^1/\Bbb^1(\varpi)$ is quasi-compact, we may assume that $X$ is quasi-compact.
	Let us consider the following Cartesian diagram
	\[
	\xymatrix{
		X \ar[r]& \Bbb^1/\Bbb^1(\varpi)\\
		X'\ar[r]\ar[u] & \Bbb^1.\ar[u]
	}
	\]
	Since the projection $\Bbb^1\to \Bbb^1/\Bbb^1(\varpi)$ is smooth of relative dimension $1$, the base change $X'\to X$ is as well.
	Since the morphism $X'\to \Bbb^1$ is \'{e}tale, the dimension of $X'$ is $1$.
	Therefore, the dimension of $X$ is $0$.
	Moreover, by \cite[Proposition 3.14]{OhShimizu25}, $X$ is smooth over $\Spa C$.
	Since $X$ is quasi-compact and $C$ is algebraically closed, we find that $X=\bigsqcup_{i=1}^n \Spa C$.
	The fiber product $R=X\times_{\Bbb^1/\Bbb^1(\varpi)} X$ defines an equivalence relation on $X=\bigsqcup_{i=1}^n \Spa C$.
	Since $R\subset X\times X=\bigsqcup_{1\leq i,j\leq n} \Spa C$, $R$ is also a disjoint union of finitely many copies of $\Spa C$.
	Therefore, the quotient $X/R\cong \Bbb^1/\Bbb^1(\varpi)$ is isomorphic to $\bigsqcup_{i=1}^m \Spa C$ for some $m$.
	Since $\Bbb^1\to \Bbb^1/\Bbb^1(\varpi)\cong \bigsqcup_{i=1}^m \Spa C$ is surjective and $\Bbb^1$ is connected, we find that $m=1$.
	Consequently, $\Bbb^1/\Bbb^1(\varpi)$ is isomorphic to $\Spa C$, which is obviously a contradiction.
\end{example}

\begin{remark}
	In \cite[Proposition 3.3.17]{War17}, it is claimed that every \'{e}tale sheaf on $\Aff_K^{\op}$ which is a rigid analytic Artin stack is a quasi-analytic space.
	However, there is a gap in the last part of the proof.
	The proof is based on the argument in \cite[Th\'{e}or\`{e}me 8.1]{Laumon-Bailly00}, and its final step relies on the fact that $\Abb^1_L\to \Abb^1_K$ is an open map of schemes for an extension $L/K$ of separably closed fields.
	However, the rigid analytic analogue of this fact is false.
	In other words, for a (non-trivial) extension $L/K$ of algebraically closed complete non-archimedean fields, the morphism $\Bbb^1_L\to \Bbb^1_K$ is not open.
	Consequently, the proof of \cite[Proposition 3.3.17]{War17} is incomplete.
\end{remark}

A key point of this counterexample is that the diagonal is an open immersion, and the author expects that if the diagonal is a closed immersion, such pathological phenomena do not occur.
Combining this observation with \cite{ConradTemkin09} (and \cite[Theorem 3.1.5]{War17}), the author proposes the following conjecture.
\begin{conjecture}\label{conjecture:separated rigid analytic Artin stack is representable}
Let $K$ be a complete non-archimedean field, and let $\Fcal$ be an \'{e}tale sheaf on $\Aff_K^{\op}$.
If the \'{e}tale sheaf $\Fcal$ is a separated rigid analytic Artin stack, then $\Fcal$ is representable by a separated rigid space.
\end{conjecture}

%%%%%%%%%%%%%%%%%%%%%%%%%%%%%%%%%%%%%%%%%%%%%%%%%%%%%%%%%%%%%%%%%%%%%%%%%%%%%%%%%%%%%%%%%%%%%%%%%%%%%


\begin{thebibliography}{BKKN67}

\bibitem[Avr77]{Avramov77}
Luchezar~L. Avramov, \emph{Homology of local flat extensions and complete
  intersection defects}, Math. Ann. \textbf{228} (1977), no.~1, 27--37.

\bibitem[BGR84]{BGR}
S.~Bosch, U.~G\"untzer, and R.~Remmert, \emph{Non-{A}rchimedean analysis},
  Grundlehren der mathematischen Wissenschaften [Fundamental Principles of
  Mathematical Sciences], vol. 261, Springer-Verlag, Berlin, 1984, A systematic
  approach to rigid analytic geometry.

\bibitem[BH22]{BhattHansen22}
Bhargav Bhatt and David Hansen, \emph{The six functors for
  {Z}ariski-constructible sheaves in rigid geometry}, Compos. Math.
  \textbf{158} (2022), no.~2, 437--482.

\bibitem[BIP23]{BIP23}
Gebhard B\"ockle, Ashwin Iyengar, and Vytautas Pa{\v{s}}k{\={u}}nas, \emph{On
  local {G}alois deformation rings}, Forum Math. Pi \textbf{11} (2023), Paper
  No. e30, 54.

\bibitem[BKKN67]{BKKN67}
R.~Berger, R.~Kiehl, E.~Kunz, and Hans-Joachim Nastold,
  \emph{Differentialrechnung in der analytischen {G}eometrie}, Lecture Notes in
  Mathematics, vol. No. 38, Springer-Verlag, Berlin-New York, 1967.

\bibitem[BLR95]{FRG3}
Siegfried Bosch, Werner L{\"u}tkebohmert, and Michel Raynaud, \emph{Formal and
  rigid geometry. {III}. {T}he relative maximum principle}, Math. Ann.
  \textbf{302} (1995), no.~1, 1--29.

\bibitem[CG]{Conrad-Gabber}
Brian Conrad and Ofer Gabber, \emph{Spreading-out for rigid-analytic families},
  in preparation.

\bibitem[Con99]{Conrad99}
Brian Conrad, \emph{Irreducible components of rigid spaces}, Ann. Inst. Fourier
  (Grenoble) \textbf{49} (1999), no.~2, 473--541.

\bibitem[CT09]{ConradTemkin09}
Brian Conrad and Michael Temkin, \emph{Non-{A}rchimedean analytification of
  algebraic spaces}, J. Algebraic Geom. \textbf{18} (2009), no.~4, 731--788.

\bibitem[EG23]{EG23}
Matthew Emerton and Toby Gee, \emph{Moduli stacks of \'etale ({$\varphi,
  \Gamma$})-modules and the existence of crystalline lifts}, Annals of
  Mathematics Studies, vol. 215, Princeton University Press, Princeton, NJ,
  [2023] \copyright 2023.

\bibitem[EGH25]{EGH23}
Matthew Emerton, Toby Gee, and Eugen Hellmann, \emph{An introduction to the
  categorical p-adic {L}anglands program}, The {L}anglands program.
  {II}---geometrization of the {L}anglands correspondence, Proc. Sympos. Pure
  Math., vol. 112.2, Amer. Math. Soc., Providence, RI, [2025] \copyright 2025,
  pp.~167--419.

\bibitem[Eis95]{Eisenbud95}
David Eisenbud, \emph{Commutative algebra}, Graduate Texts in Mathematics, vol.
  150, Springer-Verlag, New York, 1995, With a view toward algebraic geometry.

\bibitem[FF18]{FF18}
Laurent Fargues and Jean-Marc Fontaine, \emph{Courbes et fibr\'{e}s vectoriels
  en th\'{e}orie de {H}odge {$p$}-adique}, Ast\'{e}risque (2018), no.~406,
  xiii+382, With a preface by Pierre Colmez.

\bibitem[FS24]{FS24}
Laurent Fargues and Peter Scholze, \emph{Geometrization of the local
  {L}anglands correspondence}, arXiv preprint arXiv:
  \href{https://arxiv.org/abs/2102.13459}{2102.13459} (2024).

\bibitem[GR03]{GR03}
Ofer Gabber and Lorenzo Ramero, \emph{Almost ring theory}, Lecture Notes in
  Mathematics, vol. 1800, Springer-Verlag, Berlin, 2003.

\bibitem[Gro66]{EGA4-3}
A.~Grothendieck, \emph{\'el\'ements de g\'eom\'etrie alg\'ebrique. {IV}.
  \'etude locale des sch\'emas et des morphismes de sch\'emas. {III}}, Inst.
  Hautes \'Etudes Sci. Publ. Math. (1966), no.~28, 255.

\bibitem[Hal17]{Hall17Openness}
Jack Hall, \emph{Openness of versality via coherent functors}, J. Reine Angew.
  Math. \textbf{722} (2017), 137--182.

\bibitem[Har98]{Hartshorne90}
Robin Hartshorne, \emph{Coherent functors}, Adv. Math. \textbf{140} (1998),
  no.~1, 44--94.

\bibitem[Hub93]{Huber93}
R.~Huber, \emph{Continuous valuations}, Math. Z. \textbf{212} (1993), no.~3,
  455--477.

\bibitem[Hub94]{Huber94}
\bysame, \emph{A generalization of formal schemes and rigid analytic
  varieties}, Math. Z. \textbf{217} (1994), no.~4, 513--551.

\bibitem[Hub96]{Hub96}
Roland Huber, \emph{{\'E}tale cohomology of rigid analytic varieties and adic
  spaces}, Aspects of Mathematics, vol. E30, Friedr. Vieweg \& Sohn,
  Braunschweig, 1996.

\bibitem[Hum75]{Humphreys75}
James~E. Humphreys, \emph{Linear algebraic groups}, Graduate Texts in
  Mathematics, vol. No. 21, Springer-Verlag, New York-Heidelberg, 1975.

\bibitem[Ked08]{Ked08}
Kiran~S. Kedlaya, \emph{Slope filtrations for relative {F}robenius},
  Ast\'erisque (2008), no.~319, 259--301, Repr\'esentations $p$-adiques de
  groupes $p$-adiques. I. Repr\'esentations galoisiennes et
  $(\phi,\Gamma)$-modules.

\bibitem[Ked16]{Ked16}
\bysame, \emph{Noetherian properties of {F}argues-{F}ontaine curves}, Int.
  Math. Res. Not. IMRN (2016), no.~8, 2544--2567.

\bibitem[Kie68]{Kiehl68}
Reinhardt Kiehl, \emph{Analytische {F}amilien affinoider {A}lgebren}, S.-B.
  Heidelberger Akad. Wiss. Math.-Natur. Kl. \textbf{1968} (1968), 23--49.

\bibitem[KPX14]{KPX14}
Kiran~S. Kedlaya, Jonathan Pottharst, and Liang Xiao, \emph{Cohomology of
  arithmetic families of {$(\varphi,\Gamma)$}-modules}, J. Amer. Math. Soc.
  \textbf{27} (2014), no.~4, 1043--1115.

\bibitem[Li]{RayLi26}
Ray Li, \emph{Overconvergent deformation theory and an application to
  {$(\varphi,\Gamma)$}-modules}, In preparation.

\bibitem[Liu08]{Liu08}
Ruochuan Liu, \emph{Cohomology and duality for {$(\phi,\Gamma)$}-modules over
  the {R}obba ring}, Int. Math. Res. Not. IMRN (2008), no.~3, Art. ID rnm150,
  32.

\bibitem[LMB00]{Laumon-Bailly00}
G\'erard Laumon and Laurent Moret-Bailly, \emph{Champs alg\'ebriques},
  Ergebnisse der Mathematik und ihrer Grenzgebiete. 3. Folge. A Series of
  Modern Surveys in Mathematics [Results in Mathematics and Related Areas. 3rd
  Series. A Series of Modern Surveys in Mathematics], vol.~39, Springer-Verlag,
  Berlin, 2000.

\bibitem[LSW26]{LSW26}
Daniel Le, Matthias Strauch, and Zichuan Wang, \emph{Block decompositions in
  the p-adic {L}anglands correspondence}, arXiv preprint arXiv:
  \href{https://arxiv.org/abs/2610.07648}{2610.07648} (2026).

\bibitem[Mat80]{Mat80}
Hideyuki Matsumura, \emph{Commutative algebra}, second ed., Mathematics Lecture
  Note Series, vol.~56, Benjamin/Cummings Publishing Co., Inc., Reading, MA,
  1980.

\bibitem[Mat86]{Matsumura-Comm-Ring}
\bysame, \emph{Commutative ring theory}, Cambridge Studies in Advanced
  Mathematics, vol.~8, Cambridge University Press, Cambridge, 1986, Translated
  from the Japanese by M. Reid.

\bibitem[Mik26a]{Mikami25m}
Yutaro Mikami, \emph{The $p$-adic monodromy theorem over algebraic-affinoid
  algebras}, arXiv preprint arXiv:
  \href{https://arxiv.org/abs/2604.12280}{2604.12280} (2026).

\bibitem[Mik26b]{Mikami24}
Yutaro Mikami, \emph{{$(\varphi,\Gamma)$}-modules over relatively discrete
  algebras}, arXiv preprint arXiv:
  \href{https://arxiv.org/abs/2409.14145}{2409.14145} (2026).

\bibitem[OS25]{OhShimizu25}
Gyujin Oh and Koji Shimizu, \emph{Moduli stacks of crystals and isocrystals},
  arXiv preprint arXiv: \href{https://arxiv.org/abs/2504.14801}{2504.14801}
  (2025).

\bibitem[RJRC26]{RJRC23}
Joaqu\'{\i}n Rodrigues~Jacinto and Juan~Esteban Rodr\'{\i}guez~Camargo,
  \emph{Solid locally analytic representations}, arXiv preprint arXiv:
  \href{https://arxiv.org/abs/2305.03162}{2305.03162} (2026).

\bibitem[SGA73]{SGA4-3}
\emph{Th\'eorie des topos et cohomologie \'etale des sch\'emas. {T}ome 3},
  Lecture Notes in Mathematics, vol. Vol. 305, Springer-Verlag, Berlin-New
  York, 1973, S\'eminaire de G\'eom\'etrie Alg\'ebrique du Bois-Marie
  1963--1964 (SGA 4), Dirig\'e{} par M. Artin, A. Grothendieck et J. L.
  Verdier. Avec la collaboration de P. Deligne et B. Saint-Donat.

\bibitem[SW20]{SW20}
Peter Scholze and Jared Weinstein, \emph{Berkeley lectures on {$p$}-adic
  geometry}, Annals of Mathematics Studies, vol. 207, Princeton University
  Press, Princeton, NJ, 2020.

\bibitem[War17]{War17}
Evan~B Warner, \emph{Adic moduli spaces}, Ph.D. thesis, Stanford University,
  Department of Mathematics, August 2017,
  \url{https://purl.stanford.edu/vp223xm9859}.

\bibitem[SGA7]{SGA7}
\emph{Groupes de monodromie en g\'eom\'etrie alg\'ebrique. {I}}, Lecture Notes
  in Mathematics, vol. 288, Springer-Verlag, Berlin-New York, 1972, S\'eminaire
  de G\'eom\'etrie Alg\'ebrique du Bois-Marie 1967--1969 (SGA 7 I), Dirig\'e{}
  par A. Grothendieck. Avec la collaboration de M. Raynaud et D. S. Rim.

\bibitem[SP]{stacks-project}
The Stacks~Project Authors, \emph{Stacks project},
  \url{https://stacks.math.columbia.edu}.

\end{thebibliography}
\end{document}